\documentclass[12pt, reqno]{amsart}

\usepackage[margin=1.15in]{geometry}
\usepackage{amscd,amssymb, amsmath,amsfonts, wasysym, mathrsfs, enumitem, stmaryrd, mathtools,hhline,color,enumitem, tikz-cd}
\usepackage{graphicx}
\usepackage[all, cmtip]{xy}
\usepackage[T1]{fontenc}
\usepackage{newtxtext}
\usepackage{newtxmath}
\usepackage{textcomp}
\usepackage{url}
\usepackage{array}
\usepackage{faktor}
\usepackage{thmtools,bm}
\usepackage{enumitem}
\usepackage[pagebackref=true]{hyperref}%,colorlinks=true, linkcolor=violet,  citecolor=hot, urlcolor=magenta]{hyperref}%make all the references and links clickable
\usepackage{nameref}
\usepackage{cleveref}
\usepackage{comment}
\usetikzlibrary{arrows,positioning,calc,intersections}
\setlist[itemize]{wide = 0pt, labelwidth = 2em, labelsep*=0em, itemindent = 0pt, leftmargin = \dimexpr\labelwidth + \labelsep\relax, noitemsep,topsep = 1ex,}
\setlist[enumerate]{wide = 0pt, labelwidth = 2em, labelsep*=0em, itemindent = 0pt, leftmargin = \dimexpr\labelwidth + \labelsep\relax, noitemsep,topsep = 1ex}

\definecolor{hot}{RGB}{65,105,225}

\usetikzlibrary{shapes,arrows,positioning}
\setlist[itemize]{wide = 0pt, labelwidth = 2em, labelsep*=0em, itemindent = 0pt, leftmargin = \dimexpr\labelwidth + \labelsep\relax, noitemsep,topsep = 1ex,}
\setlist[enumerate]{wide = 0pt, labelwidth = 2em, labelsep*=0em, itemindent = 0pt, leftmargin = \dimexpr\labelwidth + \labelsep\relax, noitemsep,topsep = 1ex}

\theoremstyle{plain}
\newtheorem{thmx}{Theorem}
\renewcommand{\thethmx}{\Alph{thmx}} 
\newtheorem{thm}{Theorem}[section]  
\newtheorem{lem}[thm]{Lemma}
\newtheorem{claim}[thm]{Claim} 
\newtheorem{proposition}[thm]{Proposition}
 
\newtheorem{conjecture}[thm]{Conjecture}

\newtheorem{corx}[thmx]{Corollary} 
\theoremstyle{definition}
\newtheorem{dfn}[thm]{Definition}

\theoremstyle{remark}
\newtheorem{rem}[thm]{Remark}

\numberwithin{equation}{section}  

\theoremstyle{plain}
\newlist{thmlist}{enumerate}{1}
\setlist[thmlist]{wide = 0pt, labelwidth = 2em, labelsep*=0em, itemindent = 0pt, leftmargin = \dimexpr\labelwidth + \labelsep\relax, noitemsep,topsep = 1ex, font=\normalfont, label=(\roman*), ref=\thethm.(\roman{thmlisti})}

\addtotheorempostheadhook[thm]{\crefalias{thmlisti}{thm}}

\addtotheorempostheadhook[assumpsion]{\crefalias{thmlisti}{assumption}}

\addtotheorempostheadhook[cor]{\crefalias{thmlisti}{cor}}

\addtotheorempostheadhook[proposition]{\crefalias{thmlisti}{proposition}}

\addtotheorempostheadhook[dfn]{\crefalias{thmlisti}{dfn}}

\addtotheorempostheadhook[lem]{\crefalias{thmlisti}{lem}}
\addtotheorempostheadhook[main]{\crefalias{thmlisti}{main}}

\addtotheorempostheadhook[rem]{\crefalias{thmlisti}{rem}}

\newlist{thmenum}{enumerate}{1} % also creates a counter called 'propenumi'
\setlist[thmenum]{wide = 0pt, labelwidth = 2em, labelsep*=0em, itemindent = 0pt, leftmargin = \dimexpr\labelwidth + \labelsep\relax, noitemsep,topsep = 1ex, font=\normalfont, label=(\roman*), ref=\thethmx.(\roman{thmenumi})}%{label=\alph*), ref=\thethmx~(\alph*)}
\crefalias{thmenumi}{thmx} 

\newlist{corlist}{enumerate}{1} % also creates a counter called 'propenumi'
\setlist[corlist]{wide = 0pt, labelwidth = 2em, labelsep*=0em, itemindent = 0pt, leftmargin = \dimexpr\labelwidth + \labelsep\relax, noitemsep,topsep = 1ex, font=\normalfont, label=(\roman*), ref=\thecorx.(\roman{corlisti})}%{label=\alph*), ref=\thethmx~(\alph*)}
\crefalias{corlisti}{corx} 

\crefname{lem}{Lemma}{Lemmas} 
\crefname{conjecture}{Conjecture}{Conjectures}
\crefname{thm}{Theorem}{Theorems}
\crefname{proposition}{Proposition}{Propositions}
\crefname{dfn}{Definition}{Definitions}
\crefname{rem}{Remark}{Remarks}
\crefname{cor}{Corollary}{Corollaries}
\crefname{corx}{Corollary}{Corollaries}
\crefname{problem}{Problem}{Problems}
\crefname{thmx}{Theorem}{Theorems}
\crefname{claim}{Claim}{Claims}
\crefname{assumption}{Assumption}{Assumptions}
\crefname{main}{Main Theorem}{Main Theorems}

\def\ep{\varepsilon}

\def\rank{{\rm rank}\,}

\newcommand{\cS}{\mathcal{S}}

\makeatletter
\newcommand*{\rom}[1]{\expandafter\@slowromancap\romannumeral #1@}
\makeatother

\makeatletter     %Replace the Section ...  by the symbol \S ...  in Cref
\newcommand{\crefnames}[3]{%
	\@for\next:=#1\do{%
		\expandafter\crefname\expandafter{\next}{#2}{#3}%
	}%
}
\makeatother

\crefnames{part,chapter,section}{\S}{\S\S}

\newcommand{\sD}{\mathscr{D}}

\newcommand{\sH}{\mathscr{H}}

\newcommand{\sL}{\mathscr{L}}
\newcommand{\sM}{\mathscr{M}}

\newcommand{\sP}{\mathscr{P}}

\newcommand{\sS}{\mathscr{S}}

\newcommand{\sV}{\mathscr{V}}

\newcommand{\sX}{\mathscr{X}}

\newcommand{\cC}{\mathcal C}

\newcommand{\cL}{\mathcal L}

\newcommand{\cF}{\mathcal F}
\newcommand{\cM}{\mathcal M}
\newcommand{\cN}{\mathcal N}
\newcommand{\cU}{\mathcal U}
\newcommand{\cW}{\mathcal W}
\newcommand{\cO}{\mathcal O}
\newcommand{\cR}{\mathcal R}

\newcommand{\bB}{\mathbb{B}}
\newcommand{\bC}{\mathbb{C}}
\newcommand{\bD}{\mathbb{D}}

\newcommand{\bF}{\mathbb{F}}

\newcommand{\bN}{\mathbb{N}}

\newcommand{\bP}{\mathbb{P}}

\newcommand{\bR}{\mathbb{R}}

\newcommand{\bZ}{\mathbb{Z}}

\def\db{\bar{\partial}}

 \def\d{\partial} 
\def\hess{\sn\partial\bar\partial}

\def\sn{\sqrt{-1}}

\def\End{{\rm \small  End}}

\def\vol{\text{\small  Vol}}

\def\btau{{\bm{\tau}}}

\def\bsigma{{\bm{\sigma}}}

\newcommand{\Hom}{{\rm Hom}}

\newcommand{\GL}{{\rm GL}}

\newcommand{\SL}{{\rm SL}}

\begin{document} 
	\title[$\Gamma$-dimension in families]{Deformation Openness of Big Fundamental Groups\\
		 and Applications}

	\author[Y. Deng]{Ya Deng}
	\address{CNRS, Institut \'Elie Cartan de Lorraine, Universit\'e de Lorraine, Site de
		Nancy,   54506 Vand\oe uvre-lès-Nancy, France} 
	\email{ya.deng@math.cnrs.fr}
	\urladdr{https://ydeng.perso.math.cnrs.fr} 
	
	\author[C. Mese]{Chikako Mese}
	\address{Johns Hopkins University, Department of Mathematics, Baltimore, MD}
	\email[Chikako Mese]{cmese@math.jhu.edu} 
	\urladdr{https://sites.google.com/view/chikaswebpage/home}
	
	\author[B. Wang]{Botong Wang}
	\address{\parbox{\linewidth}{Department of Mathematics, University of Wisconsin-Madison, 480 Lincoln Drive, Madison WI 53706-1388\\
	School of Mathematics, Institute for Adanced Study, 1 Einstein Drive, Princeton, NJ 08540}}
	\email {wang@math.wisc.edu}
	\urladdr{https://people.math.wisc.edu/~bwang274/}  
	
	\begin{abstract}
In 2001, de Oliveira, Katzarkov, and Ramachandran conjectured that the property of smooth projective varieties having big fundamental groups is stable under small deformations. This conjecture was proven by Benoît Claudon in 2010 for surfaces and for threefolds under suitable assumptions. In this paper, we prove this conjecture for smooth projective varieties admitting a big complex local system. Moreover, we address a more general conjecture by Campana and Claudon concerning the deformation invariance of the \(\Gamma\)-dimension of projective varieties.

As an application, we establish the deformation openness of pseudo-Brody hyperbolicity for projective varieties endowed with a big and semisimple complex local system.

To achieve these results, we develop the deformation regularity of equivariant pluriharmonic maps into Euclidean buildings and  Riemannian symmetric spaces in families, along with techniques from the reductive and linear Shafarevich conjectures.

	\end{abstract}
	\maketitle
	\tableofcontents
	\section{Introduction} 
In \cite{Sha13}, Shafarevich conjectured that the universal cover of a smooth complex projective variety $X$ is holomorphically convex. If true, this conjecture implies the existence of a proper holomorphic fibration ${\rm sh}_X:X\to {\rm Sh}(X)$ (referred to as the Shafarevich morphism), which contracts to a point precisely those closed subvarieties $Z$ for which ${\rm Im}[\pi_1(Z)\to \pi_1(X)]$ is finite. In \cite{Kol93, Cam94}, Campana and Kollár independently established the existence of such a map, up to   birational equivalence   (so-called \emph{
	$\Gamma$-reduction} or \emph{Shafarevich map}, cf. \cref{thm:Campana} below). In \cite{Eys04,EKPR12}, Eyssidieux, Katzarkov, Pantev and Ramachandran proved the Shafarevich conjecture for smooth projective varieties with \emph{linear fundamental groups}.  	Recall that a finitely generated group  
$\Gamma$ is called \emph{linear} (resp. \emph{reductive}) 
if there exists an almost faithful (i.e. with finite kernel) linear (resp. \emph{reductive})  representation $ \Gamma\to \GL_N(\bC)$.  For more recent progress on the Shafarevich conjecture, see \cite{DY23}.

\medspace

  In \cite[\S 18.4]{Kol95}, Koll\'ar  asked the question of how  Shafarevich maps for complex projective varieties vary  in families.  
	 Precise conjectures were later proposed 	 by de Oliveira, Katzarkov, and Ramachandran \cite{dOKR02}, as well as by Campana  and Claudon \cite{Cla10}. This paper aims to explore these conjectures, offering some applications to the hyperbolicity of algebraic varieties.
	\subsection{Deformation of varieties with big fundamental groups}
	Let $X$ be a  compact K\"ahler manifold $X$.   We say $X$ has     \emph{big} (or \emph{generically large} in \cite{Kol95}) fundamental group if for any positive-dimensional irreducible subvariety $Z$ of $X$ passing through a very general point, the image ${\rm Im}[\pi_1(Z)\to \pi_1(X)]$ is an infinite group. Similarly, a representation $\varrho:\pi_1(X)\to \GL_N(\bC)$ is called \emph{big} if for any positive-dimensional irreducible subvariety $Z$ of $X$  passing through a very general point, the image $\varrho({\rm Im}[\pi_1(Z)\to \pi_1(X)])$ is an infinite group.  
	
As seen in the above definition, we emphasize that throughout this paper, when discussing properties of the fundamental group \(\pi_1(X)\) of a compact Kähler manifold \(X\), or representations \(\pi_1(X) \to \GL_N(\bC)\), we are not solely referring to the abstract group structure of \(\pi_1(X)\). Instead, we also take into account the structure inherently tied to the complex structure of \(X\) itself.

	In \cite[Conjecture 1.1]{dOKR02}, de Oliveira, Katzarkov, and Ramachandran proposed the following intriguing  conjecture. 
	\begin{conjecture}[\cite{dOKR02}]\label{conj:kollar}
		Let $f:\sX\to \bD$ be a smooth projective family of varieties over the unit disk $\bD$. For each $t\in \bD$, let $X_t:=f^{-1}(t)$ denote the fiber over $t$.  If $X_0$ has big fundamental group, then $X_t$ also has big  fundamental group for sufficiently small $t\in \bD$.
	\end{conjecture}  
\cref{conj:kollar} as stated here is  more general than in \cite[Conjecture 1.1]{dOKR02}, as we do not assume the existence of a linear representation of \(\pi_1(X_0)\) with infinite image. This conjecture was proved by Claudon for the case \(\dim X_0 = 2\) (cf. \cite[Corollaire 1.1]{Cla10}). Furthermore, he established \cref{conj:kollar} for \(\dim X_0 = 3\) (cf. \cite[Théorème 0.4]{Cla10}), contingent upon a conjecture regarding the virtual abelianity of the fundamental groups of special Kähler orbifold surfaces (cf. \cite[Conjecture 3]{Cla10}). 
However, beyond these works, there has been no progress on this conjecture in higher-dimensional cases.

In this paper, we prove the following theorem. 

\begin{thmx}[=\cref{mainA}] \label{thm:kollar}
	\cref{conj:kollar} holds if there exists a big representation \( \varrho:\pi_1(X_0) \to \GL_{N}(\bC) \). In particular, it holds if \(\pi_1(X_0)\) is linear.
\end{thmx}

%\cref{thm:kollar} is actually a special case of \cref{main2} below. However, since some readers may find it sufficient to focus on varieties with big fundamental groups, we present it separately here.
	\subsection{$\Gamma$-dimension under deformation}
In \cite{Cam94}, Campana introduced the notion of the \emph{$\Gamma$-dimension} (distinct from the Von Neumann dimension in \cite{Ati76}!), based on his construction of the \emph{$\Gamma$-reduction}. This construction was also obtained  independently by Kollár  for algebraic varieties \cite{Kol93} under the name \emph{Shafarevich map}. 
	\begin{thm}[\cite{Cam94,Kol93}]\label{thm:Campana}
		Let $X$ be a  compact  K\"ahler manifold and let $H\triangleleft \pi_1(X)$ be a normal subgroup. Then there exists proper  surjective, \emph{almost holomorphic } map $\gamma_{(X,H)}:X\dashrightarrow \Gamma_H(X)$ with connected fibers such that for any closed positive-dimensional subvariety $Z\subset X$ passing through a very general point of $X$, $\gamma_{(X,H)}(Z)$ is a point if and only if ${\rm Im}[\pi_1(Z)\to \pi_1(X)/H]$ is finite.  Such a map is unique up to bimeromorphic equivalence. \qed
	\end{thm}
Here, an almost holomorphic map refers to a meromorphic map whose indeterminacy locus does not dominate its image. When $H$ is a finite subgroup, one can verify that $\gamma_{(X,H)}$  is bimeromorphic to $\gamma_{(X,\{e\})}$, and in this case we  simply write $\gamma_{X}:X\dashrightarrow \Gamma(X)$ for $\gamma_{(X,H)}:X\dashrightarrow \Gamma_H(X)$, and call it the $\Gamma$-reduction of $X$.
	
	The following definition is introduced in \cite[D\'efinition 4.1]{Cam94}.
\begin{dfn}[$\Gamma$-reduction, $\Gamma$-dimension]
	 	Let $X$ be a compact K\"ahler manifold and let $H\triangleleft \pi_1(X)$ be a normal subgroup.  The  almost holomorphic map $\gamma_{(X,H)}:X\dashrightarrow \Gamma_H(X)$ in \cref{thm:Campana} is called the \emph{$\Gamma$-reduction} or \emph{Shafarevich map} of $(X,H)$. The  \emph{$\Gamma$-dimension} of $(X,H)$, denoted by $\gamma d(X,H)$,  is $\dim \Gamma_H(X)$. 
\end{dfn}
\begin{rem}
  The $\Gamma$-dimension, roughly speaking, measures the positivity of the fundamental group of a compact Kähler manifold. By \cref{thm:Campana}, it can be viewed as an analogue of the Kodaira dimension of canonical bundles for the fundamental groups of compact Kähler manifolds, with the $\Gamma$-reduction playing a role similar to the Iitaka fibration. In two extreme cases,  $X$ has big (resp. finite)  fundamental group if and only if $\gamma d(X)=\dim X$ (resp. $\gamma d(X)=0$). In \cite{Cla10}, varieties with big fundamental groups are also referred to as being of  
	 \emph{$\pi_1$-general type}, analogous to the concept of maximal Kodaira dimension. In \cite{Cam95}, Campana  established an upper bound for $\gamma d(X)$ in terms of an algebraic-geometric invariant $\kappa^+(X)$, which is  determined by the positivity properties of   $ \Omega^p_X$ for any $p>0$.
\end{rem}

As an analogue to the invariance of plurigenera proved by Siu \cite{Siu98,Siu02} (cf. also \cite{Pau07}), Campana and Claudon proposed the following conjecture, which is considerably more general than \cref{conj:kollar}.
\begin{conjecture}[\protecting{\cite[Conjecture 2]{Cla10}}]\label{conj:Campana}
	Let $f:\sX\to \bD$ be a smooth proper K\"ahler  family.  Then $t\mapsto \gamma d(X_t)$ is a continuous (hence constant) function on $\bD$. 
\end{conjecture}
Here we say $f:\sX\to \bD$ is a \emph{smooth proper K\"ahler  family} if  $\sX$ is a K\"ahler manifold and $f$ is  a  holomorphic  proper submersion with connected fibers. 

 In \cite{Cla10},  Claudon proved  \cref{conj:Campana} when $\dim X_0=2$, and in the case where  $\dim X_0=3$ and $\kappa (X_0)\leq 2$.
 
We prove \cref{conj:Campana} for K\"ahler varieties with rigid and integral fundamental groups. 
 \begin{thmx}[=\cref{thm:3}]\label{main3}
	Let $f:\sX\to \bD$ be a proper smooth K\"ahler family.  If there exists an almost  faithful reductive representation $\varrho:\pi_1(X_0)\to \GL_N(\bC)$, which is rigid and integral, then $t\mapsto \gamma d(X_t)$ is constant on $\bD$.
\end{thmx}
\begin{rem}
By a deep theorem of Esnault and Groechenig \cite{EG}, any \emph{cohomologically rigid} local system on a smooth projective variety is integral. Consequently, \cref{main3} also holds if $X_0$ is  a smooth projective variety and $\varrho:\pi_1(X_0)\to \GL_N(\bC)$ is assumed to be only almost faithful and cohomologically rigid.
\end{rem}

We also prove the lower semi-continuity for $\Gamma$-dimension of varieties with linear fundamental groups, which includes \cref{thm:kollar} as a special case.
\begin{thmx}[=\cref{thm:reductive,mainC}]\label{main2}
Let \( f: \mathscr{X} \to \mathbb{D} \) be a smooth projective family.
\begin{thmenum}
\item \label{main:Betti}  Let \( M_t := M_{\mathrm{B}}(\pi_1(X_t), \mathrm{GL}_N)(\mathbb{C}) \) denote the Betti moduli space of \(\pi_1(X_t)\). Define \( H_t \triangleleft \pi_1(X_t) \) as the intersection of the kernels of all reductive representations \( \varrho: \pi_1(X_t) \to \mathrm{GL}_N(\mathbb{C}) \). Then  the function	 \( t \mapsto \gamma d(X_t, H_t) \) on \( \mathbb{D} \)  is   \emph{lower semicontinuous}.
\item \label{main:linear} If \( \pi_1(X_0) \) is linear, then   \( t \mapsto \gamma d(X_t) \) is also a lower semicontinuous function on \( \mathbb{D} \). 
\end{thmenum}
\end{thmx} 
To prove \cref{thm:kollar,main3,main2}, we establish the deformation regularity of equivariant harmonic maps to Euclidean buildings or symmetric spaces  in families (see \cref{harmonic2,harmonic,lem:continuity,thm:family}). We then apply these results, together with techniques in \cite{Eys04,EKPR12,DY23} on the Shafarevich conjecture to prove  \cref{thm:kollar,main3,main2}.
\subsection{Applications to hyperbolicity}
We   give an application of \cref{main:Betti} to the hyperbolicity of algebraic varieties under deformation. First, we first recall the following definition of hyperbolicity (see   \cite{Dem20,CDY22} for related results and other notions of hyperbolicity).
	\begin{dfn}[Pseudo  Brody hyperbolicity]\label{def:Picard}
	A  compact complex manifold $X$ is called \emph{pseudo-Brody hyperbolic}   if there is a proper  Zariski closed  subset $Z\subsetneq X$ such that any  non-constant holomorphic map $f:\bC\to X$  has image in $Z$.  
		If $Z=\varnothing$, $X$ is simply called  \emph{Brody hyperbolic}.  
	\end{dfn}  Recall the following classical result on the openness of Brody hyperbolicity (see   e.g. \cite[Proposition 1.10]{Dem20}). 
	\begin{thm}\label{thm:Brody}
		Let $f:\sX\to \bD$ be a  holomorphic proper submersion from a complex manifold $\sX$ to  the unit disk with connected fibers.  If $X_0$ is  Brody hyperbolic, then there exists $\ep>0$ such that $X_t$ is Brody hyperbolic   for  $|t|<\ep$. 
	\end{thm}
There has long been a folklore conjecture that such openness properties hold for pseudo-Brody hyperbolicity.
		\begin{conjecture}\label{conj:hyperbolic}
		Let $f:\sX\to \bD$ be as in \cref{thm:Brody}. If $X_0$ is pseudo-Brody   hyperbolic, then  $X_t$ is also pseudo  Brody hyperbolic for   sufficiently small $t$. 
	\end{conjecture}
	This conjecture is indeed a consequence of the following generalized Green-Griffiths-Lang conjecture.
	\begin{conjecture}[Generalized GGL conjecture]\label{conj:GGL}
Let $X$ be a smooth projective variety. Then $X$ is pseudo-Brody hyperbolic if and only if  it is of general type. 
	\end{conjecture}
 Let us explain how to prove \cref{conj:hyperbolic} assuming \cref{conj:GGL}.  If $X_0$ is pseudo-Brody hyperbolic, then $X_0$ is of general type by \cref{conj:GGL}. By the invariance of plurigenera, $X_t$ is of general type for any $t\in \bD$. By  \cref{conj:GGL} again, $X_t$ is pseudo-Brody hyperbolic. This proves \cref{conj:hyperbolic}. 

\medspace

In  \cite[Theorem C]{CDY22},   we prove \cref{conj:GGL} if there exists a big and reductive representation $\pi_1(X)\to \GL_N(\bC)$.   We will apply \cref{main:Betti} together with \cite[Theorem C]{CDY22} and \cite{DY23} to prove the following result on \cref{conj:hyperbolic}. 
\begin{thmx}[=\cref{thm:open}]\label{main4}
	Let $f:\sX\to \bD$ be a smooth projective family. Assume that there is a big and reductive representation $\varrho:\pi_1(X_0)\to \GL_N(\bC)$. If $X_0$ is pseudo-Brody   hyperbolic, then there exists $\ep>0$ such that $X_t$ is also pseudo-Brody hyperbolic for sufficiently small $t$.
	\end{thmx} 
	As a consequence, we prove the following result.
	\begin{corx}\label{corx}
		Let $X$ be a smooth projective variety.  Assume that either 
		\begin{enumerate}[label=(\alph*)]
			\item \label{item:VHS} there is a complex variation of Hodge structure ($\bC$-VHS for short) $\varrho:\pi_1(X)\to \GL_N(\bC)$ with discrete monodromy $\Gamma$ such that the period map $X\to \sD/\Gamma$ is   generically finite onto the image, or
			\item\label{item:semisimple} there is a big representation $\tau:\pi_1(X)\to \GL_N(\bC)$ such that the Zariski closure of $\tau(\pi_1(X))$ is a semisimple algebraic group.
		\end{enumerate} Then every small projective deformation of $X$ is  pseudo-Brody hyperbolic.
	\end{corx}
	\begin{rem}
		The hyperbolicity results stated in this subsection also hold if we replace pseudo-Brody hyperbolicity with the stronger notion, known as pseudo-Picard hyperbolicity, introduced in \cite{Den23,CDY22}. To lighten the notion, we consider only pseudo-Brody hyperbolicity in this paper, and interested readers can refer to \cite{CDY22} for other notions of hyperbolicity.
	\end{rem}

	\subsection{Notation and Convention}\label{sec:notation}
\begin{enumerate}[label=(\alph*)]
	\item Throughout this paper, $\pi_1(X)$ always refers to   the topological fundamental group of the variety $X$. 
	\item For a complex space $X$, we denote by $X^{\rm reg}$ the regular locus of $X$. 
	\item Unless otherwise specified, all projective varieties are assumed to be defined over the field of complex numbers.
		\item Denote by $\bD$ the unit disk in $\bC$, and by $\bD^*$ the punctured unit disk. We write $\bD_\ep := \{ z \in \bC \mid |z| < \ep \}$, $\bD^*_\ep := \{ z \in \bC \mid 0 < |z| < r \}$.
	\item A \emph{smooth projective family} \(f:\sX \to \bD\) is a smooth projective morphism from \(\sX\) to \(\bD\) with connected fibers. A \emph{smooth proper Kähler family} \(f:\sX \to \bD\) is a holomorphic proper submersion from a Kähler manifold \(\sX\) to \(\bD\) with connected fibers. For any \(t \in \bD\), let \(X_t := f^{-1}(t)\) denote the fiber of \(f\) over \(t\). For notational simplicity, we often write \(X\) instead of \(X_0\). Denote by \(\pi_{\sX}:\widetilde{\sX} \to \sX\) the universal covering map.

		\item\label{def:rhot} Let \(f:\sX \to \bD\) be as above, and let \(\varrho:\pi_1(X_0) \to \GL_N(K)\) be a representation, where \(K\) is any field. We denote by \(\varrho_t:\pi_1(X_t) \to \GL_N(K)\) the representation induced by the composition of \(\varrho\) with the natural isomorphism \(\pi_1(X_t) \to \pi_1(X_0)\) induced by \(f\).
		
				\item Let \(\Omega \subset \bC^n\) be an open domain endowed with a Riemannian metric \(g\), and let \(z \in \Omega\) be any point. We denote by \(\bB_r(z)\) the Euclidean ball of radius \(r\) in \(\Omega\) centered at \(z\), and by \(B_r(z)\) the ball of radius \(r\) in \(\Omega\) centered at \(z\) with respect to the Riemannian metric \(g\).
				
		\item Let \(G\) be a reductive group defined over a non-archimedean local field \(K\). We denote by \(\Delta(G)_K\) the Bruhat-Tits building of \(G\), which is a non-positively curved (NPC) space. Denote by \(d(\bullet, \bullet)\) the distance function on \(\Delta(G)_K\). Let \(\sD G\) denote the derived group of \(G\), which is semisimple.
		
		\item A linear representation \(\varrho:\pi_1(X) \to \GL_N(K)\), where \(K\) is a field, is called \emph{reductive} if the Zariski closure of \(\varrho(\pi_1(X))\) is a reductive algebraic group over \(\overline{K}\). 
\end{enumerate} 
 \subsection*{Acknowledgment} We would like to thank Frédéric Campana, Benoît Claudon,  Philippe Eyssidieux,   Claire Voisin for helpful discussion.  The first author is supported in part by ANR-21-CE40-0010.  The second author is  supported in part by NSF DMS-2304697. 
	Part of this research was performed while the second author was visiting the Mathematical Sciences Research Institute (MSRI), now becoming the Simons Laufer Mathematical Sciences Institute (SLMath), which is supported by the National Science Foundation (Grant No. DMS-1928930). 
\section{Deformation of harmonic bundles} 
The aim of this section is to prove \cref{harmonic2}, which was used in the proofs of \cref{main3} and \cref{main2}. 
 \subsection{Some preliminary on Higgs bundle and harmonic bundle}\label{subsec:harmonic}
We first  recall the   definition of Higgs bundles and  harmonic bundles. We refer the readers to  \cite{Cor88,Sim88,Sim92} for further details.
 \begin{dfn}\label{Higgs}
 	Let $X$ be a complex manifold. A \emph{Higgs bundle} on $X$  consists of  $(E,\db_E,\theta)$ where $E$ is a holomorphic vector bundle with $\db_E$ its complex structure, and $\theta:E\to E\otimes \Omega^1_X$ is a holomorphic one form with value in $\End(E)$, say \emph{Higgs field},  satisfying $\theta\wedge\theta=0$. We will usually  write $(E,\theta)$ instead of $(E,\db_E,\theta)$  if no confusion arises. 
 	%	\begin{eqnarray}\label{higgs triple}
 		%	(\db +\theta)^2=0.
 		%	\end{eqnarray} 
 \end{dfn}

 Let  $(E,\theta)$ be a Higgs bundle  over a complex manifold $X$.   
 Suppose $h$ is a smooth hermitian metric of $E$.  Denote by $\nabla_h$  the Chern connection with respect to $h$, and by $\theta^*_h$  the adjoint of $\theta$ with respect to $h$.  The metric $h$ is  \emph{harmonic} if the operator $D_h:=\nabla_h+\theta+\theta_h^*$
 is integrable, that is, if  $D_h^2=0$. We write $\theta^*$ instead of $\theta_h^*$ to simply the notation if no confusion arises. 
 \begin{dfn}[Harmonic bundle] A harmonic bundle on a complex manifold $X$ is
 	a Higgs bundle $(E,\theta)$ endowed with a  harmonic metric $h$.
 \end{dfn}
 \begin{dfn}\label{def:caninical form}
 	 For a harmonic bundle $(E,\theta,h)$ on a complex manifold $X$,  its \emph{canonical form}  is a semi-positive closed $(1,1)$-form on $X$ defined by 
 	 $
 	 \sn{\rm tr}(\theta\wedge\theta^*).
 	 $ 
 \end{dfn} 
 	Let $(X,\omega)$ be a compact K\"ahler manifold.   Assume that $\varrho:\pi_1(X)\to \GL_N(\bC)$ is a reductive representation. Let $(V_\varrho,D)$ be the flat bundle on $X$ with the monodromy representation being $\varrho$. Let $h$ be any smooth hermitian metric for $V_\varrho$. Then there exists a unique decomposition $D=\nabla_h+\Phi_h$ such that $\nabla_h$ is a unitary connection of  $(V_\varrho,h)$ and $\Phi_h$ is a self-adjoint operator, namely, we have  $$\langle \Phi_h(e),e'\rangle_h=\langle e,\Phi_h(e')\rangle_h, \quad \forall \ e,e'\in V_\varrho.$$
 	Say $h$ is a \emph{harmonic metric} if $\nabla_h^*\Phi_h=0$. In this case, by the Siu-Sampson formula, we have the pluriharmonicity of $h$. In other words, if we decompose $\nabla_h=\nabla_h'+\nabla_h''$ and $\Phi_h=\Phi_h'+\Phi_h''$ into the $(1,0)$ and $(0,1)$-parts, then $(V_\varrho,\nabla_h'',\Phi'_h,h)$ is a harmonic bundle. By the theorem of Corlette \cite{Cor88}, such harmonic metric exists and is unique up to some obvious ambiguity.  We will also refer to $(V_\varrho,D,h)$ as a \emph{harmonic bundle} if no confusion arises.  In this case, the canonical form associated with the harmonic bundle $(V_\varrho,D_h'',\Phi'_h,h)$ introduced in \cref{def:caninical form} is uniquely defined, and we denote it by $\omega_\varrho:=\sn{\rm tr}(\Phi_h'\wedge \Phi_h'')$.  Note that it does not depend on the choice of the K\"ahler metric $\omega$ on $X$.
 	\begin{dfn}[Canonical form]\label{def:caninical form2}
 		We call the above semi-positive closed $(1,1)$-form $\omega_\varrho$ the \emph{canonical form}  associated with the reductive representation $\varrho:\pi_1(X)\to \GL_N(\bC)$. 
 	\end{dfn}
 	We leave it to the reader to verify as an exercise that for any $g \in \GL_N(\bC)$, we have $\omega_\varrho = \omega_{g \varrho g^{-1}}$, where $g \varrho g^{-1}$ is the conjugate of $\varrho$ by $g$.

 	\subsection{Deformation of canonical forms}
 	Let us  prove the main result of this section on the continuity of canonical forms in families. 
 	\begin{thm}\label{harmonic2} 
 		Let $f:\sX\to \bD$ be a proper smooth K\"ahler family. Let $\varrho:\pi_1(X)\to \GL_N(\bC)$ be a reductive representation. Then the fiberwise defined canonical form $\omega_{\varrho_t}$ on  $X_t$ associated with the representation $\varrho_t:\pi_1(X_t)\to \GL_N(\bC)$ varies continuously with respect to $t\in \bD$. 
 	\end{thm}
  The proof of \cref{harmonic2}  follows essentially from the same arguments in \cite{Sim92,Sim94b}, based on Uhlenbeck's weak compactness theorem \cite{Uhl82}.   We first recall some preliminary results.

     Let $g$ denote the Euclidean metric $\sum d z_i \cdot d \bar{z}_i$ of $\bD^n$. Let $\omega$ be a Kahler form on $\bD^n$ such that there exists a constant $C_0>0$ such that $C_0^{-1} \cdot \omega \leq g \leq C_0 \cdot \omega$. Let $\left(E,  \theta,h\right)$ be a harmonic bundle on $\bD^n$.   We have the expression $\theta=\sum_{i=1}^{n} f_i \cdot d z_i$ for holomorphic sections $f_i \in \operatorname{End}(E)$ on $\bD^n$. 	Recall the following result in \cite{Moc06}.
 	\begin{lem}[\protecting{\cite[Lemma 2.13]{Moc06}}]\label{lem:Moc}
 		  There exist a constant  $C_1$   depending  only on $ \operatorname{rank} E$, the upper bound of absolute values of eigenvalues of $f_1,\ldots,f_n$ and the above constant $C_0$,  such that 
 		 $$
 \left|\theta\right|_{\omega,h}^2(x) \leq C_1, \quad \forall x \in \bD_{\frac{1}{2}}^n.  
 		 $$  
 		 Here $ \left|\theta\right|_{\omega,h}$ denotes the norm of $\theta$ with respect the metric $h$ and $\omega$.  \qed
 	\end{lem} 
 	\begin{lem}\label{lem:Moc2}
 	There exists a constant $C_2$ depending only on   $  \operatorname{rank} E$ and $C_0$ such that absolute values of eigenvalues of $f_1(x),\ldots,f_n(x)$  for each $x\in \bD_{\frac{1}{2}}^n$  are bounded from above by $C_2\lVert \theta\rVert^2_{\omega,h}$, where
 	$$
 	\lVert \theta\rVert^2_{\omega,h}:=\int_{\bD^n} |\theta|_{\omega,h}^2d\vol_\omega.
 	$$
 	\end{lem}
 	\begin{proof}
Consider the characteristic polynomial $P_i:=\det (t-f_i)=t^r+\sigma_{i,1}t^{r-1}+\cdots+\sigma_{i,r}$, where $r=\rank E$ and  $\sigma_{i,j}$ are holomorphic functions on $\bD^n$.  For any $x\in \bD^n$, we denote by $\lambda_i(x)$ the maximum of the  absolute values of eigenvalues  of $f_i(x)$.  Then by
the classical inequalities between the norms of roots and coefficients of a polynomial, one has 
$$
\lambda_i(x)\leq 2\max_{k=1,\ldots,r} |\sigma_{i,k}(x)|^{\frac{1}{k}}.
$$ 
Note that there exists a constant $C_3$ depending only on $\rank E$ such that 
$$
s(x):=\max_{k=1,\ldots,r} |\sigma_{i,k}(x)|^{\frac{2}{k}}\leq C_3|f_i(x)|^2_h.
$$
For each $z\in \bD_{\frac{1}{2}}^n$, we set  $$\Omega_z:=\{w\in \bD^n\mid   |w_i-z_i|\leq \frac{1}{2} \mbox{ for each }i  \}.$$ 
Since $s(x)$ is a psh function, we have 
	\begin{align*}
	\log |s(z)|^2&\leqslant  \frac{4^n}{\pi^n }\int_{\Omega_z}  \log |s(w)|^2 d\mbox{vol}_{g}\\
	&\leqslant \log \big(\frac{4^n}{\pi^n }\cdot \int_{\Omega_z}   |s(w)|^2 d\mbox{vol}_{g}\big)\\  
	&\leq  \log  (\frac{4^n}{\pi^n })+ \log \int_{\bD^n}   |s(w)|^2 d\mbox{vol}_{g}\\
&\leq  \log  C_3(\frac{4^n}{\pi^n })+ \log \int_{\bD^n}   |f_i|_h^2 d\mbox{vol}_{g}.  
\end{align*}
On the other hand, note that there exists a constant $C_4$ depending only on $C_0$ such that
$$
\max_{i=1,\ldots,n}  \int_{\bD^n}   |f_i|_h^2 d\mbox{vol}_{g}\leq   C_4 \int_{\bD^n}   |\theta|_{\omega,h}^2 d\mbox{vol}_{\omega}=C_4\lVert \theta\rVert_{\omega,h}^2. 
$$
The above two inequalities yield the lemma. 
 	\end{proof}
\cref{lem:Moc,lem:Moc2} imply the following result. 
 	\begin{lem}\label{lem:uniform}
 		  There exists  a uniform constant $C$  depending only on   $ \rank E$,  the constant $C_0$ above, and the $L^2$-norm $\lVert \theta\rVert_{\omega,h}^2$ such that 
 	$$
 	\left|\theta\right|_{\omega,h}^2(x) \leq C, \quad \forall x \in \bD_{\frac{1}{2}}^n.
 	$$   \qed
 	\end{lem} 
 	\begin{proof}[Proof of \cref{harmonic2}]
 		It suffices to prove the continuity of $\omega_{\varrho_t}$ at $t=0$.   Write $X$ for $X_0$ for notational simplicity.  We fix a smooth trivialization $F:\sX\to X\times \bD$ and let $F_t:X_t\to X$ be the induced diffeomorphism. To prove the theorem, it  is equivalent to show that $F_{t,*}\omega_{\varrho_t}$ are  continuously varying forms on $X$ at $t=0$. 
 		
 		  Fix a K\"ahler metric $g$ on $X$ and let $g_t:= g|_{X_t}$ be the induced K\"ahler metric on $X_t$.  Let $(V_\varrho,D)$ be the flat vector bundle on $\sX$ whose monodromy representation is $\varrho$.  By Corlette's theorem, for each $t\in \bD$, there exists a harmonic metric $h_t$ for $(V_t,D_t):=(V_\varrho,D)|_{X_t}$.  Let $(E_t,\theta_t,h_t)$ be the corresponding harmonic bundle on $X_t$. By the same arguments as we will see in  \eqref{unienergybd}, the   energy  of the $\varrho_t$-equivariant harmonic map from $\widetilde{X}_t$ to $\GL_N/U_N$ induced $(E_t,\theta_t,h_t)$ is a continuous function on $\bD$. Hence,   there exists a constant $C_1>0$ such that the energy
 		  $$  \int_{X_t}|\theta_{t}+\theta_{t}^*|_{g_t,h_t}^2d\vol_{g_t}\leq C_1 
 		  $$
 		  for any $t\in \bD_{\frac{1}{2}}$.   This implies that \begin{align}\label{eq:uniform2}
 		  	 \lVert\theta_t\rVert_{g_t,h_t}^2\leq C_1 \mbox{ for any } t\in \bD_{\frac{1}{2}}.
 		  \end{align}
  	We can find a constant $\ep\in (0,\frac{1}{2})$ and  admisisble coordinate systems    $\{\Omega_1,\ldots,\Omega_m\}$ centered at points in $X_0$ such that the following properties hold:
 		  \begin{enumerate}[label=(\alph*)]
 		  	\item  there are biholomorphisms $\varphi_i: \bD^n\times \bD_\ep\to \Omega_i $ with $f\circ \varphi_i(z,t)=t$.
 		  	\item Set $\Omega_i(\frac{1}{2}):=\varphi_i(\bD_{\frac{1}{2}}^n\times \bD_\ep)$. Then $\cup_{i=1}^{m}\Omega_i(\frac{1}{2})=f^{-1}(\bD_{\ep})$.
 		  	\item There exists a constant $C_0>0$ such that for the Euclidean metric $g=\sum d z_i \cdot d \bar{z}_i$ of $\bD^n$, 
 		  	for each $t\in \bD_{\ep}$,   we have
 		  	$$
 		  	C_0^{-1}\varphi_i^*g_t\leq g\leq C_0\varphi_i^*g_t.
 		  	$$
 		  \end{enumerate}
 		  By \cref{lem:uniform} and \eqref{eq:uniform2}, this implies that there exists a uniform constant $C>0$ such that  for any $t\in \bD_{\frac{1}{2}}$ and any $x\in X_t$, we have
\begin{align}\label{eq:uniform}
	 	  |\theta_t(x)|^2_{g_t,h_t}\leq C.
\end{align} 
 	Let $C_2$ be a constant such that $C_2^{-1} F_t^*g_0\leq g_t\leq C_2  F_t^*g_0$ for each $t\in \bD_{\ep}$. 	  Then we have the uniform bound for the Chern curvature of $(E_t,h_t)$ with respect to the metric $F_t^*g_0$ as follows: 
\begin{align}\label{eq:Uhlenbeck}
	 	|  R(E_t,h_t)(x)|_{F_t^*g_0,h_t}\leq 	C_2|  R(E_t,h_t)(x)|_{g_t,h_t}=C_2|[\theta_t,\theta_t^*](x)|_{g_t,h_t}\leq 2C^2C_2.
\end{align}  		
Pick $\{t_i\}_{i\in \bN}$ be  an arbitrary sequence of points in $\bD_{\ep}$ that converge  to $0$. To simplify notation,  let $F_i: X_i \to X$ be $F_{t_i}: X_{t_i} \to X$ and $(E_i, \theta_i, h_i)$ be $(E_{t_i}, \theta_{t_i}, h_{t_i})$. Using \eqref{eq:Uhlenbeck}, we can apply Uhlenbeck's weak compactness theorem \cite{Uhl82}, following the approach outlined in \cite[\S 7]{Sim94b}.

 Fix some large positive integer $p$. 
 		By    \cite[Proposition 7.9]{Sim94b}, after subtracting a subsequence,  there exists a harmonic bundle  $(E,\theta,h)$ on $X$   together with  a sequence of $L_2^p$-isometries (i.e. gauge transform in \cite{Uhl82}) $\eta_i:F_{i,*}(E_i,h_i)\stackrel{\simeq }{\to }(E,h)$ such that 
 		$
 		\eta_{i,*}F_{i,*}(\theta_i)-\theta
 		$, 	$
 		\eta_{i,*}F_{i,*}(\theta_i^*)-\theta^*
 		$, and  $
 		\eta_{i,*}F_{i,*}(D_i)-D
 		$ converges to zero strongly in $C^0$.   Here $D_i$ (resp. $D$) is flat connection associated with the harmonic bundle $(E_i,\theta_i,h_i)$ (resp. $(E,\theta,h)$). %Let us prove that $(E,\theta,h)$ is isomorphic to $(E_0,\theta_0,h_0)$. By Corlette's theorem on the unicity of harmonic bundles \cite{Cor88}, it suffices to prove their corresponding flat bundles are isomorphic.  
 		
 	Note that the monodromy representation  $\tau_i:\pi_1(X)\to \GL_N(\bC)$ of the flat bundle $$(\eta_{i,*}F_{i,*}(E_i),\eta_{i,*}F_{i,*}(D_i))$$ is conjugate to that of $D_i$.  Hence there exists an element $g_i\in \GL_N(\bC)$ such that $\tau_i=g_i\varrho g^{-1}_i$. Hence we have $[\tau_i]=[\varrho]$ in $M_{\rm B}(\pi_1(X),\GL_N)(\bC)$.   Let $\tau$ be the the monodromy representation of $D$.  Since $
 	\eta_{i,*}F_{i,*}(D_i)-D
 	$ converges to zero strongly in $C^0$, it follows that $\lim\limits_{i\to\infty}[\tau_i]=[\tau]$.
 	Hence $\tau$ is conjugate to $\varrho$.   One thus has
\begin{align}\label{eq:C1}
 	\omega_\varrho=\sn{\rm tr}(\theta_0\wedge\theta_0^*)=			\sn{\rm tr}(\theta\wedge\theta^*)=\omega_\tau.
\end{align}  Note that
\begin{align}\label{eq:C0}\nonumber
	 F_{i,*}	(\omega_{\varrho_{t_i}})&=F_{i,*}\left(	\sn{\rm tr}(\theta_i\wedge\theta_i^*) \right)
	 =	\sn{\rm tr}\left(F_{i,*}(\theta_i)\wedge   F_{i,*}(\theta_i^*)\right)\\
	 &=\sn{\rm tr}\left(\eta_{i,*}F_{i,*} (\theta_i)\wedge  	\eta_{i,*}F_{i,*}(\theta_i^*)\right).
\end{align} 
 	Since 
 $
 \eta_{i,*}F_{i,*}(\theta_i)-\theta
 $ 	and $
 \eta_{i,*}F_{i,*}(\theta_i^*)-\theta^*
 $ 
  converges to zero strongly in $C^0$, 
 $$
  \sn{\rm tr}\left(\eta_{i,*}F_{i,*}(\theta_i)\wedge  	\eta_{i,*}F_{i,*}(\theta_i^*)\right) 
 $$ converges to $
 	\sn{\rm tr}(\theta\wedge\theta^*) 
 $  strongly in $C^0$.  \eqref{eq:C1} and \eqref{eq:C0} imply that $F_{i,*}(\omega_{\varrho_{t_i}})$ converges to $\omega_\varrho$ strongly in $C^0$.   As $\{t_i\}$ is any sequence in $\bD_{\frac{1}{2}}$ converging to $0$, this implies the continuity of $\omega_{\varrho_t}$  at $X_0$. The theorem is proved.  		 
 	\end{proof}
 	\begin{rem}
Later, in \cref{thm:family}, we will establish a regularity result for equivariant harmonic maps into symmetric spaces in families. This result enables us to strengthen \cref{harmonic2}, showing that \(\omega_{\varrho_t}\) varies \emph{smoothly} with respect to \(t\), although this is not required for the proofs of \cref{main3,main2}. See \cref{rem:smooth}. 
 	\end{rem}
 	 \section{ $\Gamma$-dimension and rigid and integral fundamental groups}
 	 In this section we will prove \cref{main3}, utilizing \cref{harmonic2}.
 	\subsection{Some preliminary for $\bC$-VHS and Griffiths bundle}
We recall some notions of complex variation of Hodge structures ($\bC$-VHS for short). We refer the readers to \cite{Gri70,Sch73,Sim88,CMP,SS22} for more details. 

 	Let $X$ be a compact K\"ahler manifold and let $\tau:\pi_1(X)\to \GL_N(\bC)$ be a reductive representation underlying a $\bC$-VHS of weight $w$.  By \cite{Sim88}, it corresponds to  a \emph{system of Hodge bundles} $(E=\oplus_{p+q=w}E^{p,q},\theta)$ endowed with a Hodge metric $h$ induced by $\tau$.  Precisely, we have
 	\begin{enumerate}
 		\item  $(E,\theta,h)$ is a harmonic bundle on $X$.
 		\item  The decomposition $E=\oplus_{p+q=w}E^{p,q}$ is orthogonal with respect to $h$. 
 		\item We have
 		$$
 		\theta:E^{p,q}\to E^{p-1,q+1}\otimes \Omega_X
 		$$
 	%	\item There exists a real vector bundle $E_\bR$ such that $E=E_\bR\otimes_\bR\bC$, and $E^{p,q}=\overline{E^{q,p}}$. 
 		\item The connection $\nabla_h+\theta+\theta^*$ is flat, whose monodromy representation is $\tau$. 
 	\end{enumerate} 
 	In Griffiths \cite{Gri70} (cf. also \cite[Problem 13.3.3]{CMP}), Griffiths introduced a line bundle  	on $X$  (so-called \emph{Griffiths line bundle}) defined by
 	$$
 	L_\tau:= (\det E^{w,0})^{w}\otimes (\det E^{w-1,1})^{w-1}\otimes \cdots\otimes (\det E^{1,w-1}) 
 	$$
endowed with the metric $g$ induced by $h$. 
 	
 	We denote by $\theta_{p,q}:=\theta|_{E^{p,q}}$, and $h_{p,q}:=h|_{E^{p,q}}$. Then the Chern curvature of $(E^{p,q},h_{p,q})$ is given by
 	$$
 	R_{p,q}:=-\theta_{p,q}^*\wedge \theta_{p,q}- \theta_{p+1,q-1} \wedge \theta_{p+1,q-1}^*.
 	$$
 	Therefore, the Chern curvature of $(L_\tau,g)$ is
 	\begin{align*}
 		R(L_\tau,g)&=w{\rm tr}(-\theta_{w,0}^*\wedge \theta_{w,0})+\sum_{p=1}^{w-1}(w-q){\rm tr}(-\theta_{w-q,q}^*\wedge \theta_{w-q,q}- \theta_{w-q+1,q-1} \wedge \theta_{w-q+1,q-1}^*)\\
 		&=\sum_{q=0}^{w-1}{\rm tr}(-\theta_{w-q,q}^*\wedge \theta_{w-q,q})\\
 		&=\sum_{q=0}^{w-1}{\rm tr}( \theta_{w-q,q}\wedge  \theta_{w-q,q}^*)\\
 		& ={\rm tr}(\theta\wedge\theta^*).
 	\end{align*} 
 This implies the following result, which is already well-known to experts:
 	\begin{lem}\label{lem:Griffiths}
 		We have the integrality of the cohomology class
 		$$
 	\{\frac{1}{2\pi}\omega_\tau\}=	\{\frac{\sn}{2\pi}{\rm tr}(\theta\wedge\theta^*)\}=c_1(L_\tau)\in H^{2}(X,\bZ).
 		$$ 
 	\end{lem}
 	Let $\sD$ be the period domain of the $\bC$-VHS   $\tau$ and let $\Gamma$ be the monodromy group $\tau(\pi_1(X))$.  Assume that $\Gamma$ acts discretely on $\sD$. Then the quotient $\sD/\Gamma$ is a complex space and the period map is a holomorphic map $p:X\to \sD/\Gamma$, whose differential is $\theta$. 
 	\begin{lem}\label{lem:gamma}
 	Let $\tau:\pi_1(X)\to \GL_N(\bC)$ be a $\bC$-VHS with discrete monodromy. 	Let $\omega$ be a K\"ahler form on a compact K\"ahler $n$-fold $X$. Let $k_0$ be the largest integer $k$ such that
 		$$
 		\int_X  (\sn {\rm tr}(\theta\wedge\theta^*))^k\wedge \omega^{n-k} >0.
 		$$ 
  Then $k_0$ is the dimension of the image  $p(X)$.
 	\end{lem}
 	\begin{proof}
 		Since $\theta$ is the differential of $p$,  it follows that  $\theta|_Z$ is  trivial for any fiber $Z$ of $p$. Hence, $\sn {\rm tr}(\theta\wedge\theta^*)|_Z$ is also trivial.  On the other hand, for any $\xi\in T_{X}$ such that $dp(\xi)\neq 0$, we have $ {\rm tr}(\theta\wedge\theta^*)(\xi,\bar{\xi})=|\theta(\xi)|_h>0$. Let $X^\circ$ be a Zariski dense open subset of $p$ such that $p|_{X^\circ}$ is a proper submersion.   This implies that, at  each point $x\in X^\circ$, the semi-positive $(1,1)$-form $\sn {\rm tr}(\theta\wedge\theta^*)$ has exactly $\dim p(X)$ strictly positive eigenvalues. Therefore,   $  (\sn {\rm tr}(\theta\wedge\theta^*))^{k}\wedge \omega^{n-k} $ is a positive measure  on   $X^\circ$ for $k=\dim p(X)$, and is trivial for $k>\dim p(X)$ on $X^\circ$.  The lemma is proved. 
 	\end{proof}
 	
 	\subsection{Notion of Shafarevich morphism}
Let us give the definition of the Shafarevich morphism (cf. \cite{Eys04,Eys11,DY23}).  
\begin{dfn}[Shafarevich morphism]\label{def:Shafarevich}  
	Let \( X \) be a compact Kähler manifold, and let \( H \triangleleft \pi_1(X) \) be a normal subgroup. A proper holomorphic fibration  
$
	{\rm sh}_H: X \to {\rm Sh}_H(X)
$
	onto a complex normal variety \( {\rm Sh}_H(X) \) is called the \emph{Shafarevich morphism} associated with \( (X, H) \) if, for any closed subvariety \( Z \) of \( X \), \( {\rm sh}_H(Z) \) is a point if and only if the image  
	\[
	{\rm Im}[\pi_1(Z) \to \pi_1(X)/H]
	\]  
	is finite. We shall write \( {\rm sh}_X: X \to {\rm Sh}(X) \) for the Shafarevich morphism associated with \( (X, \{e\}) \), and refer to it as the Shafarevich morphism of \( X \).  
	
	Let \( M \) be a subset of the \emph{Betti moduli space} \( M_{\rm B}(\pi_1(X), \GL_N)(\bC) \) (cf. \cite{Sim94} for the definition). Let \( H = \bigcap_{\tau} \ker \tau \), where \( \tau \) ranges over all reductive representations \( \tau: \pi_1(X) \to \GL_N(\bC) \) with \( [\tau] \in M \). If the Shafarevich morphism \( {\rm sh}_H: X \to {\rm Sh}_H(X) \) exists, then we shall write  
	\[
	{\rm sh}_M: X \to {\rm Sh}_M(X)
	\]  
	in place of \( {\rm sh}_H: X \to {\rm Sh}_H(X) \), and call it the Shafarevich morphism associated with \( M \).  
\end{dfn}  

%The \emph{Shafarevich conjecture} asserts that the universal cover of a smooth projective variety is holomorphically convex. Consequently, the Shafarevich morphism of a smooth projective variety \( X \) exists if the conjecture holds for \( X \).  

In \cite{Eys04} (cf. \cite[Proposition 3.37]{DY23} for the singular case), it is shown that when \( X \) is projective and \( M := M_{\rm B}(\pi_1(X), \GL_N)(\bC) \), the Shafarevich morphism associated with \( M \) exists and is algebraic.

\begin{lem}\label{lem:Shafarevich2}
Let $X$ be a smooth projective variety. Set $M:=M_{\rm B}(\pi_1(X),\GL_N)(\bC)$. 	If there exists an almost faithful reductive representation $\varrho:\pi_1(X)\to \GL_N(\bC)$,  then the {Shafarevich morphism}  ${\rm sh}_M:X\to {\rm Sh}_M(X)$ associated with  $M$ is the Shafarevich morphism of $X$.  
\end{lem}
\begin{proof}
	Since $\varrho:\pi_1(X)\to \GL_N(\bC)$ is almost faithful, then for any closed subvariety $Z$ of $X$,  $\varrho({\rm Im}[\pi_1(Z)\to \pi_1(X)])$ is finite if and only if ${\rm Im}[\pi_1(Z)\to \pi_1(X)]$ is finite.  The lemma follows from \cref{def:Shafarevich}. 
		\end{proof}
		\begin{rem}\label{rem:dim}
	Let $H \triangleleft \pi_1(X)$ be a normal subgroup. 		Note that the Shafarevich morphism ${\rm sh}_H:X\to {\rm Sh}_H(X)$ associated with $(X,H)$, if exists, is bimemorphic to the $\Gamma$-reduction of $(X,H)$. Hence, we have $\dim {\rm Sh}_H(X)=\gamma d(X,H)$. 
		\end{rem}

The next result is well-known, and we provide a proof for the sake of completeness.
\begin{lem}\label{lem:Shafarevich}
Let $\tau$ be a $\bC$-VHS on a compact K\"ahler manifold  $X$ with discrete monodromy $\Gamma$. 	The Stein factorization $g:X\to Y$ of its period map $p:X\to \sD/\Gamma$ is the Shafarevich morphism associated with $\tau$.
\end{lem}	\begin{proof}
 	Let $Z$ be a connected component of any fiber of $g$. Recall that (cf. \cite{CMP}), there exists a real semisimple algebraic group $G \subset \GL_N$ such that $G(\mathbb{R})$ acts transitively on $\sD$. For any point $P \in \sD$, the stabilizer of $P$ under the action of $G(\mathbb{R})$ is a compact subgroup $K$ of $G(\mathbb{R})$. Moreover, we have $\tau(\pi_1(X)) \subset G(\mathbb{R})$. Then $\tau(\text{Im}[\pi_1(Z) \to \pi_1(X)])$ is contained in some conjugate $K'$ of $K$. Since $\Gamma$ is discrete, $K' \cap \Gamma$ is finite. Thus, $\tau(\text{Im}[\pi_1(Z) \to \pi_1(X)]) \subset K' \cap \Gamma$ is finite.
 	
 	Now, let $Z$ be a closed subvariety of $X$ such that $\tau(\text{Im}[\pi_1(Z) \to \pi_1(X)])$ is finite. Then, for some finite étale cover $Z'$ of $Z$, $\tau(\text{Im}[\pi_1(Z') \to \pi_1(X)])$ is trivial. By the rigidity of VHS (cf. \cite[Proposition 7.24]{Sch73}), the pullback of the $\mathbb{C}$-VHS underlying $\tau$ on the desingularization of $Z'$ is thus trivial. Hence, the composite map $Z' \to Z \stackrel{p}{\to} \sD / \Gamma$ is trivial. This completes the proof of the lemma.
 	\end{proof}
 	
 	\subsection{Proof of \cref{main3}}
 	In this subsection, we will prove \cref{main3}.
 	\begin{thm}[=\cref{main3}]\label{thm:3}
 		Let $f:\sX\to \bD$ be a proper smooth K\"ahler family.  If there exists a faithful reductive representation $\varrho:\pi_1(X_0)\to \GL_N(\bC)$, which is rigid and integral, then $t\mapsto \gamma d(X_t)$ is constant on $\bD$.
 	\end{thm}
 	\begin{proof}
 		By the integrality and rigidity of $\varrho$, after replacing $\varrho$  by some conjugate, there exists a number field $L$ such that  we have 
 		$$
 		\varrho:\pi_1(X_0)\to \GL_N(O_L).
 		$$   Hence we have
 		$$
 		\varrho_t:\pi_1(X_t)\to \GL_N(O_L) 
 		$$
 		for each $t\in \bD$ by the definition of $\varrho_t$.

 	Let ${\rm Ar}(L)$ be the set of archimedean places of $L$,   consisting of all embedding $\sigma:L\to \bC$.  Then the direct sum
 	$$
 	\oplus_{\sigma\in {\rm Ar}(L)}\sigma\varrho_t:\pi_1(X_t)\to \oplus\GL_N(\bC)
 	$$
 	has discrete image for each $t\in \bD$. 	Let $\tau_t:\pi_1(X_t)\to \GL_{M}(\bC)$   be $ 	\oplus_{\sigma\in {\rm Ar}(L)}\sigma\varrho_t$.  
 		
Since the rigidity of \( \varrho \) implies the rigidity of \( \varrho_t \), as we have the natural isomorphism between \( M_{\rm B}(\pi_1(X_t), \GL_N) \) and \( M_{\rm B}(\pi_1(X_0), \GL_N) \), by \cite{Sim92}, for any embedding \( \sigma \in \text{Ar}(L) \), \( \sigma \varrho_t: \pi_1(X_t) \to \GL_N(\mathbb{C}) \) corresponds to a \( \mathbb{C} \)-VHS on \( X_t \). Hence, \( \tau_t \) underlies a \( \mathbb{C} \)-VHS with discrete monodromy. We stress here that the period domains of these \( \mathbb{C} \)-VHS \( \tau_t \) might vary in \( t \).

 		Let $\omega_{\tau_t}$ be the canonical form on $X_t$ associated with $\tau_t$ defined in \cref{def:caninical form2}.  By \cref{harmonic2}, the cohomology classes $\{\omega_{\tau_t}\}_{t\in \bD}$ vary continuously in $H^2(\sX,\bR)\simeq H^2(X_t,\bR)$ with respect to $t$.   
 		
 		On the other hand,  by \cref{lem:Griffiths}, $\frac{1}{2\pi} \omega_{\tau_t}\in  H^{2}(X_t,\bZ)$ for each $t$. Therefore, the section $\left(t\mapsto \{\omega_{\tau_t}\}\right)\in \Gamma(\bD,R^2f_*(\bR))$  is flat with respect to the Gauss-Manin connection. 
 		
 		Let $\omega$ be a K\"ahler form on $\sX$ which gives the polarization. Denote by $\omega_t=\omega|_{X_t}$. Then the smooth section   $\left( t\mapsto \{\omega_t\}\right)\in \Gamma(\bD,R^2f_*(\bR))$ is also flat with respect to the Gauss-Manin connection of the local system $R^2f_*(\bR)$. It follows that  for each $k\in \bN$, 
 		\begin{align*}
 			h_k:\bD&\to \bN\\
 			t&\mapsto \int_{X_t} ( \omega_{\tau_t})^k\wedge (\omega_t)^{n-k}
 		\end{align*} 
 		is a constant function on $\bD$, where $n$ denotes the relative dimension of $f$.  Set
 		$$
 		k_0=\max_k\{h_k\neq 0\}.
 		$$
 	Since $\tau_t$ has discrete image,   its image $\Gamma_t:=\tau_t(\pi_1(X_t))$ acts discretely on $\sD_t$, and thus $\sD_t/\Gamma_t$ is a complex space. Denote by $p_t:X_t\to \sD_t/\Gamma_t$    the corresponding period map of $\tau_t$. By \cref{lem:gamma}, $k_0$ is the dimension of the image   $p_t(X_t)$. By \cref{lem:Shafarevich}, we have $ \dim {\rm Sh}_{\tau_t}(X_t)=k_0$ for each $t\in \bD$, where ${\rm sh}_{\tau_t}:X_t\to  {\rm Sh}_{\tau_t}(X_t)$ is the Shafarevich morphism of $\tau_t$.   Since $\varrho_t$ is assumed to be almost faithful, it follows that $\ker\tau_t$ is a finite subgroup of $\pi_1(X_t)$. This implies that ${\rm sh}_{\tau_t}$  is the Shafarevich morphism of $X_t$.  Hence $\gamma d(X_t)=k_0$ for any $t\in \bD$.  The theorem is  proved. 
 	\end{proof}

 			\section{Deformation of  harmonic maps   into Euclidean buildings} 
 	To prove \cref{thm:kollar,main2}, we need to apply techniques used in studying the reductive Shafarevich conjecture (see \cite{Eys04,DY23}). This involves considering local systems over non-archimedean local fields on projective varieties. In this context, we apply tools developed by Gromov and Schoen \cite{GS92} regarding the existence of equivariant harmonic maps to Bruhat-Tits buildings. In this section, we will explore the deformation of these harmonic maps in families (see \cref{harmonic2}). Analogous to the canonical form introduced in \cref{def:caninical form2}, there is a notion of \emph{canonical currents} (see \cref{def:canonical}) associated with representations of fundamental groups into algebraic groups over non-archimedean local fields. Our ultimate goal is to apply \cref{harmonic} to establish a result similar to \cref{harmonic2} concerning the deformation of canonical currents (see \cref{lem:continuity}).
 \subsection{NPC spaces} 
 For the definitions in this subsection, we refer the readers to \cite{bridson-haefliger,Rou09,KP23}. 
 \begin{dfn}[Geodesic  space] Let $(X,d_X)$ be a metric space.  A curve $\gamma:[0, \ell] \rightarrow X$ into $X$ is called a geodesic if the length $d_X(\gamma(a),\gamma(b))=b-a$ for any subinterval $[a, b] \subset[0,\ell]$.  A metric space $(X,d_X)$ is a \emph{geodesic space} if there exists a geodesic connecting every pair of points in $X$.
 \end{dfn}
 \begin{dfn}[NPC space]An NPC (non-positively curved) space $(\cN,d_\cN)$ is a complete geodesic space that satisfies the following condition: for any three points $P,Q,R\in \cN$ and a  geodesic $\gamma:[0, \ell] \rightarrow \cN$ with $\gamma(0)=Q$ and $\gamma(\ell)=R$, we have
 	$$
 	d^{2} (P, Q_{t} ) \leq(1-t) d^{2}(P, Q)+t d^{2}(P, R)-t(1-t) d^{2}(Q, R)
 	$$
 	for any $t\in [0,1]$, where $Q_{t}:=\gamma(t\ell)$.
 \end{dfn}
 
 An NPC space generalizes a Hadamard manifold, which is a complete, simply connected Riemannian manifold with non-positive sectional curvature. A relevant example of an NPC space that is not a Riemannian manifold, pertinent to this paper, is the Bruhat-Tits building $\Delta(G)_K$ associated with a reductive algebraic group $G$ defined over a non-archimedean local field $K$. While we do not provide the full definition of Bruhat-Tits buildings here, readers can refer to sources such as \cite{Rou09} and \cite{KP23} for precise details. Notably, a Bruhat-Tits building consists of a collection of apartments, each of which is an isometric and totally geodesic embedding of Euclidean space $\mathbb{R}^N$, where $N = \dim(\Delta(G)_K)$. Additionally, the group $G(K)$, which denotes the $K$-points of $G$, acts isometrically on $\Delta(G)_K$ and transitively on its set of apartments. The dimension of $\Delta(G)_K$ equals the $K$-rank of the algebraic group $G$, which is the dimension of a maximal $K$-split torus in $G$.
 
In this paper, we mainly work on two types of NPC spaces:  
\begin{enumerate}[label=(\arabic*)]  
	\item The Riemannian symmetric spaces \(\GL_{N}(\mathbb{C})/{\rm U}_N\) and \(\SL_N(\mathbb{C})/{\rm SU}_N\).  
	\item The Bruhat-Tits building \(\Delta(G)_K\) of a reductive algebraic group \(G\) over a non-archimedean local field \(K\).  
\end{enumerate}

 %	A smooth Riemannian manifold with nonpositive sectional curvature  is  an NPC  space. Among these, the Bruhat-Tits building $\Delta(G)$ associated with a reductive algebraic group  $G$ defined over a non-archimedean local field  $K$ is noteworthy an example  of NPC spaces.   We will not provide the lengthy definition of Bruhat-Tits buildings here, but interested readers can find precise definitions in references such as \cite{Rou09} and \cite{KP23}.  It is noteworthy that $G(K)$ acts isometrically on the building $\Delta(G)$, and transitively on its set of apartments.  Here, $G(K)$ denotes the group of $K$-points of $G$. The dimension of $\Delta(G)$ is equal to the $K$-rank of the algebraic group $G$, which is the dimension of a maximal $K$-split torus in $G$. 
 
 %Let $G$ be a semisimple algebraic group over a non-archimedean local field $K$. The Bruhat-Tits building $\Delta(G)$ associated with $G$ is a locally finite Euclidean building, and 
 
 \subsection{Some preliminary of harmonic maps to NPC spaces}\label{subsec:NPC}
 For further details of  definitions and results in this subsection, we refer the readers to \cite{GS92,KS93,KS97}, \cite[\S 2.2]{BDDM} and \cite[\S 4]{DM24}. 
 
 Consider a map $f: \Omega\to \cN$ from an  $n$-dimensional Riemannian manifold $(\Omega, g)$ to an NPC space $(\cN, d_\cN)$.  When the target space $\cN$ is a smooth  Riemannian manifold of nonpositive sectional curvature,  the energy of a smooth map $f: \Omega \rightarrow \cN$ is
 $$
 E^{f}=\int_{\Omega}|d f|^{2} {\rm dvol}_g
 $$
 where $(\Omega, g)$ is a Riemannian domain and ${\rm dvol}_g$ is the volume form of $\Omega$.   We say $f: \Omega\to \cN$ is harmonic if it is locally energy minimizing; i.e.~for any $x \in \Omega$, there exists $r>0$ such that the restriction $\left.u\right|_{B_{r}(x)}$ minimizes energy amongst all finite energy maps $v: B_{r}(x) \rightarrow  {\cN}$ with the same boundary values as $\left.u\right|_{B_{r}(x)}$. Here $B_{r}(x)$ denotes the geodesic ball of radius $r$ centered at $x$.
 
 In this paper, we also consider the target $\cN$ to be  NPC spaces, not necessarily smooth. Let us recall the definition of  harmonic maps in this context  (cf. \cite{KS93} for more details).
 
 Let $(\Omega, g)$ be a bounded Lipschitz Riemannian domain. Let $\Omega_{\ep}$ be the set of points in $\Omega$ at a distance least $\ep$ from $\partial \Omega$.  Denote by $S_{\ep}(x):=\partial B_{\ep}(x)$. We say $f: \Omega \rightarrow \cN$ is an $L^{2}$-map (or that $f \in L^{2}(\Omega, \cN)$ ) if for some point $P\in \Omega$, we have  
 $$
 \int_{\Omega} d^{2}(f(x), P) d \mathrm{vol}_{g}<\infty.
 $$
 For $f \in L^{2}(\Omega, \cN)$, define
 $$
 e^f_{\ep}: \Omega \rightarrow \mathbb{R}, \quad e^f_{\ep}(x)= \begin{cases}\int_{y \in S_{\ep}(x)} \frac{d^{2}(f(x), f(y))}{\ep^{2}} \frac{d \sigma_{x, \ep}}{\ep^{n-1}} & x \in \Omega_{\ep} \\ 0 & \text { otherwise }\end{cases}
 $$
 where $\sigma_{x, \ep}$ is the induced measure on $S_{\ep}(x)$. We define a family of functionals
 $$
 E_{\ep}^{f}: C_{c}(\Omega) \rightarrow \mathbb{R}, \quad E_{\ep}^{f}(\varphi)=\int_{\Omega} \varphi e^f_{\ep} d\vol_{g} .
 $$
 We say $f$ has \emph{finite energy}, denoted by   $f \in W^{1,2}(\Omega, \cN)$,  if
 $$
 E^{f}[\Omega]:=\sup _{\varphi \in C_{c}(\Omega), 0 \leq \varphi \leq 1} \limsup _{\ep \rightarrow 0} E_{\ep}^{f}(\varphi)<\infty .
 $$ 
 In this case, it was proven in \cite[Theorem 1.10]{KS93} that   there exists an absolutely continuous function $e^f(x)$ with respect to   Lebesgue measure,   which we call the \emph{energy density}, such that $e^f_{\ep}(x) d \mathrm{vol}_{g} $ converges weakly to $ e^f(x) {\rm dvol}_g$ as $\ep$ tends to $0$. In analogy to the case of smooth targets, we write $|\nabla f|^{2}(x)$ in place of $e^f(x)$. Hence $|\nabla f|^{2}(x)\in L^1_{\rm loc}(\Omega)$. In particular, the (Korevaar-Schoen) energy of $f$ in $\Omega$ is
 \begin{align}\label{eq:defenergy}
 	E^{f}[\Omega]=\int_{\Omega}|\nabla f|^{2} {\rm dvol}_g . 
 \end{align}
 Furthermore, we say $f \in L^2(\Omega, \cN)$ has \emph{locally finite energ}y if for any $x_0 \in \Omega$, there exists a bounded Lipschitz Riemannian subdomain $\Omega_{x_0}$ containing $x_0$ such that $f \in W^{1,2}(\Omega_{x_0}, \cN)$.
 \begin{rem} \label{equivariantmaps}
 	If $(M,g)$ is a closed Riemannian manifold and $\varrho:\pi_1(M)\to {\rm Isom}(\cN)$  is a homomorphism, where ${\rm Isom}(\cN)$ is the isometry group of the NPC space $(\cN,d_\cN)$, then for any  $\varrho$-equivariant  map $u:\widetilde{M}\to \cN$ with locally finite energy, the energy density $|\nabla u|^2$ is a locally $L^1$-function on $\widetilde{M}$ that is invariant under $\pi_1(M)$-action. It then descends to a function on $M$,  which we abusively denote by $|\nabla u|^2$.  Denote
 	\[
 	^gE^u := \int_M |\nabla u|^2 d\vol_g
 	\]
 	if this integral exists (as a finite number). 		In this case,  we say $u$ is a  finite energy map and define $^gE^u$ as the energy of $u$. We write 
 	$$
 	u \in W^{1,2}_\varrho(\tilde M, \cN).
 	$$   If the dependence of energy on the metric $g$ is clear from context, we will omit the superscript $g$ and write $E^u$ instead.
 \end{rem}

 \begin{dfn}[Harmonic maps]
 	We say a continuous map $f: \Omega \rightarrow \cN$ from a Lipschitz Riemannian domain $(\Omega,g)$ is harmonic if it is locally energy minimizing; more precisely, at each $x \in \Omega$, there exists  
 	$r>0$ such that $u|_{B_r(x)} \in W^{1,2}(B_r(x),\cN)$ and $E^{u|_{B_r(x)}} \leq E^v$ for any $v \in W^{1,2}(B_r(x),\cN)$ with $\mbox{tr}(u|_{B_r(x)}) =\mbox{tr}(v)$.  Here,  $\mbox{tr}(\cdot)$ is the trace of a finite energy map defined in \cite[\S1.12]{KS93}.
 \end{dfn} 
 
 For  $V \in \Gamma \Omega$ where $\Gamma \Omega$ is the set of
 Lipschitz vector fields  on $\Omega$,  in \cite[\S 2.3]{KS93}, the \emph{directional energy} $|f_*(V)|^2$  is similarly
 defined.  The real valued $L_{\rm loc}^1$ function $|f_*(V)|^2$ generalizes
 the norm squared on the directional derivative of $f$.  The
 generalization of the pull-back metric is the continuous, symmetric, bilinear, non-negative and tensorial operator
 \[
 \pi_f(V,W)=\Gamma \Omega \times \Gamma \Omega \rightarrow
 L^1(\Omega, \bR)
 \]
 where
 \[
 \pi_f(V,W)=\frac{1}{2}|f_*(V+W)|^2-\frac{1}{2}|f_*(V-W)|^2.
 \]
 We refer to  \cite[\S 2.3]{KS93} for more
 details.

 Let  $(x_1, \dots, x_n)$ be local coordinates of $(\Omega,g)$, and $g=(g_{ij})$, $g^{-1}=(g^{ij})$ be the local metric expressions.  Then energy density function of
 $f$ can be written (cf.~\cite[(2.3vi)]{KS93})
 \[
 |\nabla f|^2 =  g^{ij} \pi_f(\frac{\partial}{\partial x_i}, \frac{\partial}{\partial x_j})
 \]
 Next assume  $(\Omega, g)$ is a Hermitian domain and let $(z_1=x_1+ix_2, \dots, z_n=x_{2n-1}+ix_{2n})$ be local complex coordinates.   If we extend $\pi_f$ linearly  over $\bC$, then we have
 \[
 \frac{1}{4}|\nabla f|^2= g^{i\bar j} \pi_f ( \frac{\partial f}{\partial z_i},
 \frac{\partial f}{\partial \bar z_j}).
 \]

 \begin{dfn}[Proper action] \label{def:proper}
 	Let $\Gamma\subset {\rm Isom}(\cN)$ be a finitely generated group with  $\gamma_1, \dots, \gamma_p$ being the finite set of generators, where  ${\cN}$ be an NPC space. Define
 	\begin{align}\label{eq:delta}
 		\delta: \cN \to [0,\infty), \ \ \  \delta(P) = \max d( \gamma_\alpha.P, P),  \alpha=1, \dots,p.
 	\end{align} 
 	We will call the action \emph{proper} if the sublevel sets of the function $\delta$ are bounded in $\cN$; i.e. if $P_0 \in \cN$, then for any $L > 0$, there is a number $R > 0$ (depending on $L, P_0$) so that
 	$\{P \in \cN :\delta(P)\leq L\} \subset B_R(P_0).$	
 \end{dfn}
% The following result by \cite[Theorem 2.2.1]{KS97}  is crucial for us. 
The following result is implicitly addressed in \cite{KS97}. For the sake of completeness, we provide a detailed proof here.
 \begin{lem}\label{lem:proper}
 	Let $\Gamma$ and $\cN$ be as in \cref{def:proper}. If $\Gamma$ does not fix a point at the visual boundary $\partial \cN$ of $\cN$ (see \cite[Chapter II.8]{bridson-haefliger} for the definition), then the action of $\Gamma$ on $\cN$ is proper. 
 \end{lem}   
 \begin{proof}
 Let $\overline{\cN}:=\cN\cup\partial \cN$. We equip $\overline{\cN}$ with the \emph{cone topology} (see \cite[Chapter II, Definition 8.6]{bridson-haefliger}) Since $\cN$ is locally compact, $\overline{\cN}$ is compact. 
 
 	Assume that $\Gamma$ is not proper. Then there exists  some $L>0$ such that $C_L:=\{P\in \cN\mid \delta(P)\leq L\}$ is an unbounded subset in $\cN$. Then we can choose  a sequence   of points $\{P_n\}_{n\in \bN}\subset \cN$ such that $\lim\limits_{n\rightarrow\infty}P_n=Q$ for some $Q\in \partial\cN$. Note that for any $P_n$, we have
 	$$
 	d(\gamma_\alpha. P_n,P_n)\leq L,
 	$$
 	where $\{\gamma_1,\ldots,\gamma_p\}$ is a finite set of generators of $\Gamma$. 
 	This implies that, for any $\gamma\in \Gamma$,  one has
 	$$
 	d(\gamma. P_n,P_n)\leq \ell(\gamma) L,
 	$$
 	where $\ell(\gamma)$ is the word length of $\gamma$ with respect to  $\{\gamma_1,\ldots,\gamma_p\}$.  Thus, we have $\gamma.Q=Q$ for each $\gamma\in \Gamma$. This contradicts with the assumption that $\Gamma$ does not fix a point at $\partial N$.  The lemma is proved. 
 \end{proof}

 \begin{dfn}\label{proper}
 	Let  $M$ be a closed Riemannian manifold without boundary,  $\cN$ be a locally compact NPC space, and $\varrho:\pi_1(M)\to \cN$ be a representation.  The metric space $(L^2_{\varrho}(\tilde M, \cN),d_2)$ is defined as follows:  	Let $L^2_{\varrho}(\tilde M, \cN)$ to be the space of $\varrho$-equivariant, locally $L^2$-maps. The  distance function $d_2$ is given by
 	$$
 	d_2(u, v) = \left( \int_M d^2(u(x),v(x))\, d\mu  \right)^{\frac{1}{2}}.
 	$$
 	Here, since $x \to  d^2(u(x),v(x))$ is $\pi_1(M)$-invariant on $\tilde M$, it descends to a
 	function on $M$. By \cite[Lemma 2.1.2]{KS97},  $(L^2_{\varrho}(\tilde M, \cN), d_2)$ is an NPC space.  Note that  $W^{1,2}_{\varrho}(\tilde M, \cN)$ is a  subspace of $L^2_{\varrho}(\tilde M, \cN)$ (cf.~\cref{equivariantmaps}).
 \end{dfn}
 
 \begin{dfn}\label{def:energymini}
 	% Let $X$, $\cN$, and $\varrho$ be as in \cref{proper}.  We say a finite energy $\varrho$-equivariant map\footnote{Do we need assume Lipschitz continuous?  {\color{red} See footnote (b)}} $u: \widetilde{X} \rightarrow \cN$ is   \emph{within $\varepsilon$ of minimizing}  if the energy $E^u$ satisfies 
 	%		$$
 	%		E^u \leq \inf E^v+\varepsilon,
 	%		$$ 
 	%		where the $v$ range through all finite energy $\varrho$-equivariant 
 	Let $M$, $\cN$, and $\varrho$ be as in \cref{proper}.  We say $\{u_k\}$ is a energy minimizing sequence if $u_k \in W^{1,2}_\varrho(\tilde M, \cN)$ and 
 	$$
 	\lim_{k \to \infty} E^{u_k} =\inf_{v \in W^{1,2}_\varrho(\tilde M, \cN)} E^v.  
 	$$   
 \end{dfn}

 %%%%%%%%%

 %	\begin{thm}[\protecting{\cite[Theorem 3.9]{KS97}}]
 	%	\label{ksthm3.9}
 	% Let $X$ be a \sout{closed} {\color{red} compact} Riemannian manifold,  $\cN$ an NPC space and $\varrho:\pi_1(X)\to {\rm Isom}(\cN)$ be a representation.   If $\left\{u_k: \widetilde{X} \rightarrow \cN\right\} \subset Y$ converges locally uniformly, in the pullback sense, to a limit map $u: \bar{X} \rightarrow \cN$ and  $u_k$ is within $\varepsilon_k$ of minimizing, with $\varepsilon_k \rightarrow 0$ as $k \rightarrow 0$, then $u$ is an energy minimizing equivariant map. Furthermore, the Sobolev energy density measures and the directional energy density measures of the $u_k$ converge weakly to those of $u$.
 	%	\end{thm}

 \begin{proposition} \label{arzela}
 	Assume  $M$ is compact,  $\cN$ is  locally compact and  $\varrho$ is proper.
 	If $\{u_k\}$ is a energy minimizing sequence  that are uniformly locally Lipschitz continuous,  then a subsequence of $\{u_k\}$  converges uniformly to a harmonic map $v$. Moreover, the   energy density measures and the directional energy density measures of the $u_k$ converge weakly to those of $v$. 
 \end{proposition}   
 \begin{proof}
 	Let $\cF \subset \widetilde M$ be a fundamental domain for the action of $\pi_1(M)$ on $\widetilde M$. We fix a point $P_0 \in \cN$.  The compactness of  $\overline{\cF}$ and the uniform local Lipschitz continuity of $\{u_k\}$ implies that there exists $c>0$ such that
 	$$d(u_k(\gamma_\alpha x), u_k(x)) \leq  cd_{\widetilde M} (\gamma_\alpha x, x), \ \ \forall x \in \cF, \ \ \alpha=1,\dots,p
 	$$
 	where $\gamma_1, \dots, \gamma_p$ are the generators of $\pi_1(M)$.
 	The equivariance of $u_k$ implies  that  for any $x \in \cF$ and any    $\alpha\in \{1,\dots,p\}$, we have
 	$$
 	d(\varrho(\gamma_\alpha)(u_k(x)), u_k(x))= 	d( u_k(\gamma_\alpha.x), u_k(x))  \leq  c_1:=c \sup_{\beta=1, \dots, p; x\in \cF}  d_{\widetilde M} (\gamma_\beta(x), x).
 	$$
Since $\varrho$ is proper, it follows that $\delta(u_k(x)) \leq c_1$ for $x \in \cF$ and for all $k$, where $\delta$ is the function defined in \eqref{eq:delta}. 
 	\begin{claim}
 		There exists some constant $R>0$ such that $d(u_k(\gamma.\overline{\cF}),\varrho(\gamma).P_0) \leq    R$ for any $k \in \bN$ and any $x\in \overline{\cF}$.
 	\end{claim}
 	\begin{proof}
 		By \cref{def:proper}, there exists some $R>0$ such that
 		$$
 		\{P \in \cN \mid \delta(P)\leq c_1\} \subset B_R(P_0). 
 		$$
 		It follows that $u_k(\cF)\subset B_R(P_0)$ for any $k \in \bN$.  
 		Since $u_k$  is  $\varrho$-equivariant and continuous, it follows that for any  $x\in \overline{\cF}$, and any $k \in \bN$, we have
 		$$
 		d(u_k(\gamma.x),\varrho(\gamma).P_0)=   		 d(\varrho(\gamma).u_k( x),\varrho(\gamma).P_0)=d( u_k( x), P_0)\leq  R. 
 		$$  
 	\end{proof}  
 	Since $\cN$ is locally compact, $u_k$ is uniformly locally Lipschitz continuous, and $\widetilde{M}=\cup_{\gamma\in \pi_1(M)}\overline{\cF}$,   by the Arzela-Ascoli theorem, the above claim yields that there exists  a subsequence of $\{u_k\}$ that converges    to a  map $v\in W^{1,2}_\varrho(\tilde M, \cN)$.  The lower semicontinuity of energy implies that $v$ is energy minimizing, and thus $\nu$ is harmonic.   Last statement follows from  \cite[Theorem 3.9]{KS97}. 
 \end{proof}

 %\begin{thm}[\protecting{\cite[Theorem 3.9]{KS97}}]
 %Let $M$ be a closed Riemannian manifold.  Let $\varrho:\pi_1(M)\to G(K)$ be a representation.   Let $\left\{u_k: \widetilde{M} \rightarrow \cN\right\}$ be a sequence of $\varrho$-equivariant maps to an NPC space $\cN$. Assume $\left\{u_k\right\}$ converges locally uniformly, in the pullback sense, to a limit map $u: \bar{M} \rightarrow X$. Suppose that the energies of the $u_k$ are uniformly bounded and that $u_k$ is within $\varepsilon_k$ of minimizing, with $\varepsilon_k \rightarrow 0$ as $k \rightarrow 0$. Then $u$ is a minimizing equivariant map. Furthermore, the  energy density measures and the directional energy density measures of the $u_k$ converge weakly to those of $u$. 
 %\end{thm}
 
 \subsection{Regular and singular points of harmonic maps}

 %	 	\begin{dfn}[Regular points and singular points] \label{def:sing}
 	% 		Let $X$ be a projective variety, $\widetilde X$ its universal cover, $\pi_X:\widetilde X \to X$ the universal covering map, and $u:\widetilde X \to \Delta(G)_K$ a  harmonic map.		A point $x \in \widetilde X$ is said to be a {\it regular point} of ${u}$  if there exists a neighborhood $\cU$ of $x$ in $X$ and an apartment $A \subset \Delta(G)_K$ such that $ {u}(\cU) \subset A$.    The set of regular points of $u$ is denoted by $ {\cR}(u)$.   Assume further that there is a representation $\varrho:\pi_1(X)\to  G(K)$ such that  $ {u}$ is $\varrho$-equivariant. Since $G(K)$ acts transitively on the apartments of $\Delta(G)_K$,    it follows that if $x \in \widetilde{X}$ is a regular point  of $ {u}$, then every point of $\pi_X^{-1}(\pi_X(x))$ is a  regular point  of $ {u}$. 
 	% 		We abusively denote by $ \cR( {u})\subset X$ the image of the set of regular points of $u$ under the universal cover map $\widetilde{X}\to X$. 
 	% 	\end{dfn} 
 
 In \cite{GS92}, Gromov and Schoen investigated the regularity properties of harmonic maps into a class of NPC spaces they defined as $F$-connected. It is straightforward to verify that Bruhat-Tits buildings satisfy the conditions for being $F$-connected.

 \begin{dfn}[$F$-connected] An $N$-dimensional complex $\mathfrak F$ is called {\it $F$-connected} if it is an NPC space, each of its simplices is isometric to a linear image of the standard simplex and any two adjacent simplices are contained in an apartment $A$,  an isometric and totally geodesic subcomplex isometric to $\bR^N$.
 \end{dfn}
 
 \begin{dfn}[Regular points and singular points]  	\label{def:regular}	 Let $u:\Omega \to \mathfrak F$ be a harmonic map from a Riemannian domain into an $F$-connected complex.  	A point $x \in \Omega$ is said to be a {\it regular point} of ${u}$  if there exists a neighborhood $\cU$ of $x$ in $X$ and an apartment $A \subset \mathfrak F$ such that $ {u}(\cU) \subset A$.    The set of regular points of $u$ is denoted by $ {\cR}(u)$.   
 \end{dfn}
 
 \begin{thm}[\cite{GS92}]
 	\label{gsregularity}
 	If  $u:\Omega \to \mathfrak F$ is a harmonic map from an $n$-dimensional Riemannian domain $\Omega$ into an $F$-connected complex, then the Hausdorff dimension of $\cS(u)$ is at most $n-2$.
 \end{thm}

 % \begin{dfn} \label{singset}
 	% 	%For a point $P_0$ in a $N$-dimensional $F$-connected complex  $\cF$, let $\cF_{P_0}$ denote the tangent cone of $\cF$ at $P_0$.
 	% 	A point $x_0 \in \Omega$ is a \emph{regular point} of a harmonic map $u:\Omega \rightarrow \cF$ if there exists $\sigma_0>0$ and a $N$-flat $A \subset \cF$ such that $u(B_{\sigma_0}(x_0)) \subset A$.  The set $\mathcal R(u)$ of regular points is called the \emph{regular set}.  The set $\mathcal S(u)=\Omega \backslash \mathcal R(u)$ is the \emph{singular set}.  A point of $\mathcal S(u)$ is a \emph{singular point}.
 	% \end{dfn}

 \begin{dfn}\label{ord}
 	Let $u$ be as in \cref{gsregularity}.  For $x \in \Omega$,  set
 	\[
 	{\rm Ord}^u(x) = \lim_{\sigma \rightarrow 0} \frac{e^{c\sigma^2} \sigma \displaystyle{ \int_{B_\sigma(x)} |\nabla u|^2 d\mu} }
 	{\displaystyle{\min_{Q \in \Delta(G)} \int_{\partial B_\sigma(x)} d^2(u,Q) d\Sigma}}.
 	\]
 \end{dfn} 
 As a limit of non-decreasing sequence of functions, $x \mapsto {\rm Ord}^u(x)$ is a upper semicontinuous function.  Thus, we have the following:
 \begin{enumerate}[label=(\alph*)]
 	\item By \cite[Lemma 1.3]{GS92}, ${\rm Ord}^u(x) \geq 1$ for all $x \in \Omega$.
 	\item By \cite[Theorem 6.3.(i)]{GS92},  if $x_i \rightarrow x$ and ${\rm Ord}^u(x_i)>1$, then ${\rm Ord}^u(x) >1$. 
 \end{enumerate} 
 \begin{lem}[\cite{GS92},  proof of Theorem 6.4]  \label{snot}
 	Let $u$ be as in \cref{gsregularity} and $\tilde{\mathcal S}_0(u)$ to be the set of points $x \in \Omega$ such that ${\rm Ord}^u(x)>1$.  Then $\tilde{\mathcal S}_0(u)$ is a closed set such that $\dim_{\mathcal H}(\tilde{\cS}_0(u)) \leq n-2$. \qed
 \end{lem}
 
 \begin{lem}[\cite{GS92}, proof of Proposition 2.2, Theorem 2.3] \label{lem23}
 	Let $u$  be as in  \Cref{gsregularity}.  For $x \in \Omega$, let $\alpha:={\rm Ord}^u(x)$. There exists a constant $c>0$ and $\sigma_0>0$ such that 
 	\[
 	\sigma \mapsto \frac{e^{c\sigma^2}}{\sigma^{n-1+2\alpha}} \int_{\partial B_\sigma(x)} d^2(u,u(x)) d\Sigma
 	\]
 	and
 	\[
 	\sigma \mapsto \frac{e^{c\sigma^2}}{\sigma^{n-2+2\alpha}} \int_{B_\sigma(x)} |\nabla u|^2 d\mu
 	\]
 	are non-decreasing functions in the interval $(0,\sigma_0)$. \qed
 \end{lem}

 \begin{rem} \label{rem:energydensity}
 	For a finite energy map $u:\Omega \rightarrow \mathfrak F$ into an $F$-connected complex,   $|\nabla u|^2\in L^1_{\rm loc}$ is not necessarily defined at all points of $\Omega$.  On the other hand,  it follows from Lemma~\ref{lem23} that 
 	for a harmonic map $u$, 
 	we can define $|\nabla u|^2$ at every point of $x \in \Omega$ by setting
 	\[
 	|\nabla u|^2(x) = \lim_{\sigma \rightarrow 0} \frac{1}{c_n \sigma^n}
 	\int_{B_\sigma(x)} |\nabla u|^2 \, d\mu
 	\] 
 	where $c_n \sigma^n$ is the volume of a ball or radius $\sigma$ in Euclidean space.  
 \end{rem}

 To further study  the local behavior of harmonic maps, we consider the following notion (cf.~\cref{localproperties} below).
 
 \begin{dfn}[$F$-connected cone]
 	An $F$-connected complex $T$ is called an {\it $F$-connected cone} if there exists an isometric embedding of $T$ in Euclidean space as a geometric cone; that is,  after identifying $T$ as a subset of Euclidean space via the isometric embedding, we have 
 	\begin{equation} \label{cone}
 		Q \in T \Rightarrow \lambda Q \in T, \forall \lambda \in [0,\infty).
 	\end{equation} 
 	Henceforth, we will always assume that $T$ is a subset of Euclidean space.  Note that  the origin $\vec{0}$ of the Euclidean space is  the cone point, i.e.~the vertex of the cone. 
 \end{dfn}
 
 \begin{rem}	\label{localproperties}
 	A neighborhood of a point $P_0 \in \Delta(G)_K$ is isometric to a neighborhood of the origin in the tangent cone $T_{P_0}\Delta(G)$ (cf. \cite[p.~190]{bridson-haefliger} for the definition).  Two simplices (which are actually simplicial cones) in $T_{P_0}\Delta(G)$ are contained in a totally geodesic subcomplex $T_{P_0}A$ where $A$ is an apartment of $\Delta(G)_K$. In other words, $T_{P_0}\Delta(G)_K$ is  an  $N$-dimensional $F$-connected  complex.
 	Thus, when we study the local behavior of harmonic maps $u:\Omega \to \Delta(G)_K$ at a point $x_0 \in \Omega$, we can assume that $u$ maps into the $N$-dimensional, $F$-connected cone   $T:=T_{P_0}\Delta(G)_K$ where $P_0=u(x_0)$.   
 \end{rem}
 % \[
 % \int_{\partial B_1(0)} d^2(u^2_*, u^2_*(0)) d\Sigma =\lim_{n \rightarrow \infty} \int_{\partial B_1(0)} d^2(u^2_{s_j}, u^2_{s_j}(0)) d\Sigma=0.
 % \]
 % Thus, for any sequence $s_j \rightarrow 0$, the blow up maps $u_{s_j}$ converges locally uniformly to 
 % \[
 % u_*=L: \bB_1(0) \rightarrow \bR^m.
 % \]
 % 

 %%%%%

 The following theorem follows from \cite[Sections 5 and 6]{GS92}
 \begin{thm}[\cite{GS92}]
 	\label{thm5.1}
 	Let $T$ be an $F$-connected cone  and  $l: \bR^n \to T$ be a homogeneous degree 1 harmonic map; i.e.~$l(\lambda x)=\lambda l(x)$.  
 	\begin{thmlist}
 		\item  The map $l$ can be decomposed as $l=J \circ v$ where $J:\bR^m \to T$ is an isometric and totally geodesic map  and $v:\bR^n \to \bR^m$ is a linear map of rank $m$.
 		\item Let $S \subset T$ be the union of all apartments containing the $m$-flat $\bF=l(\bR^n)$.  Then $S$ is an $F$-connected cone isometric to $\bR^m \times T_2$ where $T_2$ is a lower dimensional $F$-connected cone.  Furthermore, $S$ is essentially regular and $l$ is effectively contained in $S$  in the sense of \cite[Section 5]{GS92}.
 		\item\label{item:dependence} Let $\cM$ be a family of  Riemannian metrics defined on a Euclidean unit ball $\bB_1(0) \subset \bR^n$    
 		and, for $g \in \cM$, define
 		\begin{equation} \label{normg}||g||=\sum_{i,j=1}^n \sup_{\bB_1(0)} |g_{ij}-\delta_{ij}| + \sum_{i,j,k=1}^n \sup_{\bB_1(0)}\left| \frac{\partial g_{ij}}{\partial x_k} \right| 	+ 	\sum_{i,j,k,l=1}^n \sup_{\bB_1(0)}\left| \frac{\partial^2 g_{ij}}{\partial x_k \partial x_l} \right|	  \end{equation}
 		where  $g=(g_{ij})$ is the matrix expression of $g$  with respect to the Euclidean coordinates $(x_1, \dots, x_n)$.    Then there exists  $\kappa_0>0$, $\delta_0>0$ and $\sigma_0\in (0,1)$ such that if  $g \in \cM$ with $||g||<\kappa_0$ and   $u:(\bB_1(0), g) \to T$ is a harmonic map with 
 		\[
 		\sup_{\bB_1(0)} d(u(x),l(x)) \leq \delta_0,
 		\]
 		then $u(\bB_{\sigma_0}(0)) \subset S$. \qed
 	\end{thmlist}
 \end{thm} 
 	\begin{rem} 
 		\Cref{item:dependence} explicitly addresses the dependence on the domain metric $g$ in \cite[Theorem 5.1]{GS92}. This dependence can be deduced by following the proof in either \cite{GS92} or \cite[Lemma 5.5]{DM21}. Specifically, $\kappa_0$ must be chosen sufficiently small such that, if $||g|| \leq \kappa_0$,
 		then $g$ meets the criteria stated prior to \cite[Lemma 5.5]{DM21}; namely:
 		
 		\begin{itemize}
 			\item For any smooth submanifold $S$ of $\bB_1(0)$,   $\frac{15}{16} \vol(S) \leq \vol_{g}(S) \leq \frac{17}{16} \vol (S)$ where $\vol(S)$ is the Euclidean volume.
 			\item The error term $e^{c\sigma^2}$ that appears in the monotonicity formulae of \cref{lem23} is $\leq 2$ for all $\sigma \in (0,1]$.
 			\item There exists a constant $c_0$  such that for any subharmonic function $f:\bB_1(0) \rightarrow \bR$  with respect to the metric $g$, we have 
 			$$
 			\sup_{\bB_{\frac{15r}{16}}(0)} f  \leq  \frac{c_0}{r^{n-1}} \int_{\partial \bB_r(0)} f d\Sigma.
 			$$
 		\end{itemize}
 	\end{rem}

 Let $u:\Omega \rightarrow T$ be a harmonic map into an $F$-connected cone with cone point  $u(x_0)=0$. 
 Use normal coordinates centered at $x_0$ to identify $x_0=0$ and let  $\bB_r(0)=\{x=(x^1, \dots, x^n) \in \bR^n:  |x|<r\}$.  
 Let 
 \[
 \mu(\sigma)= \left(\sigma^{1-n} \int_{\partial \bB_\sigma(0)} d^2(u,u(0)) d\Sigma \right)^{-\frac{1}{2}}.
 \]
 
 \begin{dfn} \label{blowups} The {\it blow up map} is defined by
 	\[
 	u_\sigma:\bB_1(0) \rightarrow T, \ \ \ u_\sigma(x)=\mu(\sigma) u(\sigma x).
 	\]
 	Here, we are using the property \cref{cone} of the geometric cone $T$ in Euclidean space. 	\end{dfn}
 
 By \cite[Proposition 3.3]{GS92} and the paragraph proceeding it, there exists a sequence $\sigma_i \rightarrow 0$ such that  $u_{\sigma_i}$ converges locally uniformly to a non-constant homogeneous harmonic map $u_*$ of degree $\alpha:={\rm Ord}^u(x_0)$. Consequently,   \cref{thm5.1} (see also \cite[Theorem 6.3]{GS92}) implies the following.
 
 \begin{lem}[\cite{GS92}] \label{item:splitting}
 	If ${\rm Ord}^u(x_0)=1$, then 
 	there exists $\sigma_0>0$ such that $u(\bB_{\sigma_0}(x_0)) \subset  \bR^m \times T_2$ where $T_2$ is a lower dimensional $F$-connected cone.  If we write 
 	\begin{equation} \label{productstructure}
 		u=(u^1, u^2): \bB_{\sigma_0}(x_0) \rightarrow \bR^m \times T_2,
 	\end{equation}
 	then $u^1: \bB_{\sigma_0}(x_0) \rightarrow \bR^m$ is a smooth harmonic map of rank $m$ and $u^2:\bB_{\sigma_0}(x_0) \rightarrow T_2$ is a harmonic map  with $ {\rm Ord}^{u^2}(x_0) >1$.
 \end{lem}
 % 	If ${\rm Ord}^u(x_0)=1$, then have the following: 
 % 	\begin{enumerate}[label=(\alph*)]
 	% 		\item  By~\cite[Proposition 3.1]{GS92},  there exists $m \in \{1, \dots, \min\{ n,N\} \}$ such that 
 	% 		\[
 	% 		u_*=J \circ v\big|_{\bB_1(0)}
 	% 		\]
 	% 		for  an isometric and totally geodesic  embedding $J:\bR^m \rightarrow T$ and   a  linear map $v:\bR^n \rightarrow \bR^m$ of full rank.  
 	% 		\item By \cite[Lemma 6.2]{GS92}, the union of all $N$-flats of $T$ containing $J(\bR^m)$ is isometric to   $\bR^m \times B$ where $B$ is a $(N-m)$-dimensional, $F$-connected cone.
 	% 		\item \label{item:splitting}By  \cite[Theorem 6.3]{GS92}, there exists $\sigma_0>0$ such that $u(\bB_{\sigma_0}(x_0)) \subset  \bR^m \times B$.  If we write 
 	% 		\begin{equation} \label{productstructure}
 		% 			u=(u^1, u^2): \bB_{\sigma_0}(x_0) \rightarrow \bR^m \times B,
 		% 		\end{equation}
 	% 		then $u^1: \bB_{\sigma_0}(x_0) \rightarrow \bR^m$ is a smooth harmonic map of rank $m$ and $u^2:\bB_{\sigma_0}(x_0) \rightarrow B$ is a harmonic map  with $ {\rm Ord}^{u^2}(x_0) >1$.
 	% 	\end{enumerate}

 \begin{rem} \label{def:sing}
 Let \( X \) be a smooth projective variety, \( \widetilde{X} \) its universal cover, \( \pi_X: \widetilde{X} \to X \) the universal covering map, \( \varrho: \pi_1(X) \to G(K) \) a representation, and \( u: \widetilde{X} \to \Delta(G)_K \) a \( \varrho \)-equivariant harmonic map.  
 
 Since \( G(K) \) acts transitively on the apartments of \( \Delta(G)_K \), it follows that if \( x \in \widetilde{X} \) is a regular point of \( u \), then every point of \( \pi_X^{-1}(\pi_X(x)) \) is also a regular point of \( u \).   
 We abusively denote by \( \cR(u) \subset X \) the image \( \pi_X(\cR(u)) \).   
 \end{rem}
 
 \begin{dfn}  \label{dfn:M}
 	Let $u:\Omega \rightarrow \Delta(G)_K$ be as in \cref{def:sing}.   For $x_0 \in \tilde \cS_0(u)$ (cf.~\cref{snot}), define $m_{x_0}=0$.  For $x_0 \in \Omega \backslash \tilde \cS_0(u)$, let $m_{x_0}$ be  the integer $m$ in \cref{productstructure}.   Let
 	${\rm rank}(u):=\sup_{x_0 \in \Omega} m_{x_0}$.  We say the  point $x_0 \in \Omega$ is a  \emph{critical point} if $m_{x_0} < {\rm rank}(u)$.  We denote the set of critical points by $\tilde \cS(u)$.  Define $\tilde \cR(u) = \Omega \backslash \tilde \cS(u)$. 
 \end{dfn}  
 \subsection{Equivariant harmonic maps to Bruhat-Tits buildings in families}
 Throughout this subsection, let $f:\sX\to \bD$ be a smooth projective family.  Let $\pi_{\sX}:\widetilde{\sX}\to \sX$ be the universal covering map. Then the  fiber $\widetilde{X_t}$ of $\widetilde{\sX}\to \bD$ at $t\in\bD$ is the universal cover  of $X_t:=f^{-1}(t)$.  
 \begin{dfn} \label{rmk:omega}   For $x_0 \in \widetilde{X}_0$ (resp. $x_0\in X_0$), we may take an open neighborhood $\Omega$ of $x_0$ in $\widetilde{\sX}$ (resp. $\sX$) together with  a biholomorphism $\varphi: \bD^n\times\bD_\ep\to \Omega$   such that $x_0$ is identified with $(0,\dots, 0,0)$ via $\varphi$ and $f\circ \pi_{\sX}\circ \varphi(z_1,\ldots,z_n,t)=t$ (resp. $f\circ \varphi(z_1,\ldots,z_n,t)=t$).    For simplicity of notation, we write ${\bf z}=(z_1, \dots, z_n) \in {\bD}^n$. We call such coordinate system an \emph{admissible coordinate system} centered at $x_0$.
 \end{dfn}
 
 We first prove the following regularity result in the absolute case. 
 \begin{lem}\label{lem:continuous harmonic}
 	If $u:\Omega\to \Delta(G)$ is a harmonic map from a Lipschitz Riemannian domain, then the energy density $|\nabla u|^2$ is a \emph{continuous function}.
 \end{lem}
 \begin{proof}
 	%By \cite[Theorem 2.4]{GS92}, $u$ is locally Lipschitz continuous. Thus, for $\sigma>0$ small enough,  the function  
 	By the absolute continuity of Lebesgue integrals, 
 	\begin{align*}
 		e_\sigma:\Omega &\to \bR \\
 		x &\mapsto \frac{e^{c\sigma^2}}{\sigma^{n-2+2\alpha}} \int_{\bB_\sigma(x)} |\nabla u|^2 d\mu
 	\end{align*} 
 	is continuous. By \cref{lem23}, $e_\sigma$ are non-decreasing functions in the interval $(0,\sigma_0)$ for some $\sigma_0>0$, whose limit is $|\nabla u|^2(x)$.     As a decreasing limit of continuous functions, the energy density function  $|\nabla u|^2(x)$  is upper semicontinuous. 
 	
 	Pick any $x_0\in \Omega$.  If ${\rm Ord}^u(x_0)>1$, then by \cite[Lemma 4.13]{DM24}, $|\nabla u|^2(x_0)=0$. The upper semi-continuity of $|\nabla u|^2(x)$ implies that $|\nabla u|^2(x)$ is continuous at $x_0$.
 	
 	If ${\rm Ord}^u(x_0)=1$, by \cref{item:splitting}, there exists $\sigma_0>0$ such that $u(\bB_{\sigma_0}(x_0)) \subset  \bR^m \times \cF_2$, where $\cF_2$ is another $F$-complex.  By \eqref{productstructure}, we write
 	\begin{equation*} 
 		u=(u^1, u^2): \bB_{\sigma_0}(x_0) \rightarrow \bR^m \times \cF_2,
 	\end{equation*}
 	and  $u^1: \bB_{\sigma_0}(x_0) \rightarrow \bR^m$ is a smooth harmonic map of rank $m$, and $u^2:\bB_{\sigma_0}(x_0) \rightarrow \cF_2$ is a harmonic map  with $ {\rm Ord}^{u^2}(x_0)>1$. Hence  $|\nabla u^2|(x)$ is continuous at $x_0$ by the above arguments. Since $|\nabla u|^2(x)=|\nabla u^1|^2(x)+|\nabla u^2|(x)$, and $u^1$ is smooth, it follows that $|\nabla u|(x)$ is continuous at $x_0$. The lemma is proved. 
 \end{proof}
 In the relative case, we have the following result.  
 
 \begin{lem} \label{preliminarycontinuity}
 Let \( T \) be an \( F \)-connected cone, and let \( g_t \) be a   family of Riemannian metrics on \( \bD^n \). Suppose \( u_t: \bD^n \to T \), \( t \in \bD_\ep \), is a family of uniformly Lipschitz harmonic maps with respect to \( g_t \), such that
 	\begin{itemize}
 		\item $u_t({\bf z})$ varies uniformly in $t$ in the following sense:  \ $\forall \ep_0>0$, $\exists \delta_0>0$ such that    ${\bf z}\in \bD^n,  t_1,t_2\in \bD_\ep$ with $|t_1-t_2|<\delta_0$ $\Rightarrow$ $|u_{t_1}({\bf z})-u_{t_2}({\bf z})|<\ep_0$. 
 		\item The component functions of $g_t({\bf z})$ with respect to the coordinates $(z_1, \dots, z_n)$ of $\bD^n$ varies smoothly in $t$.
 		\item For any $\bB_r({\bf z}) \subset \bD^n$, and any   $t_0\in \bD_\ep$
 		\begin{equation} \label{applyportmanteau}
 			\lim_{t \to t_0} \int_{\bB_r({\bf z})} |\nabla u_t|^2 d\vol_{g_t} = \int_{\bB_r({\bf z})} |\nabla u_{t_0}|^2 d\vol_{g_{t_0}}. 
 		\end{equation}
 	\end{itemize}
 	Then the map $({\bf z},t)=(z_1, \dots, z_n,t)   \mapsto |\nabla u_t|^2({\bf z})$ is a continuous function in $\bD^n \times \bD_\ep$.% with ${\rm Ord}^{u_{t_0}}({\bf z}_0)>1$.		
 \end{lem} 
 \begin{proof}
 	
 	We will show that $({\bf z},t) =(z_1, \dots, z_n,t) \mapsto |\nabla u_t|^2(z_1, \dots, z_n)$ is continuous at $({\bf 0},0) \in \bD \times \bD_\ep$ where ${\bf 0}=(0,\dots, 0) \in \bD^n$.  The same argument will show that this function is continuous at any other point of $\bD^n \times \bD_\ep$.  
We consider the following two cases:  
 	\\
 	\\
 	{\sc Case 1:} ${\rm Ord}^{u_0}({\bf 0}) >1$.  Thus, $|\nabla u_0|^2({\bf 0})=0$.   
 	Assume on the contrary  that $({\bf z},t)  \mapsto |\nabla u_t|^2({\bf z})$ is not  continuous  at $({\bf 0},0)$; i.e.~there exists $\ep>0$ and ${\bf z}_i=(z_{i1}, \dots, z_{in}) \to {\bf 0}$ such that $\ep \leq |\nabla u_{t_i}|^2({\bf z}_i)$. 
 	
 	To derive a contradiction, we first apply the monotonicity property of energy in \cref{lem23} and note that the order is $\geq 1$, to obtain
 	\begin{equation} \label{mp}
 		\ep \leq \frac{e^{cr}}{\omega_{2n} r^{2n}} \int_{\bB_r({\bf z}_i)} |\nabla u_{t_i} |^2 d\vol_{g_{t_i}}
 	\end{equation}
 	where $\omega_{2n}$ is the volume of the unit ball in $\bR^{2n}$. Second,
 	\begin{eqnarray*}
 		\left|  \int_{\bB_r({\bf z}_i)} |\nabla u_{t_i} |^2 d\vol_{g_{t_i}}  -   \int_{\bB_r({\bf 0})} |\nabla u_{t_i} |^2 d\vol_{g_{t_i}} \right|
 		&  \leq & \int_{\bB_r({\bf z}_i) \triangle \bB_r({\bf 0})} |\nabla u_{t_i} |^2 d\vol_{g_{t_i}}
 		\\
 		& \leq & 
 		C \vol_{g_{t_i}}(\bB_r({\bf z}_i) \triangle \bB_r({\bf 0}))
 	\end{eqnarray*}
 	where $C$ is the Lipschitz bound.
 	Combining the above two inequalities with the fact that $\vol(\bB_r({\bf z}_i) \triangle \bB_r({\bf 0})) \to 0$ as $i \to \infty$, we obtain 
 	\begin{eqnarray*}
 		\ep & \leq & \lim_{i \to \infty}  \frac{1}{\omega_{2n} r^{2n}} \int_{\bB_r({\bf 0})} |\nabla u_{t_i} |^2 d\vol_{g_{t_i}}.
 	\end{eqnarray*}
 	Combining this inequality with assumption \cref{applyportmanteau}, we get
 	\[
 	\ep \leq \frac{1}{\omega_{2n} r^{2n}} \int_{\bB_r({\bf 0})} |\nabla u_0 |^2 d\vol_{g_0}.\]
 	By \cref{rem:energydensity}, we have 
 	\[
 	\lim_{r \to 0}  \frac{1}{\omega_{2n} r^{2n}} \int_{\bB_r({\bf 0})} |\nabla u_0 |^2 d\vol_{g_0} = |\nabla u_0 |^2({\bf 0}),
 	\]  
 	we conclude  $\ep \leq |\nabla u_0 |^2({\bf 0})$, 
 	a contradiction. 
 	\\
 	\\
 	{\sc Case 2:} ${\rm Ord}^{u_0}({\bf 0}) =1$.  	Choose normal coordinates centered at $\bf 0$.
 	Let $(g_{t\sigma,ij})$ be the matrix expression of the metric $g_{t\sigma}(x):=g_t(\sigma x)$ in $\bB_1(0)$.  We will let $\cM$ be the collection of metrics $g_{t\sigma}$.  Note that $||g_{t\sigma}||\to 0$ (cf.~\cref{thm5.1})  as  $|t| \to 0$ and $\sigma \to 0$.
 	
 	Denote the distance function on $T$ by $d_T$.	  Define a scaling factor
 	$$
 	\mu(\sigma)= \left(\sigma^{1-n} \int_{\partial \bB_\sigma(0)} d^2(u_0,u_0(0)) d\Sigma \right)^{-\frac{1}{2}}. 
 	$$ and a new distance function on $T$ by setting $d_\sigma:=\mu(\sigma) d_T$.  The blow up maps of $u_0$ at $x_0$ is defined by 
 	$$u_{0\sigma}:\bB_1(0) \to (T,d_\sigma), \ \ \ u_{0\sigma}(x)=u_0(\sigma x).$$
 	Choose a subsequence $\sigma_i \to 0$ such that $u_{0\sigma_i}$ converges to a tangent map $u_{0*}$.   Also define 
 	$$u_{t\sigma}:\bB_1(0) \to (T,d_\sigma), \ \ \ u_{t\sigma}(x)=u_t(\sigma x).
 	$$  
 	Here, we note that $u_{t\sigma}$ is not the usual blow up map of $u_t$ since the choice of the scaling $\mu(\sigma)$ to define the target distance function is not consistent with the scaling used for blow up maps of $u_t$.  For the blow up maps of $u_t$, the correct scaling is  $ \left(\sigma^{1-n} \int_{\partial \bB_\sigma(0)} d^2(u_t,u_t(0)) d\Sigma \right)^{-\frac{1}{2}}$. 
 	
 	As in \cref{localproperties}, we assume $u_0$ maps into the tangent cone $T:=T_{P_0}(\Delta(G))$.  Since  ${\rm Ord}^{u_0}({\bf 0}) =1$,  $u_{0*}:\bB_1(0) \to T$ is  a homogeneous degree 1 harmonic map.  By \cref{thm5.1}, $u_{*0}$ is effectively contained in a essentially regular subcomplex $\bR^m \times T_2$ where $T_2$ is a lower dimensional $F$-connected cone.

 	% 		Since  the order of $u_0$ at $x_0$ is equal to 1,  the  image of $u_{0*}$ is a flat $\bF$ by \cite[Section 3]{GS92}.   Let $S \subset T$ be the subbuilding defined as the  union of all apartments  containing $\bF$.  By \cite[Lemma 6.2]{GS92}, $S$ is isometric to $\bR^m \times B$ where $B$ is a lower dimensional $F$-connected cone.  Henceforth, we will denote $S$ by $\bR^m \times B$.  

 	Let $\kappa_0>0$, $\delta_0>0$ and $\sigma_0>0$ be from Theorem~\ref{thm5.1}.  Fix $\tau_1>0$ and $i$ sufficiently large such that $||g_{t\sigma_i} ||<\kappa_0$ for $|t|<\tau_1$ and $\sup_{\bB_1(0)}d(u_{0\sigma_i}, u_{0*}) <\delta_0$.  Since $u_{t\sigma_i}$ converges uniformly to $u_{0\sigma_i}$, there exists $\tau_0 \in (0,\tau_1)$ such that $|t|<\tau_0$ implies  $\sup_{\bB_1(0)}d(u_{t\sigma_i}, u_{0*}) <\delta_0$.  Hence, \cref{thm5.1} implies  that there exists $\sigma_0>0$ such that  $u_{t\sigma_i}(\bB_{\sigma_0}(0)) \subset \bR^m \times B$, or equivalently, $u_t(\bB_{\sigma_0\sigma_i}(0)) \subset  \bR^m \times B$.	 % In other words, shrinking $\Omega$ if necessary, we can assume $u(\Omega) \subset \bR^m \times B$ and 
 	We write 
 	\begin{equation} \label{product}
 		u_t=(u_t^1,u_t^2), \ \ \ \ u_t^1:\bB_{\sigma_0\sigma_i}(0) \to \bR^m, \ u_t^2: \bB_{\sigma_0\sigma_i}(0) \to T_2
 	\end{equation} 
 	in terms of the product structure.  Since $u_t^1:\bB_{\sigma_0\sigma_i}(0) \to \bR^m$ is a smooth harmonic map for each $t$, the uniform convergence of $u_t^1$ to $u_0^1$ implies the local $C^\infty$-convergence of $u_t^1$ to $u_0^1$ by elliptic regularity.  Thus,  
 	$$
 	({\bf z},t) \mapsto	\frac{\partial u^1_{i}}{\partial z_j}(z_1,\ldots,z_n,t)  \mbox{ is continuous in }\bB_{\sigma_0\sigma_i}(0), \ \ \forall  i=1,\ldots,N \mbox{ and } j=1,\ldots,m.			
 	$$
 	Since ${\rm Ord}^{u_0^2}({\bf 0})>1$, {\sc Case 1} implies that  $({\bf z}, t) \to |\nabla u_t^2|^2({\bf z})$ is continuous at $({\bf 0}, 0)$.  Combining the continuity of $|\nabla u_t^1|^2$ and $|\nabla u_t^2|^2$ at $({\bf 0},0)$, we conclude that   $({\bf z}, t) \to |\nabla u_t|^2({\bf z})$ is  continuous at $({\bf 0},0)$. The lemma  is proved. 
 \end{proof}
 Let us now state and prove the main result of this section.   
 \begin{thm} \label{harmonic} 
 	Let $f:\sX\to \bD$ be the smooth projective family. 	Let $\varrho:\pi_1(\sX)\to G(K)$ be a Zariski dense representation, where $G$ is a semisimple algebraic group over    a non-archimedean local field $K$.  Then there is  a    $\varrho$-equivariant  map   $u:\widetilde{\sX}\to\Delta(G)_K$ such that
 	\begin{enumerate}[label=(\arabic*)] 
 		\item \label{item:1}For  each $t\in \bD$,  $u_t:=u|_{\widetilde{X}_t}:\widetilde{X}_t\to \Delta(G)_K$  is    $\varrho_t$-equivariant and pluriharmonic.
 		\item \label{item:2} Let $m={\rm rank}(u_0)$ and let  $x_0 \in \tilde \cR(u_0)$; i.e.~$x_0$ is a non-critical point of $u_0$ defined in \cref{dfn:M}.   Assume that $u_0$ maps some neighborhood of $x_0$ into the tangent cone $T:=T_{u(x_0)}(\Delta(G)_K)$ (cf.~\cref{localproperties}). Then there exists an admissible coordinate system  $\Omega\simeq \bD^n\times\bD_\ep$ centered at $x_0$ (cf. \cref{rmk:omega}), an apartment $A$ of $T$,  and a flat $\bF \subset A$ satisfying the following:  
 		%	Then there exists an apartment $A$ and a flat $\bF \subset A$ satisfying the following:
 		\begin{enumerate}[label=(\roman*)]
 			\item  \label{item:i} $u_0(\Omega \cap \widetilde{X}_0) \subset A$. (cf.~\cref{rmk:omega}).			\item\label{item:ii} The union of all flats parallel to $\bF$ is a subbuilding isometric to $\bR^m \times B$ where $B$ itself is another Euclidean building.  Then $u(\Omega) 	\subset \bR^m \times B$.  
 			\item \label{item:iii}
 			Write $u_t=u|_{\bD^n \times \{t\}} = (u_t^1, u_t^2)$ where 
 			$$
 			u_t^1: \bD^n \to \bR^m \mbox{ \ and \ } u_t^2: \bD^n \to B.
 			$$  
 			Here we adopt the notation that if $m=\dim(\Delta(G)_K$), then  $u_t=u_t^1$.   
 			Define  
 			$$u^1(z_1, \dots, z_n,t) := u_t^1(z_1, \dots, z_n)
 			$$ and let   $u^1=(u^1_{1}, \dots, u^1_{m})$ be its expression in terms of coordinate functions.		
 			Then we have the following:  \begin{enumerate}[label=(\alph*)]
 				\item \label{item:a}$u_0^2$ is identically constant.
 				\item  \label{item:b}The  partial derivatives of the coordinate functions of $u^1$ with respect to $z_j$ are continuous in $\Omega$, i.e. for each $i=1,\ldots,m$ and $j=1,\ldots,n$,
 				$$
 				(z_1,\ldots,z_n,t) \mapsto \frac{\partial u^1_{i}}{\partial z_j}(z_1,\ldots,z_n,t) \in C^0(\bD^n \times \bD_\ep).
 				$$
 				\item  \label{item:c0}$(z_1, \dots, z_n,t) \mapsto |\nabla u_t^2|^2(z_1, \dots, z_n)$ is a continuous function on $\Omega=\bD^n \times \bD_\ep$ and $|\nabla u_t^2|^2 \to 0$ as $t \to 0$.  
 			\end{enumerate}	 
 		\end{enumerate}
 	\end{enumerate} 
 \end{thm}  
 
 \begin{proof}[Proof of \cref{harmonic}]
 	Let $\cC$ be a $\varrho(\pi_1(X))$-invariant minimal closed convex subset of $\Delta(G)_K$ introduced in \cite[Lemma 2.2]{BDDM}. The $\varrho$-equivariant map $u:\widetilde{\sX}\to\Delta(G)_K$ is  defined by setting $u|_{\widetilde{X}_t}$ to be the  $\varrho$-equivariant pluriharmonic map into $\cC$, whose existence is established in \cite{GS92}, and uniqueness is shown in the proof of  \cite[Theorem 2.1]{BDDM}. 		This proves the existence assertion in \Cref{item:1}. 
 	
 	Before we give a proof of \Cref{item:2}, we need a preliminary result (cf.~\cref{convergence} below). Let $F_t: \widetilde{X}_0 \to \widetilde{X}_t$ be a lift of a   smooth family of  maps $X_0 \to X_t$ depending smoothly on $t \in \bD$ such that $F_0$ is the identity map.   Let $g_t$ be the pullback via $F_t$ of  the restriction to $\widetilde{X}_t$ of the Riemannian metric on $\widetilde{\sX}$.
 	For any $\varrho$-equivariant map $v: X_0 \to \Delta(G)_K$, let $^{g_t}E^v$ be the energy of $v$ with respect to the metric $g_t$ on $X_0$.  Note that, since $u_0$ and $u_t \circ F_t$ are both $\varrho$-equivariant,  the function $d^2(u_0, u_t \circ F_t)$ defined on $\widetilde{X}_0$ is $\pi_1(\sX)$-invariant, and hence descends to $X_0$.  Same is true for the   energy density functions and the directional energy density functions of $u_0$ and $u_t \circ F_t$.

 	\begin{claim} \label{convergence}
 		The function $u_t \circ F_t$ converges uniformly to $u_0$  on $\widetilde{X}_0$.			%Furthermore, the Sobolev energy density measures and the directional energy density measures of $u_t \circ F_t$ converges weakly to those of $u$.  
 	\end{claim}
 	\begin{proof}
 		Since the sequence of metrics $g_t$ converges uniformly in $C^{\infty}$ to $g_0$, there exists a constant $C>0$ such that
 		\begin{equation} \label{unienergybd}
 			(1-Ct){^{g_0} E^{u_0}} \leq (1-Ct){^{g_0} E^{u_t \circ F_t}}  \leq {^{g_t} E^{u_t \circ F_t}} \leq {^{g_t} E^{u_0}} \leq (1+Ct) {^{g_0} E^{u_0}}.
 		\end{equation}
 		Here, the first inequality is due to $u_0$ being energy minimizing with respect to the metric $g_0$.  The second inequality is due to the fact that $g_t \to g_0$.  In particular, if $(x_1, \dots, x_{2n})$ is the real coordinates on $X_0$ and $g_t^{ij}$ is the expression of the co-metric of $g_t$ with respect to these coordinates, then $g_t^{ij}=(1+O(t))g^{ij}_0$. Thus, the energy density function of $u_t$ with respect to these coordinates is $$|\nabla u_t|_{g_t}^2=g_t^{ij}\frac{\partial u_t}{\partial x_i} \cdot \frac{\partial u_t}{\partial x_j} =  (1+O(t))g^{ij}_0 \frac{\partial u_t}{\partial x_i} \cdot \frac{\partial u_t}{\partial x_j}= (1+O(t)) |\nabla u_t|_{g_0}^2$$ where $|\nabla u_t|_{g_t}^2$ (resp.~$|\nabla u_t|_{g_0}^2$) is the energy density function of $u_t$ with respect to the domain metric $g_t$ (resp.~$g_0$) and we use the notation $\frac{\partial u}{\partial x_i} \cdot \frac{\partial u}{\partial x_j}$ to denote the pullback inner product $\pi(\partial_{x_i}, \partial_{x_j})$ defined in \cite[Section 2.3]{KS93} with the coordinate vector fields $(\partial_{x_1}, \dots, \partial_{x_{2n}})$.  The third (fourth) inequality follows from a reason analogous to  the first (second) inequality. 
 		
 		In particular,  \cite[Theorem 2.4]{GS92} implies that  there exists a constant $C_0$ such that, for all $t$ is sufficiently close to $0$,  $u_t \circ F_t$ is Lipschitz continuous with Lipschitz constant bounded by $C_0$ (with respect to the distance function on $\widetilde{X}_t$ induced by the metric $g_t$, and hence also with respect to the metric $g_0$ since these metrics are uniformly equivalent).

 		Suppose that $u_t \circ F_t$ does not converge uniformly to $u_0$ as $t\to 0$.			
 		Then there exists $\ep>0$, a sequence $t_i \to 0$, $x_i \in \widetilde{X}_0$ such that  $d(u_{t_i} \circ F_{t_i}(x_i), u_0(x_i)) \geq \ep$.  
 		Since $\varrho$ is a Zariski dense representation, it does not fix a point at infinity. By \cite[Theorem 2.2.1]{KS97}, the action of $\varrho$ is proper (cf.~\cref{def:proper}). 
 		Furthermore,  $\lim\limits_{i\to\infty}{^{g_0}E^{u_{t_i}} \circ F_{t_i}} -{^{g_0}E^{u_0}}=0$ by (\ref{unienergybd}), which shows that $\{u_{t_i} \circ F_{t_i}\}$ is a minimizing sequence.
 		Taking a subsequence if necessary,   Lemma~\ref{arzela} implies that $\{u_{t_i} \circ F_{t_i}\}$ converges uniformly to a $\varrho$-equivariant energy minimizing map $v$.
 		%			By \cite[Theorem 3.9]{KS97} to conclude that    $v$ is a $\varrho$-equivariant pluriharmonic map.\footnote{What is the definition of $\nu$ here? Is it some limit of $ u_{t_i} \circ F_{t_i} $, whose existence is ensured by \cite[Theorem 3.9]{KS97}, as $\varrho$ is proper?} 
 		By the uniqueness of the equivariant harmonic map into $\cC$ (cf.~\cite[Section 2.7]{BDDM}), we conclude $u_0=v$. Thus, $\lim\limits_{i\to\infty}d(u_{t_{i}} \circ F_{t_{i}}(x_i), u_0(x_i)) = 0$.  This contradiction proves $u_t \circ F_t$ converges uniformly to $u_0$.
 		%				
 		%\footnote{{\color{red}ADDED: I am a bit confused since in \cite{KS97}, the domain together with the Riemannian metric are fixed. Here our domain is fixed, but the metric $g_t$ varies. Do we need to modify the arguments in \cite[Theorem 3.9]{KS97} to prove the claim?} {\color{blue} The point is that the energy of $u_t  \circ F_t$ measured with respect to $g_0$ is close to the energy of $u_0$.  So we can apply Theorem 3.9 of Korevaar-Schoen.}} 
 	\end{proof}

 	We now give a proof of  \Cref{item:2}. 	%For an appropriate choice of $F_t$ above, we have  $u_t(z_1, \dots, z_n)=u(z_1, \dots, z_n,t)$ for $(z_1, \dots, z_n,t) \in \Omega$.  
 	Since $x_0 \in\tilde \cR(u_0) \subset \cR(u_0)$, there exists a admissible coordinate system $\Omega\simeq \bD^n\times\bD_\ep$ centered at  $x_0$   in \cref{rmk:omega}, and an apartment $A$ such that  $u_0(\bD^n) \subset A$.  This proves  \Cref{item:i}. 
 	Moreover,  we have ${\rm Ord}^{u_0}(x_0)=1$ and  the  image of $u_{0*}$ is a $m$-dimensional flat $\bF$ since $x_0 \in \tilde \cR(u_0)$ and rank$(u_0)=m$. Here $u_{0*}$ is defined just before \cref{item:splitting}.   Adopting the notation for {\sc Case 2} in the proof of \cref{preliminarycontinuity}, let $\bR^m \times B \subset T$ be the subbuilding defined as the  union of all apartments  containing $\bF$.   
 	
By \cref{convergence}, and following the argument in the proof of \cref{preliminarycontinuity}, we apply \cref{thm5.1} to conclude that there exist  $\tau_0>0$  and $\delta>0$ such that $u_t(\bD_\delta^n) \subset \bR^m \times B$ for $|t|<\tau_0$.  In other words, by rescaling the coordinates of $\Omega$ via   the transformation $({\bf z},t)\mapsto (\frac{\bf z}{\delta}, t)$ and replacing $\ep$ by  $\tau_0$,    we obtain $u(\bD^n\times\bD_{\ep}) \subset \bR^m \times B$. This  completes the proof of \Cref{item:ii}.
 	
 	Let us prove \Cref{item:iii}. Since  
 	$u_t(\bD^n) \subset \bR^m \times B$, we can write $u_t=(u_t^1, u_t^2)$.  
 	By shrinking $\Omega$ if necessary, we can assume that the rank of $du_0^1$ is equal to $m$ at all points of $\bD^n$ which then implies that $u_0^2$ is identically a constant.  This proves  \Cref{item:a}.  
 	\Cref{item:b} follows from  the same argument as in {\sc Case 2} of \cref{preliminarycontinuity}.

 	Next,
 	analogously to \eqref{unienergybd},  for any $t_0\in \bD_{\ep}$ we have 
 	\[
 	(1-Ct){^{g_{t_0}} E^{u_{t_0} \circ F_{t_0}}} \leq (1-Ct){^{g_{t_0}} E^{u_t \circ F_t}}  \leq {^{g_t} E^{u_t \circ F_t}} \leq {^{g_t} E^{u_{t_0} \circ F_{t_0}}} \leq (1+Ct) {^{g_{t_0}} E^{u_{t_0} \circ F_{t_0}}},
 	\]
 	which implies  the convergence of the total energy 
 	\begin{equation} \label{ttotnot}
 		\lim_{t \to t_0} {^{g_t} E^{u_t \circ F_t}}  =  {^{g_{t_0}} E^{u_{t_0} \circ F_{t_0}}}.
 	\end{equation}
 	Thus,  following the proof of \cref{convergence}, we can show that $u_t \circ F_t({\bf z})$,  varies uniformly in $t$ in the sense of \cref{preliminarycontinuity}.  Since $u_t^1$ varies uniformly in $t$  (see the sentence after \eqref{product} in proof of \cref{preliminarycontinuity}), we have that
 	\begin{equation} \label{portmanteau2}
 		\lim_{t \to t_0} \int_{\bB_r({\bf z})} |\nabla u^1_t|^2 d\vol_{g_t} = \int_{\bB_r({\bf z})} |\nabla u^1_{t_0}|^2 d\vol_{g_{t_0}}
 	\end{equation}
 	and 
 	$u_t^2({\bf z})$ varies uniformly in $t$ in the sense of \cref{preliminarycontinuity}. 
 	\begin{claim}\label{claim:portmanteau}
 		For any Euclidean ball $\bB_r({\bf z}) \subset \bD^n$,  we have
 		\begin{equation} \label{portmanteau3}
 			\lim_{t \to t_0} \int_{\bB_r({\bf z})} |\nabla u^2_t|^2 d\vol_{g_t} = \int_{\bB_r({\bf z})} |\nabla u^2_{t_0}|^2 d\vol_{g_{t_0}}. 
 		\end{equation}
 	\end{claim}
 	\begin{proof} 
 		Consider the probability measure  $\mu_t$  on $X_0$ defined by
 		$$
 		d\mu_t =\frac{|\nabla (u_{t_i} \circ F_t)|^2 d\vol_{g_t}}{^{g_t}E^{u_t \circ F_t}}.
 		$$
 		Combining the weak convergence of the energy density measure of $u_t \circ F_t$ to $u_{t_0} \circ F_{t_0}$ (which follows from~\cref{arzela}),  the convergence of  total energy (cf.~\eqref{ttotnot}) implies the weak convergence of $\mu_t$ to $\mu_{t_0}$.  By \cref{subsec:NPC}, $\mu_0$ is an absolutely continuous measure. Hence  $\partial \bB_r({\bf 0})$ is a continuity set for $\mu_0$, i.e.~$\mu_0(\partial \bB_r({\bf 0}))=0$.  By Portmanteau's theorem (cf. \cite[Theorem 2.1]{Bil99}), one has  
 		\[
 		\lim_{t \to t_0} \mu_t(\bB_r({\bf 0})) =\mu_{t_0}(\bB_r({\bf 0})).
 		\] 
 		Therefore, using the fact that  $\lim\limits_{t\to t_0}{^{g_t} E^{u_t \circ F_t}}  ={^{g_{t_0}} E^{u_{t_0} \circ F_{t_0}}}$, we conclude
 		\[
 		\lim_{t \to t_0} \int_{\bB_r({\bf z})} |\nabla u_t|^2 d\vol_{g_t} = \int_{\bB_r({\bf z})} |\nabla u_{t_0}|^2 d\vol_{g_{t_0}}. 
 		\] 
 		Since $|\nabla u_t|^2=|\nabla u^1_t|^2+|\nabla u^2_t|^2$,   \eqref{portmanteau3} follows from \eqref{portmanteau2}. The claim is proved. 
 	\end{proof}  
 	Recall that $u_t^2({\bf z})$ varies uniformly in $t$ in the sense of \cref{preliminarycontinuity}.    By  \cref{convergence,claim:portmanteau}, conditions in \cref{preliminarycontinuity} are satisfied. Hence 	  the map $$({\bf z},t)=(z_1, \dots, z_n,t)   \mapsto |\nabla u_t|^2({\bf z})$$ is a continuous function in $\bD^n \times \bD_\ep$. \Cref{item:c0}  follows  from  the fact that $|\nabla u_0^2({\bf z})|^2=0$ for any ${\bf z}\in \bD^n$.  
 	This completes the proof of \Cref{item:iii}, and hence the proof of \Cref{item:2}.  The theorem is proved. 
 \end{proof}

 \section{Lower semi-continuity of $\Gamma$-dimension: reductive case}
In this section, we prove \cref{main2} in the case where $\pi_1(X_0)$ is reductive. We apply  \cref{harmonic2,harmonic} in combination with techniques from the reductive Shafarevich conjecture \cite{Eys04,DY23}.
 \subsection{Deformation of canonical currents}
  Let $X$ be a compact K\"ahler manifold and 	let $\varrho:\pi_1(X)\to G(K)$ be a Zariski dense representation, where $G$ is a reductive  group over a non-archimedean local field $K$.    Let $u:\widetilde{X}\to \Delta(G)$ be a  $\varrho$-equivariant pluriharmonic map.
 
Note that any $x \in \cR(u)$ has an open neighborhood $\Omega_x \subset \widetilde X$ such that $u(\Omega_x)\subset A$ where $A$ is an apartment of $\Delta(G)$.    We fix an isometry $i_A:A\to \bR^N$. Write $i_A\circ u|_{\Omega_x}=(u_{1},\ldots,u_N)$.    The pluriharmonicity of $u$ implies that $\hess u_{i}=0$ for each $u_i$.  We consider a smooth  semi-positive $(1,1)$-form on $\Omega_x$ defined by 
 $$
 \sqrt{-1}\sum_{i=1}^{N}\partial u_{i}\wedge  \db u_{i}. 
 $$
 Note that such $(1,1)$-form  does not depend on the choice of the isometry $i_A$.
 By \cite[\S 3.3.2]{Eys04},  it is also invariant under $\pi_1(X)$-action. Therefore, it descends to a smooth real closed semi-positive $(1,1)$-form on $\cR(u)$.  
 It is shown in \cite[p. 540]{Eys04} that it extends to a  positive closed $(1,1)$-current $T_{\varrho}$  on $X$ with continuous potential.  
 \begin{dfn}[Canonical current]\label{def:canonical}
 	The above closed positive $(1,1)$-current 	$T_\varrho$ on $X$ is called the \emph{canonical current} of $\varrho$.   
 \end{dfn}  
% \begin{rem}
 	%  Note that it is not clear where the above $\varrho$-equivariant harmonic map $u$ is unique. However, by \cite[Theorem B]{DM24},  $T_\varrho$ does not depend on the choice of  $u$.  
% \end{rem}

 		Let $f:\sX\to \bD$ be the smooth projective family, and let  $\varrho:\pi_1(\sX)\to G(K)$ be a Zariski dense representation, where $G$ is a semi-simple algebraic group defined over a non-archimedean local field $K$. We apply \cref{harmonic}. Then there is  a    $\varrho$-equivariant continuous map   $u:\widetilde{\sX}\to\Delta(G)_K$ such that
  $u_t:=u|_{\widetilde{X}_t}:\widetilde{X}_t\to \Delta(G)_K$  is  a $\varrho$-equivariant pluriharmonic for each $t\in \bD$ satisfying the properties therein.  Precisely, pick any $x_0\in \tilde{\cR}(u_0)$. There exists an open neighborhood $\Omega$  of $x_0$ in $\widetilde{\sX}$ such that $\pi_{\sX}:\Omega\to \pi_{\sX}(\Omega)$ is an isomorphism, and  we have 
 $$
 u(\Omega)\subset \bR^m\times B.
 $$
 Here $B$ is also a Euclidean  building, and $\Omega\simeq \bD^n\times \bD_\ep$ such that $f\circ \pi_{\sX}(z_1,\ldots,z_n,t)=t$ within this coordinate system.   Write $u_t=u|_{\bD^n \times \{t\}} = (u_t^1, u_t^2)$ where 
 $$
 u_t^1: \bD^n \to \bR^m \mbox{ \ and \ } u_t^2: \bD^n \to B.
 $$  
 Then $u_0^2$ is constant. We abusively write $\Omega$ for $\pi_{\sX}(\Omega)$.   Consider the canonical current $T_{\varrho_t}$ on $X_t$ associated with $\varrho_t:\pi_1(X)\to G(K)$.   In this case, we have
 $$
 T_{\varrho_0}|_{\Omega\cap X_0}=\sqrt{-1}\sum_{i=1}^{m}\partial u_{0,i}^1\wedge\bar\partial u_{0,i}^1
 $$
 where we write $u_t^1:=(u_{t,1}^1,\ldots,u_{t,m}^1)$.  By \cref{harmonic},  
  $
 \sqrt{-1}\sum_{i=1}^{m}\partial u_{t,i}^1\wedge\bar\partial u_{t,i}^1
 $ 
 is a smooth semi-positive $(1,1)$-form on $\bD^n$, that varies continuously with    $t\in \bD_\ep$.  
 
 Fix any $t\in \bD_\ep$. Pick any $x\in \bD^n\cap \cR(u^2_t)$, where $ \cR(u^2_t)$ is the regular set of the pluriharmonic map $u^2_t$ defined in \cref{def:regular}. Then there exists a neighborhood $\Omega_2$ of $x$ in $\bD^n$ such that 
 $
 u^2_t(\Omega_2)\subset A_2,
 $ 
 where $A_2$ is an apartment of the building $B$. We fix an isometry $i_{A_2}:A_2\to \bR^{N-m}$.  Write $i_{A_2}\circ u_{t}^2=(u_{t,m+1}^2,\ldots,u_{t,N}^2)$. Then  by our construction of canonical currents, we have
 $$
 T_{\varrho_t}|_{\Omega_2}=\sqrt{-1}\sum_{i=1}^{m}\partial u_{t,i}^1\wedge\bar\partial u_{t,i}^1+\sn\sum_{j=m+1}^{N}\partial u_{t,j}^2\wedge\bar\partial u_{t,j}^2.
 $$  
 Since $x$ is an arbitrary point in $\bD^n\cap \cR(u^2_t)$, then  $$
 T_{\varrho_t}|_{\bD^n\cap \cR(u^2_t)}\geq  \sqrt{-1} \sum_{i=1}^{m}\partial u_{t,i}^1\wedge\bar\partial u_{t,i}^1.
 $$ Since $T_{\varrho_t}$ has continuous potential, and $\bD^n\backslash \cR(u^2_t)$ is a closed subset of $\bD^n$ of Hausdorff codimension at least two by  \cref{gsregularity},  it follows that the above inequality holds over the whole $\bD^n$.   In conclusion, we have the following result: 
 \begin{lem}\label{lem:continuity0}
  Let $f:\sX\to \bD$ be a smooth projective family and $\varrho:\pi_1(X_0)\to G(K)$ be a Zariski dense representation, where $G$ is a semi-simple algebraic group over a non-archimedean local field $K$.	There exists a full measure open subset $X_0^\circ \subset X_0$, such that   any $x_0\in X_0^\circ$ has an admissible coordinate system $\Omega\simeq \bD^n\times \bD_\ep$ in $\sX$ centered at $x_0$, with a real $(1,1)$-form  
 $
T(z,t)=\sn\sum_{i,j}a_{ij}(z,t)dz_i\wedge d\bar{z}_j
  $  on $\bD^n$ whose coefficients $a_{ij}(z,t)$ are continuous function, such that 
  \begin{thmlist} 
  	\item For each fixed $t\in \bD_\ep$, $T_{t}(z):=T(z,t)$ is a smooth semi-positive closed $(1,1)$-form on $\bD^n$.
  	\item For each $t\in \bD_\ep$, one has  $ T_{\varrho_t}|_{\Omega\cap X_t}\geq T_t$.
   	\item  $T_{\varrho_{0}}|_{\Omega\cap X_0}=T_0$.
  \end{thmlist} 
 \end{lem}
 \begin{proof}
 Within the notions above, 	we set $T(z,t):=\sqrt{-1} \sum_{i=1}^{m}\partial u_{t,i}^1\wedge\bar\partial u_{t,i}^1.$ The above arguments together with \cref{harmonic} yield the lemma. 
 \end{proof}
 When $G$ is a tori, we have a much simpler result. 
 \begin{lem} \label{lem:tori}
 	Let $f:\sX\to \bD$ be the smooth projective family. 	Let $\varrho:\pi_1(\sX)\to G(K)$ be a   representation, where $G$ is a tori  over    a non-archimedean local field $K$.  Then  the canonical current  $T_{\varrho_t}$ on $X_t$ induced by $\varrho_t:\pi_1(X_t)\to G(K)$  is a smooth  semi-positive $(1,1)$-form  on $X_t$ that varies smoothly  in $t$.
 \end{lem}
 \begin{proof}
 	The Bruhat-Tits building of $G$ is  a real Euclidean space $V:=\bR^N$ such that $G(K)$ acts on $V$ by translation (cf. \cite{KP23}). For the action of $G(K)$ on $V$, it induces a representation 
 	$
 	\tau:\pi_1(\sX)\to (V,+).
 	$  	Let the homomorphism ${\rm pr}_i:(\bR^N, +)\to (\bR, +)$ be  the projection into $i$-th factor, and let $\tau_i={\rm pr}_i\circ \tau$.   Then $\tau_i:\pi_1(\sX)\to (\bR,+)$ can be identified with an element $\lambda_i\in H^1(\sX,\bR)\simeq H^1(X_0,\bR)$ since  $H^1(\sX,\bR)\simeq {\rm Hom}(H_1(\sX,\bZ), \bR)$.  Let $g$ be a K\"ahler metric on $\sX$. We fix a smooth trivialization $F:\sX\to X_0\times \bD$, and let $F_t:X_t\to X_0$ be the induced diffeomorphism.  Then $F_t$ induces a smooth family of Riemannian metrics $g_t$ on $X_0$.   Therefore, there exists a family of smooth 1-forms $\xi_{i,t}$ on $X_0$, each is harmonic with respect to the metric $g_t$ and varying smoothly in $t$, such that their cohomology class $\{\xi_{i,t}\}={\lambda_i}$.  In this case,  $F_t^*\xi_{i,t}$ is a harmonic 1-form on $(X_t, g|_{X_t})$.    Let $\eta_{i,t} = (F_t^*\xi_{i,t})^{1,0}$ be the $(1,0)$-part of $\xi_{i,t}$, which is moreover a holomorphic 1-form on $X_t$ by the Hodge theory. Then $\eta_{i,t}$ varies smoothly with respect to $t$. In other words, for any $x_0\in X$ and any admissible coordinate system $\Omega\simeq \bD^n\times \bD_\ep$ centered at $x_0$ in $\sX$,  one has
 	$$\eta_{i,t}|_{\bD^n}=\sum_{j=1}^{n}a_j(z,t)dz_j,$$  such that $a_j(z,t)$  are smooth functions on $\bD^n\times \bD_\ep$.    By \cite{DM24}, the canonical current of $\varrho_t$ is given by
 	$$
 	T_{\varrho_t}|_{\Omega\cap X_t}=	\sn \sum_{i=1}^{n}\eta_{i,t}\wedge \overline{\eta_{i,t}}. 
 	$$
Consequently,  $
T_{\varrho_t} 
$ is a smooth closed semi-positive $(1,1)$-form on $X_t$ varying smoothly with $t$. 
 \end{proof}
\cref{lem:continuity0,lem:tori} imply the following stronger result than \cref{lem:continuity0}.
 \begin{thm}\label{lem:continuity}
 Let $f:\sX\to \bD$ be a smooth projective family and $\varrho:\pi_1(X_0)\to G(K)$ be a Zariski dense representation, where $G$ is a reductive algebraic group over a non-archimedean local field $K$.	There exists a full measure open subset $X_0^\circ\subset X_0$ such that  any $x_0\in X_0^\circ$ has an admissible coordinate system $\Omega\simeq \bD^n\times \bD_\ep$ in $\sX$ centered at $x_0$, with a real $(1,1)$-form  
 $
 T(z,t)=\sn\sum_{i,j}a_{ij}(z,t)dz_i\wedge d\bar{z}_j
 $  on $\bD^n$ satisfying 
 \begin{thmlist}
 	\item the coefficents $a_{ij}(z,t)$ are continuous function on $\bD^n\times\bD_{\ep}$.
 	\item For each fixed $t\in \bD_\ep$, $T_{t}(z):=T(z,t)$ is a smooth semi-positive closed $(1,1)$-form on $\bD^n$.
 	\item For each $t\in \bD_\ep$, one has  $ T_{\varrho_t}|_{\Omega\cap X_t}\geq T_t$.
 	\item  $T_{\varrho_{0}}|_{\Omega\cap X_0}=T_0$.
 \end{thmlist} 
 \end{thm}
 \begin{proof}
 	Consider the enlarged Bruhat-Tits building $\Delta(G)$. It is indeed the product of the Bruhat-Tits building of    $\Delta(\sD G)$ where $\sD G$ is the derived group of $G$, with a real Euclidean space $V:=\bR^N$ such that $G(K)$ acts on $V$ by translation (cf. \cite{KP23}).    Note that there is a natural action of $\sD G(K)$    on $\Delta(\sD G)$. The action of $G(K)$ on $\Delta(\sD G)$ is given by the composition of $G(K)\to \sD G(K)$ and the action of    of $\sD G(K)$    on $\Delta(\sD G)$. 
 
 We consider the representation $\sigma:\pi_1(X)\to \sD G(K)$  induced by $\varrho$, which is Zariski dense.  Let $T_{\sigma_t}$ be the canonical currents associated with $\sigma_t:\pi_1(X_t)\to \sD G(K)$. Then it satisfies the properties in \cref{lem:continuity}.
  
 On the other hand, for the action of $G(K)$ on $V$, it induces a representation 
 $
 \tau_t:\pi_1(X_t)\to (V,+) 
 $  for each $t\in \bD$. By \cref{lem:tori}, the canonical current  $T_{\tau_t}$  forms a family of smooth semi-positive closed $(1,1)$-forms  on $X_t$ that vary smoothly in $t$.  Observe that
 $
 T_{\varrho_t}=T_{\tau_t}+T_{\sigma_t} 
 $ is the canonical current associated with $\varrho_t$.  \cref{lem:continuity0} together with \cref{lem:tori}  establish the theorem.  
 \end{proof}

 \subsection{Proof of \cref{main:Betti}} 	\label{subsec:AC} 
 In this subsection, we prove \cref{main:Betti}, which is sufficient for establishing \cref{main4}. This will also imply \Cref{main2}  in the case where $\pi_1(X_0)$ is reductive.  
 \begin{dfn}[$m$-positive form]\label{def:m-positive}
 Let \( X \) be a complex manifold. A semi-positive \((1,1)\)-form \( \omega \) on $X$ is said to be \emph{\( m \)-positive} if  at each point of \( X \), \( \omega \) has at least \( m \) strictly positive eigenvalues. 
 \end{dfn}
 \begin{thm}\label{thm:reductive}
 	Let \( f: \mathscr{X} \to \mathbb{D} \) be a smooth projective family.   Let \( M_t := M_{\mathrm{B}}(\pi_1(X_t), \mathrm{GL}_N)(\mathbb{C}) \) denote the Betti moduli space of \(\pi_1(X_t)\). Define \( H_t \triangleleft \pi_1(X_t) \) as the intersection of the kernels of all reductive representations \( \varrho: \pi_1(X_t) \to \mathrm{GL}_N(\mathbb{C}) \). Then  the function	 \( t \mapsto \gamma d(X_t, H_t) \) on \( \mathbb{D} \)  is   \emph{lower semicontinuous}.
 \end{thm}
 \begin{proof} 
 We write $X$ for $X_0$.         Consider the Betti moduli space $M:=M_{\rm B}(\pi_1(X),\GL_N)(\bC)$.  By  \cite{Eys04},  \cite[Theorems 3.29]{DY23} and \cite[Proof of Theorem 4.31]{DY23} there exist  Zariski dense representations   $\{\tau_i:\pi_1(X)\to G_i(K_i)\}_{i=1,\ldots,k}$, where each $G_i\subset \GL_N$ is a reductive group over  a non-archimedean local field $K_i$ of characteristic zero, and $\bC$-VHS $\{\sigma_i:\pi_1(X)\to\GL_N(\bC)\}_{i=1,\ldots,\ell}$ such that   the Shafarevich morphism ${\rm sh}_M:X\to {\rm Sh}_M(X)$ of $M$ defined in \cref{def:Shafarevich} exists and satisfies  the following properties:
 \begin{enumerate} [label=(\roman*)]
 	\item    there exist closed positive $(1,1)$-currents    $S_1,\ldots,S_k$  and closed semi-positive $(1,1)$-forms $\omega_1,\ldots,\omega_\ell$ on ${\rm Sh}_M(X)$    such that   ${\rm sh}_M^*S_i=T_{\tau_{i}}$ and ${\rm Sh}_M^*\omega_j=\omega_{\sigma_{j}}$.  Here $T_{\tau_{i}}$  is the canonical current associated with  $\tau_{i}:\pi_1(X)\to G_i(K_i)$ defined in \cref{def:canonical} and $\omega_{\sigma_{j}}$ is the canonical form associated with $\sigma_{j}:\pi_1(X)\to \GL_N(\bC)$ defined in \cref{def:caninical form2}. 
 	\item   \label{kahler}
There is  a K\"ahler form $\omega_{{\rm Sh}_M(X)}$ on ${\rm Sh}_M(X)$   such that   	$$ \sum_{i=1}^{k}S_{i}+\sum_{j=1}^{\ell}\omega_j = \omega_{{\rm Sh}_M(X)}+\hess \psi,$$ 
where $\psi$ is a continuous function on ${\rm Sh}_M(X)$.
 	\item There exists a Zariski open subset of ${\rm Sh}_M(X)$ over which each $S_i$ is smooth.
 	\item    For each fiber $Z$ of ${\rm sh}_M$ and any reductive representation $\tau:\pi_1(X)\to \GL_N(\bC)$,  the image $\tau({\rm Im}[\pi_1(Z)\to \pi_1(X)])$  is a finite group, 
 \end{enumerate}    
  Since each $S_i$  has continuous potential, it follows that the positive closed $(1,1)$-current $(\sum_{i=1}^{k}S_i+\sum_{j=1}^{\ell}\omega_j)$ has \emph{minimal singularity} in the sense of Demailly, i.e.    it is less singular than any other positive current in its cohomology class  (cf. \cite[\S 1.4]{BEGZ} for the definition).  Let ${\rm Sh}_M (X)^\circ$ be the regular locus of ${\rm Sh}_M(X)$.  Set $m:= \dim {\rm Sh}_M(X)$.  By  \cite[Definition 1.17]{BEGZ} and \Cref{kahler}, the  \emph{non-pluripolar product} of $(\sum_{i=1}^{k}S_{i}+\sum_{j=1}^{\ell}\omega_j)$ is equal to
 $$
 \int_{{\rm Sh}_M(X)^\circ} \left(\sum_{i=1}^{k}S_{i}+\sum_{j=1}^{\ell}\omega_j\right)^{m}=\int_{{\rm Sh}_M(X)^\circ}(\omega_{{\rm Sh}_M(X)})^{m}>0.
 $$ 
 This implies that,  there exists an open subset $U$ of ${\rm Sh} _M(X)^\circ$ such that
 \begin{enumerate}[label=(\arabic*)] 
 	\item \label{item:11}${\rm sh}_M^{-1}(U)\to U$ is a proper holomorphic submersion.
 	\item \label{item:3}$\sum_{i=1}^{k}S_{i}+\sum_{j=1}^{\ell}\omega_j$ is smooth and strictly positive over $U$.
 \end{enumerate}  
 
 We define $\{\tau_{i,t}:\pi_1(X_t)\to G_i(K_i)\}_{i=1,\ldots,k}$ to be the representation as the compositions of  $\tau_i$ with the natural isomorphism $\pi_1(X_t)\to \pi_1(X)$. Similarly, the representations  $\{\sigma_{i,t}:\pi_1(X_t)\to\GL_N(\bC)\}_{i=1,\ldots,\ell}$ are defined as the compositions of   $\sigma_i$ and   $\pi_1(X_t)\to \pi_1(X)$. Note that $\sigma_{i,t}$ might not necessarily underlie a $\bC$-VHS for $t\neq 0$. Let $T_{\tau_{i,t}}$ be the canonical current on $X_t$ of $\tau_{i,t}$. We apply \Cref{lem:continuity}. Then there exists a point $x_0\in {\rm sh}_M^{-1}(U)$   such that   it has an admissible coordinate system $(\Omega\simeq \bD^n\times \bD_\ep;(z,t))$ in $\sX$ centered at $x_0$ together with a   continuous  $(1,1)$-form  
 $
 T_i(z,t)=\sn\sum_{i,j}a_{ij}(z,t)dz_i\wedge d\bar{z}_j
 $  on $\bD^n$  satisfying the following properties: 
 \begin{enumerate}[label=(\alph*)]
 	\item  \label{item:a1}$\Omega\cap X_0\subset {\rm sh}_M^{-1}(U)$.
 	\item  \label{item:a0}the coefficents $a_{ij}(z,t)$  of $T_i(z,t)$ are continuous function on  $\bD^n\times\bD_{\ep}$.
 	\item For each fixed $t\in \bD_\ep$, $T_{i,t}(z):=T_i(z,t)$ is a smooth semi-positive closed $(1,1)$-form on $\bD^n$.
 	\item \label{item:c1} For each $t$, we have $ T_{\tau_{i,t}}(z)|_{\Omega\cap X_t}\geq T_{i,t}(z)$.
 	\item \label{item:d1}  For $t=0$, we have $T_{\tau_{i,0}}(z)|_{\Omega\cap X_0}=T_{i,0}(z)$. 
 %	\item ${\rm sh}_M(z_1,\ldots,z_n)=(z_1,\ldots,z_{m})$.
 \end{enumerate}   
  Let $T_{\sigma_{i,t}}$ be the canonical form on $X_t$ associated with $\sigma_{i,t}$. By \cref{harmonic2},   the canonical forms $\omega_{\sigma_{i,t}}$ on $X_t$ vary continuously in $t$. Recall that  $$\sum_{i=1}^{k}T_{\tau_{i,0}}+\sum_{j=1}^{\ell}\omega_{\sigma_{j,0}}=\sum_{i=1}^{k}T_{\tau_{i}}+\sum_{j=1}^{\ell}\omega_{\sigma_{j}}={\rm sh}_M^{*}(\sum_{i=1}^{k}S_{i}+\sum_{j=1}^{\ell}\omega_j).$$ 
 Hence by \Cref{item:11,item:3}, \Cref{item:a1,item:a0,item:d1}, $\left(\sum_{i=1}^{k}T_{i,0}+\sum_{j=1}^{\ell}\omega_{\sigma_{j,0}}\right)|_{\Omega\cap X_0}$ is $m$-semipositive.   Note that the matrix continuous function has lower semicontinuous rank.  This implies that, after we shrink the neighborhood $\Omega$ of $x_0$ in $\sX$, there exists some  $\ep>0$ such that  for each $t\in \bD_\ep$,  we have  $X_t\cap \Omega\neq\varnothing$, and 
 $\left(\sum_{i=1}^{k}T_{i,t}+\sum_{j=1}^{\ell}\omega_{\sigma_{j,t}}\right)|_{\Omega\cap X_t}$  is  $m$-positive. 
   Therefore,  by \Cref{item:c1}, we conclude the following
   \begin{claim}\label{claim:below}
For each $t\in \bD_\ep$, the positive $(1,1)$-current    	 $\left(\sum_{i=1}^{k}T_{\tau_{i,t}}+\sum_{j=1}^{\ell}\omega_{\sigma_{j,t}}\right) |_{\Omega\cap X_t}$ is bounded from below by  the smooth closed $m$-positive $(1,1)$-form $\left(\sum_{i=1}^{k}T_{i,t}+\sum_{j=1}^{\ell}\omega_{\sigma_{j,t}}\right)|_{\Omega\cap X_t}$.
   \end{claim} 

Note that $\pi_1(X_t)\simeq \pi_1(X)$. Hence $M_t:=M_{\rm B}(\pi_1(X_t),\GL_N)(\bC)$ is naturally identified with $M$.   For each $t\in \bD_\ep$, we  consider the Shafarevich morphism ${\rm sh}_{M_t}:X_t\to {\rm Sh}_{M_t}(X_t)$,  whose existence is also proven in \cite{Eys04,DY23}. Let $Z$ be any smooth fiber  of ${\rm sh}_{M_t}$.  By \cref{def:Shafarevich},  $\tau_{i,t}({\rm Im}[\pi_1(Z)\to\pi_1(X_t)])$ is finite. By the unicity of $\tau_{i,t}$-equivariant harmonic maps in \cite[Theorem B]{DM24}, it follows that  $T_{\tau_{i,t}}|_{Z}$ is trivial.  By the unicity of harmonic bundles in \cite{Cor88},    $\omega_{\sigma_{j,t}}|_{Z}$ is also trivial.   Let ${\rm Sh}_{M_t}(X_t)^\circ$ be a Zariski dense open subset of the regular locus of  ${\rm Sh}_{M_t}(X_t)$ such that ${\rm sh}_{M_t}:X_t\to {\rm Sh}_{M_t}(X_t)$ is   smooth over ${\rm Sh}_{M_t}(X_t)^\circ$. Set $X_t^\circ={\rm sh}_{{M_t}}^{-1}({\rm Sh}_{M_t}(X_t)^\circ)$, and $m_t:=\dim {\rm Sh}_{M_t}(X_t)$. The above arguments implies   that
 $\left(\sum_{i=1}^{k}T_{\tau_{i,t}}+\sum_{j=1}^{\ell}\omega_{\sigma_{j,t}}\right)$ is \emph{at most} $m_t$-positive over $X_t^\circ$. 
 
 Note that  $X^\circ_t\cap \Omega_t$ is non-empty. By \Cref{claim:below} we conclude that $m_t\geq m$. It yields
 \begin{claim}\label{claim:last}
 	 For each $t\in \bD_\ep$,  we have  $\dim {\rm Sh}_{M_t}(X_t)\geq \dim {\rm Sh}_M(X_0)$.
 \end{claim}

 By \cref{rem:dim}, we note that the $\Gamma$-reduction $\gamma_{(X_t,H_t)}:X_t\dashrightarrow \Gamma_{H_t}(X_t)$ for $(X_t,H_t)$ is bimeromorphic to ${\rm sh}_{M_t}:X_t\to {\rm Sh}_{M_t}(X_t)$. Hence we have $\gamma d(X_t,H_t)=\dim {\rm Sh}_{M_t}(X_t)$ for each $t$.  \cref{main:Betti} then follows from \cref{claim:last}.  

Let us  prove \cref{main:linear} if $\pi_1(X_0)$ is reductive.   If there exists an almost faithful reductive representation $\varrho:\pi_1(X)\to \GL_N(\bC)$, then $\varrho_t:\pi_1(X_t)\to \GL_N(\bC)$ is also almost faithful and reductive for each $t\in \bD$. By \cref{lem:Shafarevich2}, ${\rm sh}_{M_t}:X_t\to {\rm Sh}_{M_t}(X_t)$  is the   Shafarevich morphism of $X_t$. From \cref{rem:dim},  one has
 $
\dim {\rm Sh}_{M_t}(X_t)=\gamma d(X_t). 
$  By \cref{claim:last}, this directly implies  the second claim of   \cref{main2}.  
%We prove \cref{thm:kollar}  if $\pi_1(X_0)$ is reductive.  If there exists a reductive and big representation $\varrho:\pi_1(X)\to \GL_N(\bC)$, then ${\rm sh}_M:X\to {\rm Sh}_M(X)$ is a birational morphism. Hence $\dim {\rm Sh}_M(X)=\dim X$ is maximal.  By \cref{claim:last},  $\dim {\rm Sh}_{M_t}(X_t)=\dim X$ for each $t\in \bD_\ep$. By \cref{def:Shafarevich}, one can see  that $\gamma d(X_t)\geq \dim {\rm Sh}_{M_t}(X_t)$. It follows that $\gamma d(X_t)=\dim X_t$. \cref{thm:kollar} is proved.  
\end{proof}
\section{Pseudo-Brody hyperbolicity in families}
 As an application of \cref{main:Betti}, we prove a result on  the openness of pseudo-Brody hyperbolicity. 
\begin{thm}[=\cref{main4}] \label{thm:open}
	Let $f:\sX\to \bD$ be a smooth projective family. Assume that there is a big and reductive representation $\varrho:\pi_1(X_0)\to \GL_N(\bC)$. If $X_0$ is pseudo-Brody   hyperbolic, then   $X_t$ is also pseudo-Brody hyperbolic   for sufficiently small $t$.
\end{thm}
\begin{proof}
	In  \cite[Theorem C]{CDY22},   we prove \cref{conj:GGL} if there exists a big and reductive representation $\pi_1(X)\to \GL_N(\bC)$.  Thus, by our assumption for $X_0$,  we apply \cite[Theorem C]{CDY22} to conclude that $X_0$ is of general type. By Siu's invariance of plurigenera \cite{Siu98,Pau07}, $X_t$ is of general type for any $t\in \bD$.
	
Set $M:=M_{\rm B}(\pi_1(X_0),\GL_N)(\bC)$. Since $[\varrho]\in M$, the Shafarevich morphism ${\rm sh}_M:X_0\to {\rm Sh}_M(X_0)$ is a birational morphism.  By \cite[Lemma 3.25]{DY23}, there exists another reductive  representation $\tau:\pi_1(X_0)\to \GL_N(\bC)$ such that  for any  reductive  representation $\sigma:\pi_1(X_0)\to \GL_N(\bC)$, we have 
\begin{align}\label{eq:inclusion}
	\ker\tau\subset\ker\sigma.
\end{align} Let $\tau_t:\pi_1(X_t)\to \GL_N(\bC)$  be the composite map of $\tau$ and the natural isomorphism $\pi_1(X_t)\to \pi_1(X_0)$ induced by $f$. 
\begin{claim}\label{claim:big}
	 There exists $\ep>0$ such that $\tau_t$ is a big  and reductive representation for any $t\in \bD_\ep$. 
\end{claim} 
\begin{proof}
By  \eqref{eq:inclusion}, we have
\begin{align}\label{eq:inclusion2}
	\ker\tau_t\subset\ker\sigma,\quad \forall\mbox{  reductive } \sigma:\pi_1(X_t)\to \GL_N(\bC). 
\end{align}
This implies that, for any closed subvariety $Z$ of $X_t$, if $\tau_t({\rm Im}[\pi_1(Z)\to \pi_1(X_t)])$ is finite, then $\sigma({\rm Im}[\pi_1(Z)\to \pi_1(X_t)])$ is finite  for any  reductive  representation $\sigma:\pi_1(X_0)\to \GL_N(\bC)$. 
Set $M_t:=M_{\rm B}(\pi_1(X_t),\GL_N)(\bC)$.  By \cref{def:Shafarevich}, for every $t\in \bD$,  
$
{\rm sh}_{M_t}:X_t\to {\rm Sh}_{M_t}(X_t)
$  coincides with the Shafarevich morphism ${\rm sh}_{\tau_t}:X_t\to  {\rm Sh}_{\tau_t}(X_t)$ of $\tau_t$.  
	
	By \cref{main:Betti}, there exists $\ep>0$ such that for any $t\in \bD_\ep$, we have $$\dim {\rm Sh}_{\tau_t}(X_t)=\dim {\rm Sh}_{M_t}(X_t)\geq \dim {\rm Sh}_M(X_0)=\dim X_0.$$  Therefore, $\tau_t$ is a big  and reductive representation for any $t\in \bD_\ep$.
\end{proof}  Since $X_t$ is of general type, by \cref{claim:big} we apply \cite[Theorem C]{CDY22} again to conclude that $X_t$ is pseudo-Brody hyperbolic for any $t\in \bD_\ep$. The theorem is proved. 
\end{proof}
From  \Cref{claim:big}, we obtain the following result, which may be of independent interest.
\begin{proposition}\label{prop:big}
	Let $f:\sX\to \bD$ be a smooth projective family. Assume that there is a big and reductive representation $\varrho:\pi_1(X_0)\to \GL_N(\bC)$. Then   $\tau:\pi_1(X_t)\to \GL_N(\bC)$ is big and reductive for sufficiently small $t$. \qed
\end{proposition}
Note that it is unclear  whether $\varrho_t:\pi_1(X_t)\to \GL_N(\bC)$,  induced by the big representation $\varrho$ in \cref{prop:big}, is big for sufficiently small  $t$. 

We conclude this section with the proof of \cref{corx}.
\begin{proof}[Proof of \cref{corx}] 
In Case \ref{item:VHS}, by the work of Griffiths-Schmid \cite{GS69}, the period domain $\sD$ is equipped with a natural metric that has negative holomorphic sectional curvature along the horizontal direction. Since the period map of $\varrho$ is assumed to be generically finite onto its image, it follows from the Ahlfors-Schwarz lemma that $X$ is pseudo-Brody hyperbolic.

On the other hand, by \cref{lem:Shafarevich}, $\varrho$ is big. By \cite{Sim92}, $\varrho$ is also reductive. Hence, the conditions in \cref{main4} are fulfilled, allowing us to conclude that a small deformation of $X$ is also pseudo-Brody hyperbolic.  

In Case \ref{item:semisimple},  by \cite[Theorem A]{CDY22}, $X$ is pseudo-Brody hyperbolic.   We apply \cref{main4} to conclude  that a small deformation of $X$ is also pseudo-Brody hyperbolic.   The corollary is proved.  
\end{proof}

\section{Lower semi-continuity of $\Gamma$-dimension: linear case} \label{sec:linear}
In this section, we will  prove  \cref{main:linear}.
\subsection{Regularity of  harmonic maps to symmetric spaces in families}  
\begin{lem}\label{lem:lift}
	Let $\Gamma$ be a finitely generated group. Consider the surjective isogeny $p:\SL_N(\bC)\times \bC^*\to \GL_N(\bC)$ defined by $p(A,t):=tA$. Then for any representation $\varrho:\Gamma\to \GL_N(\bC)$, there exists a finite index subgroup $\Gamma'$ of $\Gamma$ such that $\varrho|_{\Gamma'}$ lifts to a representation $\Gamma'\to \SL_N(\bC)\times \bC^*$.
\end{lem}
\begin{proof}
Set \( G := \varrho(\Gamma) \) and define \( G_1 := p^{-1}(G) \). Note that \( \ker p = \{ (\zeta_i I, \zeta_i^{-1}) \}_{i=1,\ldots,N} \), where \( \zeta_1, \ldots, \zeta_N \) are \( N \)-th roots of unity. In particular, these elements are all torsion elements. 

By Selberg's theorem, \( G_1 \) is virtually torsion-free, and thus there exists a finite index subgroup \( G_2 \) of \( G_1 \) such that \( G_2 \cap \ker p = \{ e \} \). Hence, \( p|_{G_2} : G_2 \to p(G_2) \) is an isomorphism. Note that \( p(G_2) \) is a finite index subgroup of \( G \). 

It follows that \( \Gamma' := \varrho^{-1}(p(G_2)) \) is a finite index subgroup of \( \Gamma \). We define \( \varrho' = (p|_{G_2})^{-1} \circ \varrho|_{\Gamma'} : \Gamma' \to G_2 \). Then we have 
\[
p_2 \circ \varrho' = \varrho|_{\Gamma'}.
\]
This completes the proof of the lemma.
\end{proof}
\begin{rem}
Since the $\Gamma$-dimension of a variety is invariant under étale covers, we can apply \cref{lem:lift} to reduce the proof of \cref{thm:kollar,main2} by replacing $\sX$ with a suitable finite étale cover. This allows us to assume that the representation $\pi_1(\sX) \to \GL_N(\bC)$ lifts to $\pi_1(\sX) \to \SL_N(\bC) \times \bC^*$. 
\end{rem} 
\begin{proposition}\label{prop:continuous}
Let $f:\sX\to\bD$ be a smooth projective family, and let  $\varrho:\pi_1(\sX)\to \SL_N(\bC)$  be a simple representation. There exists a  $\varrho$-equivariant  \emph{continuous} map
	 $
	u:\widetilde{\sX}\to \sS,
	$ 
	where $\sS:=\SL_N(\bC)/{\rm SU}_N$ is the  Riemannian symmetric space  of non-compact type  associated with $\SL_N(\bC)$ such that
	\begin{enumerate}[label=(\roman*)]
		 \item \label{item:existence} for each $t\in \bD$, $u|_{\widetilde{X}_t}:\widetilde{X}_t\to \sS$ is the unique $\varrho_t$-equivariant harmonic map.   
		 \item \label{item:derivativecontinuous}  the map $u$ vary smoothly in $t$.
	%	 \item \label{item:continuous metric}For the flat bundle $(\sV_\varrho,\nabla)$ on $\sX$ associated with $\varrho$, there is a continuous metric $h$ on $\sV_\varrho$ that vary smoothly in $t$, and   $(\sV_\varrho,\nabla,h)|_{X_t}$ is the harmonic bundle corresponding to $\varrho_t:\pi_1(X_t)\to \SL_N(\bC)$.   
	\end{enumerate} 
\end{proposition} 
	 \begin{proof}
We equip $\sX$ with a K\"ahler metric $g$.  We endow $X_t$ with the K\"ahler metric  \(g_t := g|_{X_t}\) for any $t\in \bD$, and $\sS$ the canonical Riemannian metric for Riemannian symmetric space.   By \cite{Cor88},    there exists a unique  $\varrho_t$-equivariant harmonic map $u_t:\widetilde{X}_t\to \sS$, which is moreover pluriharmonic. Let $u:\widetilde{X}\to \sS$ be the defined by setting $u|_{\widetilde X_t}= u_t$. This proves  \Cref{item:existence}.   
		\begin{claim}\label{claim:uniform}
			Let  $F_t:\widetilde X_0 \to \widetilde X_t$ be  the diffeomorphism induced by the $C^\infty$-trivialization $\sX\to X_0\times\bD$.   Then 	the function $u_t \circ F_t:\widetilde{X}_0\to \sS$ converges uniformly to $u_{0}$.		  In particular, $u:\widetilde{\sX}\to \sS$ is continuous. 
	 	\end{claim}
		\begin{proof}
The proof is close to that of \cref{convergence}. We abusively denote by \(g_t\) the induced Riemannian metric on \(X_0\) via the natural diffeomorphism \(X_0 \to X_t\). By the same arguments as in \eqref{unienergybd}, the energy function \(t \mapsto \{^{g_{t}}E^{u_t \circ F_t}\}\) is continuous. Fix any \(\ep \in (0,1)\). By a classical result of Eells-Sampson \cite{EL64}, \(\{u_t \circ F_t\}_{t \in \mathbb{D}_\ep }\) is uniformly Lipschitz continuous.  

By \eqref{unienergybd} again, the function \(t \mapsto \{^{g_{0}}E^{u_{t} \circ F_t}\}\) is also continuous and converges to \(\{^{g_{0}}E^{u_{0} }\}\) as \(t \to 0\). Hence, \(u_t \circ F_t : \widetilde{X_0} \to \mathscr{S}\) is an energy-minimizing sequence in the sense of  \cref{def:energymini} with \(X_0\) endowed with the fixed Riemannian metric \(g_0\).  

Since \(\varrho\) is semi-simple, the Zariski closure \(G\) of \(\varrho(\pi_1(\mathscr{X}))\) is a reductive group. Then \(\varrho(\pi_1(\mathscr{X}))\) does not fix a point at infinity of \(\mathscr{S}\). Otherwise, \(G\) would be contained in \(P(\mathbb{C})\), where \(P\) is a proper parabolic subgroup, which is impossible. By \cref{lem:proper}, the action of \(\varrho(\pi_1(\mathscr{X}))\) on \(\mathscr{S}\) is proper. Hence, we can apply \cref{arzela} to conclude that \(u_t \circ F_t\) converges uniformly to a \(\varrho_{0}\)-equivariant harmonic map \(v: \widetilde{X_0} \to \sS\) (with respect to the metric $g_0$) as \(t \to 0\). By the uniqueness theorem in \cite{Cor88}, one has \(v = u_0\).  

We thus conclude that \(u_t \circ F_t\) converges uniformly to \(u_0\) as \(t \to 0\). The same argument also shows that for any \(t_0 \in \mathbb{D}\), \(u_t \circ F_t\) converges uniformly to \(u_{t_0}\) as \(t \to t_0\). The claim is proved. 
		\end{proof} 
The exponential map of \(\sS\) provides a diffeomorphism between \(\bR^m\) and \(\sS\), thereby inducing a global coordinate system \((x_1, \ldots, x_m)\) for \(\sS\). Let \(\Gamma_{\alpha \beta}^\gamma \in C^\infty(\sS)\) denote the Christoffel symbols of the Levi-Civita connection on \(\sS\) with respect to this global coordinate system. 

We express \(u = (u^1, \ldots, u^m)\) in terms of these global coordinates on \(\sS\). Let \(\Omega \simeq \bD^n \times \bD_{\epsilon}\) be any admissible coordinate as described in \cref{rmk:omega}. We use the notation  
\[
u_t^\alpha(z_1, \dots, z_n) := u^\alpha(z_1, \dots, z_n, t), \quad \forall \alpha = 1, \dots, m.
\]  
Since each \(u_t\) is harmonic, over \(\Omega\) we have  
\begin{align}\label{eq:harmonic}
	 	\Delta_t u_t^\alpha +\Gamma_{  \beta\gamma}^\alpha(u_t)	g_t^{i\bar{j}} \frac{\partial  u_t^\beta}{\partial z_i}\frac{\partial u_t^\gamma}{\partial \bar{z}_j}=0.
\end{align} 
Here, \(\Delta_t\) denotes the Laplacian operator with respect to the metric \(g_t\), and \([g_t^{i\bar{j}}]_{i,j=1,\ldots,n}\) is the inverse   of the $n\times n$-matrix \([\langle \frac{\partial}{\partial z_i}, \frac{\partial}{\partial \bar{z}_j} \rangle_{g_t}]_{i,j=1,\ldots,n}\). 

Since the metrics \(g_t\) vary smoothly with \(t\), and \(\{^{g_t}E^{u_t}\}_{t \in \bD_\ep }\) is uniformly bounded, it follows from \cite{EL64} that there exists a constant \(C_1 > 0\) such that 
\begin{align}\label{eq:elliptic}
	\left|\frac{\partial u_t^\alpha}{\partial \bar{z}_i}({\bf z})\right| \leq C_1, \quad \left|\frac{\partial u_t^\alpha}{\partial z_i}({\bf z})\right| \leq C_1,
\end{align}
for each \(\alpha = 1, \ldots, m\), \(i = 1, \ldots, n\), \(t \in \bD_\ep \), and \({\bf z} \in \bD^n\).

Since \(u\) is continuous, there exists a constant \(C_2 > 0\) such that \(|\Gamma_{\beta\gamma}^\alpha(u_t({\bf z}))| \leq C_2\) for any \(\alpha, \beta, \gamma \in \{1, \ldots, m\}\), \(t \in \bD_\epsilon\), and \({\bf z} \in \bD^n\). Fix any \(\mu \in (0, 1)\). 

We apply the interior elliptic estimates for the elliptic equations \eqref{eq:harmonic}  (cf. \cite[Section 1.7, Lemma 3]{Sim96}) to conclude that there exists a  constant $C_3>0$ such that   for each $t\in \bD_\ep$ and any $\alpha \in \{1,\ldots,m\}$,
\begin{align}\label{eq:Holder'}
 \left| u_t^\alpha\right|_{C^{1,\mu};\bD^n_{\frac{3}{4}}}:=\sup_{{\bf z}\in \bD_{\frac{3}{4}}}|u_t^\alpha({\bf z})|+\sum_{j= 1}^{2n}  \sup_{{\bf z}\in \bD_{\frac{3}{4}}}\left| \frac{\partial u_t^\alpha}{\partial x_j}({\bf z}) \right| +\sum_{j= 1}^{2n} \sup_{{\bf z}_1, {\bf z}_2 \in  \bD_{\frac{3}{4}}^n, \ {\bf z}_1 \neq{\bf z}_2} \frac{\left| \frac{\partial u_t^\alpha}{\partial x_j}({\bf z}_1)- \frac{\partial u_t^\alpha}{\partial x_j}({\bf z}_2)\right|^\mu }{|{\bf z}_1-{\bf z}_2|}\leq C_3
\end{align} 
where $(z_1=x_1+ix_2,  \dots, z_n=x_{2n-1}+ix_{2n})$ are the coordinates in $\bD^n$.
Let  $\mu'\in (\frac{\mu}{2},\mu)$ and any $t_0\in \bD_{\ep}$.	Note that  the H\"older space $C^{1,\mu}(\bD^n_{\frac{3}{4}})$ is compactly contained in $C^{1,\mu'}(\bD^n_{\frac{3}{4}})$.  Then for any sequence $t_i\to t_0$, after taking a  subsequence if necessary,  $u_{t_i}^\alpha$ converges to some $v^\alpha$ in $C^{1,\mu'}(\bD^n_{\frac{3}{4}})$. Note that $ u_{t}^\alpha$ converges uniformly to $ u_{t_0}^\alpha$ as $t\to t_0$. It follows that $v^\alpha=u_{t_0}^\alpha$.   Hence $u_{t_i}^\alpha$ converges to   $u_{t_0}^\alpha$ in $C^{1,\mu'}(\bD^n_{\frac{3}{4}})$.  By \eqref{eq:harmonic} together with a bootstrap argument,  for any positive integer $k>0$, $u_{t_i}^\alpha$ converges to   $u_{t_0}^\alpha$ in $C^{k,\frac{\mu}{2}}(\bD^n_{\frac{1}{2}})$ as $i\rightarrow \infty$.
This proves \Cref{item:derivativecontinuous}.  The proposition is proved. 

%By applying the interior elliptic estimates for the elliptic equations \eqref{eq:elliptic} (cf. \cite[Corollary 10.3.3]{Liv21}), we conclude that there exists a constant \(C_3 > 0\) such that, for each \(t \in \bD_\epsilon\) and any \(\alpha \in \{1, \ldots, m\}\), the Hölder norm satisfies
%\begin{align}\label{eq:Holder}
%	\lVert u_t^\alpha \rVert_{C^{2, \mu}; \bD^n} := \sup_{{\bf z} \in \bD^n}  | u_t^\alpha({\bf z}) |_{C^{2, \mu}} \leq C_3.
%\end{align}   

%Pick some \(\mu' \in (0, \mu)\) and any \(t_0 \in \bD_{\epsilon}\). Note that the Hölder space \(C^{2, \mu}(\overline{\bD^n})\) is compactly contained in \(C^{2, \mu'}(\overline{\bD^n})\) (cf. \cite[Lemma 6.36]{GT01}). Then, for any sequence \(t_i \to t_0\), after passing to a subsequence, by \eqref{eq:Holder},  \(u_{t_i}^\alpha\) converges to some \(v^\alpha\) in \(C^{2, \mu'}(\overline{\bD^n})\) for each $\alpha$. By \cref{claim:uniform},  \(u_t^\alpha\) converges uniformly to \(u_{t_0}^\alpha\) as \(t \to t_0\). It follows that \(v^\alpha = u_{t_0}^\alpha\). Hence, \(u_{t_i}^\alpha\) converges to \(u_{t_0}^\alpha\) in \(C^{2, \mu'}(\overline{\bD^n})\).

%By \eqref{eq:harmonic} together with the elliptic bootstrap argument, we conclude that for any positive integer \(k > 0\), \(u_{t_i}^\alpha\) converges to \(u_{t_0}^\alpha\) in \(C^{k, \mu'}(\overline{\bD^n})\) as \(i \to \infty\).

%This proves \Cref{item:derivativecontinuous}. The proposition is proved.

	\end{proof} 
	\begin{rem}
		Note that  results similar to \cref{prop:continuous} were obtained in  \cite{EL80,Sle22} under an additional assumption that   $\Gamma:=\varrho(\pi_1(\sX))$  is a \emph{uniform and torsion free lattice} in $\SL_N(\bC)$, i.e.,  $\faktor{\sS}{\Gamma}$ is a compact Riemannian locally symmetric manifold. 
	\end{rem}
\begin{lem}\label{lem:continuous2} 	Let  $\tau:\pi_1(\sX)\to  \bC^*$  be a representation.   Consider the natural action $\bC^*$ on $\bR$ by $z.t:=t+\log |z|$.  Then there exists a $\tau$-equivariant harmonic map $u:\widetilde{\sX}\to \bR$, which vary smoothly in $t$.  \end{lem}
\begin{proof}
	Note that there exists an isomorphism $\bR\times {\rm U}(1)\to \bC^*$ defined by $(t,z)\mapsto \exp(t).z$. Denote by $\tau_1:\pi_1(\sX)\to \bR$ and $\tau_2:\pi_1(\sX)\to {\rm U}(1)$ be the projections of $\tau$ according to the above isomorphism. Then $\tau_1$ corresponds to a cohomology class $\lambda\in H^1(\sX,\bR)$.   Fix a K\"ahler metric $g$ on $\sX$. By the classical Hodge theory,   there exists   1-forms $\xi_{t}$ on $X_t$, each is harmonic with respect to the metric $g_t$ and varying smoothly in $t$, such that their cohomology class $\{\xi_{t}\}=\iota_t^*{\lambda}$, where $\iota_t:X_t\to \sX$ denotes the inclusion map. 
	
	We take a smooth section $s:\bD\to \widetilde{\sX}$ of $\widetilde{\sX}\to\sD$. Then for each $t\in \bD$, we define a smooth map $u_t(x):=\int_{s(t)}^{x}\pi_{X_t}^*\xi_t$. It follows that the union $u:\widetilde{\sX}\to \bR$ of $u_t$ is a continuous map,  which vary smoothly in $t$.    The lemma is proved. 
\end{proof}
\begin{thm}\label{thm:family}
	Let $f:\sX\to\bD$ be a smooth projective family, and let $\sigma:\pi_1(\sX)\to \GL_N(\bC)$ be a semi-simple representation. Then after we replace $\sX$ by a suitable finite \'etale cover,   for  the flat bundle $(\sV_\sigma, D)$  on $\sX$ induced by $\sigma$,  there exists a continuous metric $h$ on $\sV_\sigma$ satisfying the following properties:
\begin{thmlist}
	\item the restriction $(\sV_\sigma,D,h)|_{X_t}$ is a harmonic bundle for each $t\in \bD$.
	\item  \label{item:smooth}  Writing $V_t:=\sV_\sigma|_{X_t}$, the metric $h|_{V_t}$ varies smoothly in $t$.  
\end{thmlist} 
\end{thm}
\begin{proof}
 We might decompose $\sigma$ into a direct sum of simple representations, and it suffices to prove the theorem for the case when $\sigma$  is simple. By \cref{lem:lift},  after we replace $\sX$ by a suitable finite \'etale cover,  we may assume that $\sigma$ lifts to a simple representation $\sigma':\pi_1(\sX)\to \SL_N(\bC)\times\bC^*$.   Let $\varrho: \pi_1(\sX) \to \SL_N(\bC)$ and $\tau: \pi_1(\sX) \to \bC^*$ denote the compositions of $\sigma'$ with the first and second projections, respectively, i.e., $\SL_N(\bC) \times \bC^* \to \SL_N(\bC)$ and $\SL_N(\bC) \times \bC^* \to \bC^*$.  
 
 By \cref{prop:continuous,lem:continuous2}, there exist a continuous $\varrho$-equivariant map $u_1:\widetilde{\sX}\to \sS$ and continuous $\tau$-equivariant map $u_2:\widetilde{\sX}\to \bR$ such that $u_{1,t}=u_1|_{\widetilde{X_t}}:\widetilde{X_t}\to \sS$ is pluriharmonic and $u_{2,t}=u_2|_{\widetilde{X_t}}:\widetilde{X_t}\to \bR$ is also pluriharmonic.  
 
 	 Let $\sP_N(\bC)$ be the space of positive definite hermitian symmetric $N \times N$ matrices with Riemannian metric defined by 
 $$\langle X,Y\rangle_H={\rm tr}(H^{-1}X H^{-1}Y), \ \ \ \forall X,Y=T_H\sP_N(\bC).$$
 Note that there is a natural isometric identification
 $\sP_N(\bC) \simeq \GL_N(\bC)/{\rm U}_N$.  Indeed, $\GL_N(\bC)$ acts isometrically and transitively on $\sP_N(\bC)$ where the action of $g \in \GL_N(\bC)$ is given by  $H \mapsto gHg^\dagger$.  Since  ${\rm U}_N$ is    the stabilizer of the identity, 
 the map $\GL_N(\bC)/{\rm U}_N\to \sP_N(\bC)$ defined by $[g] \mapsto gg^\dagger$  is an isometry. The subspace $\sP_N(\bC)_1$ consisting of $H \in \sP_N(\bC)$ with ${\rm det}\, H=1$ is a totally geodesic submanifold of $\sP_N(\bC)$.  Moreover,   $\sP_N(\bC)_1$ is  invariant by the action of $\SL_N(\bC)$ with ${\rm SU}_N$ being the stabilizer of the identity.  Thus, we also have the isometric identification    $\sP_N(\bC)_1 \simeq \sS$.
 
 Observe that
 \begin{align*}
 \psi:\sP_N(\bC)_1\times \bR&\to \sP_N(\bC) \\
 	(H,t)&\mapsto e^{\frac{t}{N}}H.
 \end{align*}
 is an isometry.  
 We define the action     of $\SL_N(\bC)\times\bC^*$  on $\sP_N(\bC)_1\times \bR$ by setting $(g,z).(H,t):=(gAg^\dagger,t+2N\log |z|)$. Then  
 $\psi$ is equivariant, i.e.
 for any $(g,z)\in \SL_N(\bC)\times\bC^*$,  we have
\begin{align}\label{eq:equivariant}
 \psi((g,z).(H,t))= \psi ((gHg^\dagger,t+2N\log |z|))=gHg^\dagger e^{\frac{t}{N}} |z|^{2}=(zg).e^{\frac{t}{N}}H
\end{align}  Identifying $u_1(x)$ as an element in $\sP_N(\bC)_1$,  we have that $ u(x):=\psi(u_1(x),2Nu_2(x) )=e^{2u_2(x)}u_1(x) \in \sP_N(\bC) \simeq \GL_N(\bC)/{\rm U}_N$.  Since $u_1|_{\widetilde{X}_t}$ and $u_2|_{\widetilde{X}_t}$ are pluriharmonic and $\psi$ is an isometry,   $u|_{\widetilde{X}_t}$ is also pluriharmonic.   Note that  $u$ is  $\sigma$-equivariant  by \eqref{eq:equivariant}.  By \cref{prop:continuous,lem:continuous2}, $u|_{\widetilde{X}_t}$  varies smoothly in $t$.

 Consider the flat bundle $\sV_\sigma:=\widetilde{\sX}\times_\sigma\bC^N$ induced by $\sigma$. Then  for any $x\in \sX$, select any  $\tilde{x}\in \pi_{\sX}^{-1}(x)$ and and define the Hermitian metric  $h$ on $\sV_\sigma $  by setting $h_x:=u(x)$.     Since $u$ is $\sigma$-equivariant, this definition provides a well-defined Hermitian metric $h$ on $\sV_\sigma$.  Moreover,   $h|_{X_t}$ is the harmonic metric for $V_t:=\sV_\sigma|_{X_t}$. Since $u$ varies smoothly in $t$, the restriction  
 $h|_{V_t}$ also varies smoothly in $t$. Thus, the theorem is proved. 
\end{proof}
\begin{rem}\label{rem:smooth}
We can use \cref{thm:family} to strengthen \cref{harmonic2}, showing that \(\omega_{\varrho_t}\) varies \emph{smoothly} in  \(t\). Indeed, let \((V_t, D_t, h_t) := (\sV_\sigma, D, h)|_{X_t}\). Decompose \(D_t\) as \(D_t = \nabla_t + \Phi_t\), where \(\nabla_t\) is the metric connection associated with \((V_t, h_t)\), and \(\Phi_t\) is self-adjoint. By \cref{item:smooth}, \(\Phi_t\) varies smoothly with \(t\).

Next, decompose \(\Phi_t\) as \(\Phi_t = \Phi_t' + \Phi_t''\), where \(\Phi_t'\) and \(\Phi_t''\) are the \((1,0)\)- and \((0,1)\)-components of \(\Phi_t\), respectively. Since both \(\Phi_t'\) and \(\Phi_t''\) vary smoothly with \(t\), it follows that \(\omega_{\varrho_t} = \sn {\rm tr}(\Phi_t' \wedge \Phi_t'')\) also varies  {smoothly} with \(t\). 
\end{rem} 
\subsection{Semi-canonical form}
Let $(V,D,h)$ be a harmonic bundle on a compact K\"ahler $n$-fold $X$.  We  define the Laplacian by   setting
\begin{align}\label{eq:Laplacian}
	 \Delta:=D D^*+D^* D, 
\end{align}
where $D^*$ is the adjoint of $D$ with respect to $h$. Any $k$-form $\eta\in A^k(X,V)$ with value in $V$ is called harmonic if $\Delta(\eta)=0$. The cohomology $H_{\rm DR}^\bullet(X,V)$ is defined as  the cohomology of the following complex
$$
0\to A^0(X,V)\xrightarrow{D}A^1(X,V) \xrightarrow{D}\cdots\xrightarrow{D}A^{2n}(X,V).
$$
In \cite[Lemma 2.2]{Sim92},  it is shown that the natural inclusion
\begin{align}\label{eq:iso}
	 \sH^k(X,V)\to H_{\rm DR}^k(X,V)
\end{align}  is an isomorphism, 
where  $
\sH^k(X,V)$ denotes the space of $V$-valued harmonic $k$-forms.

We assume further that $V$ underlies a real structure, i.e. it is a complexification of a real local system $V_\bR$.  Let $\eta\in \sH^1(X,V)$ be a harmonic 1-form. We decompose $\eta=\eta^{1,0}+\eta^{0,1}$ into $(1,0)$ and $(0,1)$-parts, i.e. $\eta^{1,0}\in A^{1,0}(X,V)$ and $\eta^{0,1}\in A^{0,1}(X,V)$. 
\begin{lem}
	 If $\alpha:=\{\eta\}\in H_{\rm DR}^1(X,V_{\bR})$, then  we have $\overline{\eta^{1,0}}=\eta^{0,1}$.    Here $\overline{\eta^{1,0}}$ is the complex conjugate of $\eta^{1,0}$ with respect to the real structure by $V_\bR$.
\end{lem}
\begin{proof}
Let $\sigma:V\to V$ be the complex conjugate   of $V$ with respect to the $\bR$-structure by $V_\bR$.  Then $\sigma D=D\sigma$.  It follows that  $D(\overline{\eta})=0$.  By the definition of $H_{\rm DR}^\bullet(X,V_\bR)$,  there exists a  $D$-closed $V_\bR$-valued real 1-form $\beta$ such that $\{\beta\}=\alpha$.   It follows that
 $$
\{\overline{\eta}\}=\{\overline{\beta}\}=\{\beta\}=\alpha.
$$ 

We decompose $D=\d+\db+\theta+\theta^*$ such that  $(\d+\db,h)$ is a metric connection, and  $\theta+\theta^*$ is self-adjoint.  Then  by \cite[Lemma 2.12]{Sim92}, $\sigma D'=D''$, where $D':=\d+\theta^*$ and $D'':=\db+\theta$. Since $\eta$ is harmonic, by \cite[Lemma 2.2]{Sim92} one has
$$
D''\eta=D'\eta=0.
$$
This implies that 
$$
D''(\overline{\eta})=D'(\overline{\eta})=0.
$$
By \cite[Lemma 2.2]{Sim92} again, $\overline{\eta}\in \sH^1(X,V)$. It follows from \eqref{eq:iso} that 
$\overline{\eta}=\eta$. This implies the lemma. 
\end{proof}
\begin{dfn}[Semi-canonical form]\label{def:canonical2}
Let $(V,D,h)$ be a harmonic bundle on a compact K\"ahler manifold $X$.  For any $\alpha\in H_{\rm DR}^1(X,V_\bR)$,  let $\eta\in \sH^1(X,V)$ be the harmonic representative in $\alpha$ as above. The  \emph{semi-canonical form} associated with $\alpha$ is a smooth real $(1,1)$-form on $X$ defined by   
\begin{align}\label{eq:canonical2}
	S_{\alpha}:=\sn   \{\eta^{1,0},  \eta^{1,0}\}_h,
\end{align}
where $\{\bullet,\bullet\}_h$ denotes the operation of taking the wedge product of differential forms combined with the inner product of sections of 
$V$ induced by the metric  $h$.
\end{dfn}  
Note, however, that   $S_{\alpha}$ 
is not necessarily closed or semi-positive. Moreover, it depends on the rescaling of the harmonic metric  $h$ and is therefore referred to as \emph{semi-canonical}.
\begin{lem}\label{lem:semi-positive}
	Let $(V,D,h)$ be a harmonic bundle on a compact K\"ahler manifold $X$ such that the corresponding local system is defined over $\bR$, denoted by $ V_\bR$.  Let $g:Z\to X$ be  a holomorphic map from another compact K\"ahler manifold $Z$.    If $g^*V_\bR$ is a trivial local system, then for any $\alpha\in H_{\rm DR}^1(X,V_\bR)$,  the $(1,1)$-form	$g^*S_{\alpha}$ is \emph{closed} and  \emph{semi-positive}.  
\end{lem}
\begin{proof}  
	Let $\underline{\bR^r}:=Z\times \bR^r$ be the trivial real vector bundle of rank $r:=\rank V$ on $Z$. Let $h_0$ be the standard  metric on $\underline{\bR^r}$, and let $D_0$ be the canonical connection on $\underline{\bR^r}$.  Note that the pullback $g^*(V,D,h)$ is still a harmonic bundle over $Z$.  By the unicity of harmonic metrics in \cite{Cor88}, $g^*V_\bR$ is isomorphic to   $(\underline{\bR^r},D_0,h_0)$, up to rescaling  $h_0$
	by a constant factor.  Note that   for harmonic 1-forms, we have $g^*\sH^1(X,V)\subset \sH^1(Z,g^*V)$ (that is not true for general harmonic $k$-forms!). Since $g^*V_\bR$ is a trivial local system, it follows that $g^*\eta\simeq (\xi_1,\ldots,\xi_r)$, where $\xi_i\in \sH^1(Z,\bR)$. Hence, we have $\eta^{1,0}\simeq (\xi^{1,0}_1,\ldots,\xi^{1,0}_r)$, which implies that
	$$
	g^*\sn \{\eta^{1,0}, \eta^{1,0}\}_h= \sn \{g^*\eta^{1,0}, g^*\eta^{1,0}\}_h=\sn \sum_{i=1}^{r}\xi^{1,0}_i\wedge \xi_i^{0,1}.
	$$
	The classical Hodge theory implies that $\xi_{i}^{1,0}$ is a holomorphic 1-form on $X$ that is $d$-closed. Hence $	g^*\sn \{\eta^{1,0}, \eta^{1,0}\}_h$ is semi-positive and closed. 
\end{proof}

\subsection{Semi-canonical forms in families}
 \begin{lem}\label{lem:harmonic}
 	Let $f:\sX\to \sD$ be a smooth projective family and let $\varrho:\pi_1(\sX)\to \GL_N(\bC)$ be a semi-simple representation. We apply \cref{lem:lift} to replace $\sX$ by a finite \'etale cover such that the conditions for $\varrho$ in \cref{thm:family} are fulfilled. Let $(\sV,D)$ be the flat bundle on $\sX$ induced by $\varrho$, and let $h$ be the metric for $\sV$ constructed in \cref{thm:family}. Fix some $\alpha\in H^1_{\rm DR}(X_0,V_0)$.   Let $\alpha_t\in H_{\rm DR}^1(X_t,V_t)$ be the image of $\alpha$ under the natural isomorphism $H^1_{\rm DR}(X_0,V_0)\to H^1_{\rm DR}(X_t,V_t)$. Then  for the $V_t$-valued harmonic $1$-form $\eta_t\in \sH^1(X_t,V_t)$ such that $\{\eta_t\}=\alpha_t$,  we have  that $F_{t,*}(\eta_t)\in A^1(X_0,V_0)$ varies smoothly in $t$, where $F_t:X_t\to X_0$ is the diffeomorphism induced by  the $C^\infty$-trivialization $\sX\stackrel{C^\infty}{\to} X_0\times \bD$. In particular, if $\varrho$ has image in $\GL_N(\bR)$, the semi-canonical forms $S_{\alpha_t}$ vary smoothly in $t$.
 \end{lem}
 
\begin{proof}
By \cref{thm:family}, $(\sV,D,h)|_{X_t}$ is a harmonic bundle corresponding to the representation $\varrho_t$. 	Consider the family of Laplacian operators $\Delta_t$ for $(V_t,D_t,h_t)$ defined in \eqref{eq:Laplacian}. By \cref{thm:family}, $\Delta_t$ is a smooth family of formally
	self-adjoint   elliptic   operators. Since the function $t\mapsto \dim H_{\rm DR}^1(X_t,V_t)$ is constant, we apply \cite[Theorem 7.4]{Kod05} to conclude that $\ker \Delta_t= \sH^1(X_t,V_t)$  varies  smoothly. It follows that $\eta_t$, thus $\eta_t^{1,0}$,  varies  smoothly in $t$. By \cref{item:smooth}, 
	 $
	S_{\alpha_t}:=\sn \{\eta_t^{1,0},\eta_t^{1,0} \}_{h_t}
	$ 
	varies smoothly in $t$.  The lemma is proved. 
\end{proof} 
\subsection{1-step $\bR$-VMHS}
We begin by recalling the definition of an $\bR$-VMHS of weight length 1. For further properties of this type of $\bR$-VMHS, we refer readers to \cite{Car87,Pea00}.
\begin{dfn}[$\bR$-VMHS]
	A $\bR$-VMHS $\cM$  of weight length $1$ on   a complex manifold $X$ consists of 
	\begin{thmlist}
		\item  A real local system $\cM_X$ of finite type;
		\item  An increasing filtration $\{0\}=\cW_{-2}\subset \cW_{-1}\subset \cW_0=\cM_X$ of  by real sub-local systems;
		\item  A finite decreasing filtration $\cF^\bullet  $ by locally free   subsheaves of $\cM:=\cM_X \otimes_\bR \mathcal{O}_X$  which satisfies the  Griffiths  transversality condition:
		$$
		\nabla\left(\mathcal{F}^p \right) \subset \Omega_X^1 \otimes_{\mathcal{O}_X} \mathcal{F}^{p-1},
		$$
		and  such that $\mathcal{W}$ and $\mathcal{F}$ define an  $\bR$-MHS on each fiber $\left(\cM(x), \cW_\bullet(x), \cF^\bullet(x)\right)$ at $x$ of the vector bundle $\mathcal{M}$.  Here $\nabla$ is the flat connection on $\cM$ inherited from the local system $\cM_X$. 
	\end{thmlist} 
\end{dfn} 	
In what follows, we will denote   $\cL_i:=\cW_i/\cW_{i-1}$. Let us denote by $\sM$ the mixed period domain of $\cM$, and $\sD_i$ the period domain for $\sL_i$. Then  the $\bR$-VMHS $\cM$ induces an $\bR$-VHS on  $\cL_i$  of weight $i$ for $i=-1,0$.  Note that there exists a  natural affine bundle  $\sM\to \sD_{-1}\times\sD_0$.  Consider the $\bR$-VHS $\cL:=\cL_0^*\otimes \cL_{-1}$ on $X$, which is of weight $-1$. 
\begin{lem}\label{lem:differential}
	Let $\cM$ be an $\bR$-VMHS   of weight length $1$ on a smooth projective variety $X$. Assume that the connection on $\cL_{-1} \oplus \cL_0$ induced by $\cM$ is given by  
	\[
	D := \begin{bmatrix}
		D_{-1} & \eta \\
		0 & D_0
	\end{bmatrix},
	\]
	where $\eta \in \sH^1(X, \cL)$ and $\eta(x) \in \Omega_{X,x}^{1,0} \otimes \cL^{-1,0} \oplus \Omega_{X,x}^{0,1} \otimes \cL^{0,-1}$, with $\cL^{p,q}$ denoting the Hodge $(p,q)$-subspace of $\cL$. 
	
	Let $f: Z \to X$ be a morphism from another smooth projective variety.  
 Let $Y$ be a connected component of $Z \times_X \widetilde{X}$, and let $g: Y \to \widetilde{X}$ be the natural holomorphic map. Let $p: \widetilde{X} \to \sM$ be the mixed period map, and let $\pi_Z: Y \to Z$ be the covering map. If $f^* \cL_i$ is trivial for $i = -1, 0$, then 
	\[
	d(p \circ g) = \pi_Z^* f^* \eta^{1,0}.
	\]
	For the semi-canonical form $S_\eta$ associated with $\eta$,  $ \pi_Z^* f^*  S_\eta$ is strictly positive at the points of $Y$ where $p \circ g$ is immersive. 
\end{lem}

\begin{proof}
Since $f^* \cL_i$ is trivial for $i = -1, 0$, the composite $p\circ g(Y)\subset V(H_{-1},H_{0})$,  where $V(H_{-1},H_{0})$ is the fiber of the natural fibration $\sM\to \sD_{-1}\times\sD_0$  at some  point  $H_0\oplus H_{-1} \in \sD_{-1}\times\sD_0$.  Here $H_i$  is a real Hodge structure of  weight $i$. 

Fix a base point $y_0\in Y$.  Then for any $y\in Y$,  the parallel transport of $\pi_Z^*f^*\nabla$ along any smooth path  $\gamma$  connecting $y_0$ and $y$ induces a $\bR$-linear isomorphism of $H_{-1}\oplus H_0$ given by 
	$$
	\begin{bmatrix}
		1 & \int_\gamma  \pi_Z^*f^*\eta\\
		0& 1
	\end{bmatrix}.
	$$
	One can check that this does not depend on the choice of $\gamma$.  
	
	Let $\eta=\eta^{1,0}+\eta^{0,1}$ be the decomposition of $\eta$ with respect to $\Omega_{X}^{1,0}\otimes \cL^{-1,0}\oplus \Omega_{X}^{0,1}\otimes \cL^{0,-1}$. Then $\overline{\eta^{1,0}}=\eta^{0,1}$.  We have
	$$
	p\circ g(y)-p\circ g(y_0)=\left[ \int_{\gamma}  \pi_Z^*f^*\eta\right]=\left[ \int_{\gamma}  \pi_Z^*f^*\eta^{1,0}\right],
	$$
	where $\left[\int_{\gamma}  \pi_Z^*f^*\eta^{1,0}\right]$ denotes the image of $ \int_{\gamma}  \pi_Z^*f^*\eta^{1,0}$ under the quotient map $$\Hom(H_0,H_{-1})_\bC\to \Hom(H_0,H_{-1})_\bC/F^0\Hom(H_0,H_{-1})_\bC\simeq \oplus_{p\leq -1}\Hom(H_0,H_{-1})_\bC^{p,-1-p}.$$ 
	This implies that  
\begin{align}\label{eq:period}
	 dp\circ g= \pi_Z^*f^*\eta^{1,0}. 
\end{align}

Since $f^*\cL$ is a trivial local system,  by the same proof as in \cref{lem:semi-positive},  $f^*\eta\simeq (\xi_1,\ldots,\xi_r)$, where each $\xi_i$ is a real harmonic 1-form on $Z$. It follows that  
	$$
	f^*S_\eta=\sn \sum_{i=1}^{r}\xi^{1,0}\wedge\xi^{0,1}.
	$$ 
If $p\circ g$ is immersive at   some $y\in Y$, then by \eqref{eq:period}, we have $\pi_Z^*f^*\eta^{1,0}(v)\neq 0$ for any  non-zero $v\in T_yY$.  This is equivalent to    that $f^*S_\eta$ is strictly positive at $\pi_Z(y)$. The second assertion follows. The lemma is proved. 
\end{proof}
 \subsection{Proof of \cref{thm:kollar,main:linear}} 	
 We will apply the tools from the linear Shafarevich conjecture \cite{EKPR12}, combined with the strategy outlined in \cref{subsec:AC}, to prove \cref{main:linear}.
\begin{proposition}\label{thm:eys}
	Let $X$ be a smooth projective variety. Define \( H_{N} \triangleleft \pi_1(X) \) as the intersection of the kernels of all \emph{linear} representations \(  \pi_1(X) \to \mathrm{GL}_N(\mathbb{C}) \).  Then after replacing $X$ by a finite \'etale cover,   Shafarevich morphism ${\rm sh}_{H_N}:X\to {\rm Sh}_{H_N}(X)$ exists, and  there exist 
	\begin{enumerate}[label=(\alph*)]
		\item a family of  Zariski dense representations  $ \{\tau_i:\pi_1(X)\to G(K_i)\}_{i=1,\ldots,\ell}$ where each $G_i\subset \GL_N$ is a reductive group over  a  non-archimedean local field  $K_i$ of characteristic zero;
		\item  a reductive representation $\varrho_1:\pi_1(X)\to \GL_{N_1}(\bC)$ underlying a $\bC$-VHS $\cL_1$, that is a direct sum of representations $\pi_1(X)\to \GL_N(\bC)$.
		\item  A real $\bR$-VMHS $\cM$ on $X$, that is an extension of a $\bR$-VHS  $\cL_{0}$  by another  $\bR$-VHS  $\cL_{-1}$, with the extension class denoted by $\alpha\in H_{\rm DR}^1(X,\cL_0^*\otimes\cL_{-1})$.  The weight of $\cL_i$ is $i$ for $i=-1,0$.
		\item Set $\cL:=\cL_{0}^*\otimes\cL_{-1}$.  Let $\eta \in \sH^1(X, \cL)$ be the harmonic representative in $\alpha$. Then $\eta(x) \in \Omega_{X,x}^{1,0} \otimes \cL^{-1,0} \oplus \Omega_{X,x}^{0,1} \otimes \cL^{0,-1}$, with $\cL^{p,q}$ denoting the Hodge $(p,q)$-subspace of $\cL$. 
		\item A $(1,1)$-current $T$ on ${\rm Sh}_{H_N}(X)$   such that $T|_U$ is a strictly positive smooth $(1,1)$-form for some open subset $U$ of ${\rm Sh}_{H_N}(X)^{\rm reg}$,
	\end{enumerate} 
	such that for some $\ep>0$,
	the sum
	$
	\sum_{i=1}^{\ell}T_{\tau_i}+\omega_{\varrho_1}+\ep S_{\alpha}={\rm sh}_{H_N}^*T.
	$ 
	Here $T_{\tau_i}$ is the canonical current of $\tau_i$, $\omega_{\varrho_1}$ is the canonical form of $\varrho_1$, and $S_\alpha$ is the semi-canonical form of $\alpha$.  
\end{proposition}
\begin{proof}
	Let $M:=M_{\rm B}(\pi_1(X),\GL_N)(\bC)$ be the character variety of $\pi_1(X)$.	By  \cite[Proof of Theorem 3.29]{DY23}, there exist 
	\begin{enumerate}[label=(\alph*)]
		\item  a family of  Zariski dense representations  $ \{\tau_i:\pi_1(X)\to G(K_i)\}_{i=1,\ldots,\ell}$ where each $G_i\subset \GL_N$ is a reductive group over  a  non-archimedean local field  $K_i$ of characteristic zero;
		\item  a reductive representation $\varrho_1:\pi_1(X)\to \GL_{N_1}(\bC)$ underlying a $\bC$-VHS $\cL_1$,   that is a direct sum of representations $\pi_1(X)\to \GL_N(\bC)$;
	\end{enumerate} 
	such that the following properties hold: we   define $ \widetilde{H_M^0}$   to be the intersection of the kernels of all semisimple representations $\pi_1(X)\to \GL_N(\bC)$. Denote by $\widetilde{X_M^0}:=\widetilde{X}/\widetilde{H_M^0}$ and $\pi_0:\widetilde{X_M^0}\to X$ the covering map. Then   the period map of $\cL_1$ descends to $\phi:\widetilde{X_M^0}\to \sD_1$, where $\sD_1$ is the period map of the $\bR$-VHS $\cL_1$.     For the holomorphic map 
	\begin{align*}
		\Phi_0:	 \widetilde{X_M^0}&\to \prod_{i=1}^{\ell}\Sigma_{\tau_i}\times \sD_1\\
		x&\mapsto (s_{\tau_1}\circ \pi_0,\ldots,s_{\tau_\ell}\circ \pi_0,\phi(x)),
	\end{align*} 
	each connected component of the fiber of $\Phi_0$ is compact. Here $s_{\tau_i}:X\to \Sigma_{\tau_i}$ is the Katzarkov-Eyssidieux reduction for $\tau_i$ (cf. \cite{DY23} for the definition).    By \cite[Proof of Theorem 3.29]{DY23}, $\Phi_0$ factors through 
	$$
	\widetilde{X_M^0}\stackrel{r_M^0}{\to} \widetilde{S_M^0}(X)\stackrel{r_0}{\to} \prod_{i=1}^{\ell}\Sigma_{\tau_i}\times \sD_1 
	$$
	where $r_M^0$ is  a proper surjective holomorphic fibration  and $r_0$ is holomorphic map with each fiber being a discrete set. Moreover, $\widetilde{S_M^0}(X)$ does not contain compact subvarieties.  Therefore,  the Galois group  ${\rm Aut}(\widetilde{X_M^0}/X)$ induces an action on $\widetilde{S_M^0}(X)$ which is properly discontinuous, and such that $r_M^0$ is equivariant with respect to the action by ${\rm Aut}(\widetilde{X_M^0}/X)$.  By \cite[Lemma 3.28]{DY23}, replacing $X$ by a finite \'etale cover, we can assume that such an action on $\widetilde{S_M^0}(X)$  is free. Taking the quotient of $r_M^0$ by ${\rm Aut}(\widetilde{X_M^0}/X)$, we obtain
	\begin{equation*}
		\begin{tikzcd}
			\widetilde{X_M^0}  \arrow[r,"\pi_0"] \arrow[d,"r_M^0"] & X\arrow[d,"{\rm sh}_M^0"]\\
			\widetilde{S_M^0}(X)\arrow[r,"q_0"]\arrow[d,"r_0"] &{\rm Sh}_M^0(X)\\
			\prod_{i=1}^{\ell}\Sigma_{\tau_i}\times \sD_1 
		\end{tikzcd}
	\end{equation*}
	Here ${\rm Sh}_M^0(X)$ is called the \emph{reductive Shafarevich morphism} associated with $M$. 
	
 By \cite[\S 5.2]{EKPR12}, there is  an $\bR$-VMHS $\cM$ of weight length 1 with the mixed period domain $\sM$ (cf. \cite[Lemma 5.4]{EKPR12}) and an infinite Galois \'etale cover $\pi_1:\widetilde{X_M^1}\to X$  (cf. \cite[p. 1575]{EKPR12} for the definition) factorizing through $\pi_0:\widetilde{X_M^0}\to X$ such that   
	\begin{enumerate}[label=(\alph*)]   
		\item the mixed period domain descends to $\psi:\widetilde{X_M^1}\to \sM$;   
		\item for the holomorphic map 
		\begin{align*}
			\Phi_1:	 \widetilde{X_M^1}&\to \prod_{i=1}^{\ell}\Sigma_{\tau_i}\times \sD_1\times \sM\\ 
			x&\mapsto (s_{\tau_1}\circ \pi_1,\ldots,s_{\tau_\ell}\circ \pi_1,\phi(x),\psi(x)),
		\end{align*} 
		each connected component of the fiber of $\Phi$ is compact.  
	\end{enumerate} 
	Here we abusively use $\phi:\widetilde{X_M^1}\to \sD_1$ to denote by the composite of $\phi:\widetilde{X_M^0}\to \sD_1$ with  $\widetilde{X_M^1}\to \widetilde{X_M^0}$. 
	By \cite[p. 1576]{EKPR12}, $\Phi_1$ factors through 
	$$
	\widetilde{X_M^1}\stackrel{r_M^1}{\to} \widetilde{S_M^1}(X)\stackrel{r_1}{\to} \prod_{i=1}^{\ell}\Sigma_{\tau_i}\times \sD_1\times \sM
	$$
	where $r_M^1$ is  a proper surjective holomorphic fibration  and $r_1$ is holomorphic map with each fiber being a discrete set. 
	
	By \cite[Lemma 5.7]{EKPR12}, $\widetilde{S_M^1}(X)$  does not contain compact subvarieties. Therefore,  the Galois group  ${\rm Aut}(\widetilde{X_M^1}/X)$ induces an action on $\widetilde{S_M^1}(X)$ which is properly discontinuous, and such that $r_M^1$ is equivariant with respect to the action by ${\rm Aut}(\widetilde{X_M^1}/X)$.    Replacing $X$ by a finite \'etale cover, we assume that such an action is free. Taking the quotient of $r_M^1$ by ${\rm Aut}(\widetilde{X_M^1}/X)$, we obtain:
	\begin{equation}\label{eq:bigdia}
		\begin{tikzcd}
			\widetilde{X_M^1}\arrow[dd, bend right=45, "\Phi_1"'] \arrow[r,"\pi_1"] \arrow[d,"r_M^1"] & X\arrow[d,"{\rm sh}_M^1"] &X\arrow[d,"{\rm sh}_M^0"]\arrow[l,equal]& \widetilde{X_M^0} \arrow[d,"r_M^0"]\arrow[l,"\pi_0"']\arrow[dd, bend left=45, "\Phi_0"]\\
			\widetilde{S_M^1}(X)\arrow[r,"q_1","{\tiny \mbox{\'etale}}"']\arrow[d,"r_1"] &{\rm Sh}_M^1(X)\arrow[r,"g"]&{\rm Sh}_M^0(X) & \widetilde{S_M^0}(X)\arrow[d]\arrow[l,"q_0"',"{\tiny \mbox{\'etale}}"]\arrow[d,"r_0"]\\
			\prod_{i=1}^{\ell}\Sigma_{\tau_i}\times \sD_1\times \sM&&&\prod_{i=1}^{\ell}\Sigma_{\tau_i}\times \sD_1
		\end{tikzcd}
	\end{equation} 
	By \cite[p. 1549 \& Proposition 3.10]{EKPR12}, ${\rm Sh}_M^1$ is the Shafarevich morphism associated with $(X,H_N)$.  We apply Selberg's theorem to replace $X$ by a further finite \'etale cover such that the monodromy representation of $\cM$ has torsion free image.  By the construction of ${\rm sh}_M^1$, for each fiber $Z$ of ${\rm sh}_M^1$, $\cM|_{Z}$ has finite, thus trivial monodromy. Then $\cM$ descends to a $\bR$-VMHS of weight length $1$ on ${\rm Sh}_M^1$, which we denote by $\cM'$.

	Let $Z$ be an irreducible component of any   fiber of $g:{\rm Sh}_M^1(X)\to {\rm Sh}_M^0(X)$.    Let $\widetilde{Z_M^1}$ be any connected component of the inverse image $q_1^{-1}(Z)$. Let $\sD_{-1}\times \sD_0$ be the graded period domain of $\sM$. Note that the natural map $\sM\to \sD_{-1}\times \sD_0$ is a holomorphic vector bundle (cf. \cite{Car87}). Then there exists some $P\in \sD_{-1}\times \sD_0$ such that for the fiber $V$ of $\sM\to \sD_{-1}\times \sD_0$ at  $P\in \sD_{-1}\times \sD_0$, we have  $\psi|_{\widetilde{Z_M^1}}:\widetilde{Z _M^1}\to V$. Moreover,  $q:=\psi|_{\widetilde{Z_M^1}}$ factors as
	\begin{equation*}
		\begin{tikzcd}
			\widetilde{Z_M^1}^{\rm sn}\arrow[dr]\arrow[r,"p"]\arrow[d] & V'\arrow[d]\\
			\widetilde{Z_M^1}\arrow[r,"q"] &V
		\end{tikzcd}
	\end{equation*}
	where $	\widetilde{Z_M^1}^{\rm sn}$ is the semi-normalization of $\widetilde{Z_M^1}$,  $V'\to V$ is a linear injective map from a complex vector  space $V'$ and $p$ is a proper holomorphic map (cf. \cite[p. 1576]{EKPR12}).  Since $r_1$ has discrete fibers, and the image of $\widetilde{Z_M^1}$ under the composite map $$\widetilde{S_M^1}(X)\stackrel{r_1}{\to} \prod_{i=1}^{\ell}\Sigma_{\tau_i}\times \sD_1\times \sM\to \prod_{i=1}^{\ell}\Sigma_{\tau_i}\times \sD_1$$ is constant, it follows that $p$ is a finite map. Here $\prod_{i=1}^{\ell}\Sigma_{\tau_i}\times \sD_1\times \sM\to \prod_{i=1}^{\ell}\Sigma_{\tau_i}\times \sD_1$ is the natural projection map. 
	
	By \cite[Example 2.2.2]{EKPR12}, the $\bR$-VMHS $\cM$ satisfies the conditions in \cref{lem:differential}. Let $\alpha\in H_{\rm DR}^1(X,\cL_0^*\otimes\cL_{-1})$  be the extension class defining $\cM$, where $\cL_{-1}\oplus\cL_0$ is the graded $\bR$-VHS of $\cM$.  Let $S_\alpha$ be the semi-canonical form  associated with $\alpha$ defined in \eqref{eq:canonical2}. Since $\cM=({\rm sh}_M^1)^*\cM'$, there exists a  real   $(1,1)$-form   $S_\alpha'$ on ${\rm Sh}_M^1(X)$ such that
	$
	S_\alpha=({\rm sh}_M^1)^*S_\alpha'.
	$   Then by \cref{lem:differential} 	$S_{\alpha}'|_{Z}$ is closed and semi-positive and is strictly positive over the locus where $	q:\widetilde{Z_M^1}\to V$ is immersive.  %Since $p$ is finite, by  Demailly-P\u{a}un's criterion \cite{DP04,DHP22},  $\{S_{\alpha}'|_{Z}\}$  is a K\"ahler class.  
	
	By \cref{subsec:AC},   there exists a semi-positive closed  $(1,1)$-current $T_0$ with continuous potential on ${\rm Sh}_M^0(X)$  such that  
	\begin{enumerate}[label=(\arabic*)]
		\item One has $$ T_{\tau_1}+\cdots+T_{\tau_\ell}+\omega_{\varrho_1}= ({\rm sh}_M^0)^*T_0$$
		\item $T_0|_\Omega$ is smooth and strictly positive for some open subset $\Omega\subset {\rm Sh}_M^0(X)^{\rm reg}$.  
	\end{enumerate} 
	Here $T_{\tau_i}$ is the canonical current for $\tau_i$ defined in \cref{def:canonical}, and $\omega_\cL$ is  the canonical form associated with the representation induced by the $\bC$-VHS $\cL$.
	
	Let $Z$ be a general  fiber of $g:{\rm Sh}_M^1(X)\to {\rm Sh}_M^0(X)$ over  $\Omega$.    Then $S_\alpha'|_{Z}$ is semi-positive and strictly positive at general points of $Z$. This implies that, for a general point $z\in Z$, it has a neighborhood $U\subset {\rm Sh}_M^1(X)^{\rm reg}$ such that $(g^*T_0+\ep S_\alpha')|_U$ is  a  strictly positive smooth $(1,1)$-form (but not necessarily $d$-closed). Let $T:=g^*T_0+\ep S_\alpha'$, which is the desired $(1,1)$-current in the proposition.  The proof is accomplished. 
\end{proof} 
\begin{thm}\label{thm:linear}
	Let \( f: \mathscr{X} \to \mathbb{D} \) be a smooth projective family.  
	\begin{thmlist}
	\item \label{mainC}	 If there exists an almost faithful linear representation $\varrho:\pi_1(X)\to \GL_N(\bC)$,  then   \( t \mapsto \gamma d(X_t) \) is  a lower semicontinuous function on \( \mathbb{D} \). 
	\item \label{mainA} If there exists a big representation $\sigma:\pi_1(X_0)\to \GL_N(\bC)$, then $X_t$ has big fundamental group for sufficiently small $t$. 
	\end{thmlist}
\end{thm}
\begin{proof}
	We write $X$ for $X_0$. 
	We now apply  \cref{thm:eys} and use the notations therein without recalling it.   Then  after replacing $\sX$ by some suitable finite \'etale cover,  there exist 
	\begin{enumerate}[label=(\alph*)]
		\item  a family of  Zariski dense representations  $ \{\tau_i:\pi_1(X)\to G_i(K_i)\}_{i=1,\ldots,\ell}$ where each $G_i\subset \GL_N$ is a reductive group over  a  non-archimedean local field  $K_i$ of characteristic zero,
		\item  a reductive representation $\varrho_1:\pi_1(X)\to \GL_{N_1}(\bC)$ underlying a $\bC$-VHS $\cL_1$, that is a direct sum of representations $\pi_1(X)\to \GL_N(\bC)$.
	\item  \label{item:VMHS}A real $\bR$-VMHS $\cM$ on $X$ with torsion free monodromy, that is an extension of a $\bR$-VHS  $\cL_{0}$  by another  $\bR$-VHS  $\cL_{-1}$, with the extension class denoted by $\alpha\in H_{\rm DR}^1(X,\cL_0^*\otimes\cL_{-1})$.  The weight of $\cL_i$ is $i$ for $i=-1,0$.   Let $r_N$  be the rank of $\cM$. 
	\item Set $\cL:=\cL_{0}^*\otimes\cL_{-1}$.  Let $\eta \in \sH^1(X, \cL)$ be the harmonic representative in $\alpha$. Then $\eta(x) \in \Omega_{X,x}^{1,0} \otimes \cL^{-1,0} \oplus \Omega_{X,x}^{0,1} \otimes \cL^{0,-1}$ for any $x\in X$, with $\cL^{p,q}$ denoting the Hodge $(p,q)$-subspace of $\cL$. 
	\item \label{item:e}A $(1,1)$-current $T$ on ${\rm Sh}_{H_N}(X)$   such that $T|_U$ is a strictly positive smooth $(1,1)$-form for some open subset $U$ of ${\rm Sh}_{H_N}(X)^{\rm reg}$, and  for some $\ep>0$,
	one has 
	$
	\sum_{i=1}^{\ell}T_{\tau_i}+\omega_{\varrho_1}+\ep S_{\alpha}={\rm sh}_{H_N}^*T.
	$  
	\end{enumerate} 
	Here $T_{\tau_i}$ is the canonical current of $\tau_i$, $\omega_{\varrho_1}$ is the canonical form of $\varrho_1$, and $S_\alpha$ is the semi-canonical form of $\alpha$.   
	
Let \( F_t: X_t \to X \) be the natural diffeomorphism induced by the \( C^\infty \) trivialization \( \sX \to X \times \bD \). Define \( \cL_{i,t} := F_t^*(\cL_i) \) as the local systems on \( X_t \) for \( i = -1, 0, 1 \). Let \( \alpha_t \) denote the image of \( \alpha \) under the natural isomorphism  
\[
F_t^*:H_{\rm DR}^1(X, \cL_0^* \otimes \cL_{-1}) \to H_{\rm DR}^1(X_t, \cL_{0,t}^* \otimes \cL_{-1,t}).
\]  
Note that \( \alpha_t \) is the extension class defining \( \cM_t := F_t^*\cM \), which is the extension of \( \cL_{0,t} \) by \( \cL_{-1,t} \). 

 We apply \cref{lem:lift,thm:family} to replace $\sX$ by a finite \'etale cover such that  for the semi-simple local system $\sV$ on $\sX$ induced by $\cL_0^*\otimes\cL_{-1}$,   there is a metric $h$   for $\sV$  such that $h|_{\sV|_{X_t}}$ varies smoothly in $t$.  Note that $\sV|_{X_t}=\cL_{0,t}^* \otimes \cL_{-1,t}$. By \cref{lem:harmonic}, the semi-canonical form $S_{\alpha_t}$ associated with $\alpha_t$ vary smoothly in $t$.

	Set $m:=\dim {\rm Sh}_{H_N}(X)$.  By \Cref{item:e}, there exists an open subset $U_0$ of $X$ such that  $(\sum_{i=1}^{\ell}T_{\tau_i}+\omega_{\varrho_1}+\ep S_{\alpha})|_{U_0}$ is $m$-positive  in the sense of \cref{def:m-positive}. By \cref{claim:below},   there exists an open subset $\Omega$ of $\sX$ such that for some $\delta>0$, one has $\Omega_t:=\Omega\cap X_t\neq \varnothing$ for $t\in \bD_{\delta}$, and  
	$$ (\sum_{i=1}^{\ell}T_{\tau_{i,t}}+\omega_{\varrho_{1,t}}+\ep  S_{\alpha,t})|_{\Omega_t}$$
	is at least $m$-positive.    Here $T_{\tau_{i,t}}$ is the canonical current of $\tau_{i,t}:\pi_1(X_t)\to G_i(K_i)$, $\omega_{\varrho_{1,t}}$ is the canonical form of $\varrho_{1,t}:\pi_1(X_t)\to \GL_{N_1}(\bC)$. 
	
  Denote by $H_{N,t}$ the intersection of kernels of all linear representation $\pi_1(X_t)\to \GL_N(\bC)$.   	Consider the Shafarevich morphism ${\rm sh}_{H_{N,t}}:X_t\to {\rm Sh}_{H_{N,t}}(X_t)$, whose existence is ensured by \cref{thm:eys}. By the arguments below \cref{claim:below},  for any smooth  fiber $Z$ of   ${\rm sh}_{H_{N,t}}$,  $(\sum_{i=1}^{\ell}T_{\tau_{i,t}}+\omega_{\varrho_{1,t}})|_Z$ vanishes identically.
	
	\medspace

\noindent {\bf Proof of \cref{mainC}}: 	Since there exists an almost faithful linear representation $\varrho:\pi_1(X)\to \GL_N(\bC)$, it follows that $\varrho_t:\pi_1(X_t)\to \GL_N(\bC)$ is also almost faithful.  Then  $H_{N,t}$ is finite, which implies that ${\rm sh}_{H_{N,t}}:X_t\to {\rm Sh}_{H_{N,t}}(X_t)$  is the Shafarevich morphism of $X_t$  for each $t\in \bD$.  Therefore, for each fiber $Z$ of  ${\rm sh}_{H_{N,t}} $,  ${\rm Im}[\pi_1(Z)\to \pi_1(X_t)]$ is finite.    It follows that $\cM_t|_Z$ has finite, thus trivial monodromy as $\cM$ has torsion free monodromy. 

	Recall that $\alpha_t\in H^1(X_t,\cL_{0,t}^*\otimes \cL_{-1,t})$ is the extension class  defining $\cM_t:=F_t^*\cM$, that is the extension of $\cL_{0,t}$ by $\cL_{-1,t}$. Thus,   $\alpha_t|_Z$ is also trivial.    Hence, by \eqref{eq:canonical2}, $S_{\alpha_t}|_Z$ is trivial for each  smooth fiber $Z$ of  ${\rm sh}_{H_{N,t}} $.  
	
	Fix any $t\in \bD_\delta$.  We choose  an open subset $U_t$ of ${\rm Sh} _{H_{N,t}}(X_t)^{\rm reg}$ such that
	\begin{enumerate}[label=(\arabic*)] 
		\item ${\rm sh}_{H_{N,t}}^{-1}(U_t)\to U_t$ is a proper holomorphic submersion.
		\item  $W:=\Omega_t\cap {\rm sh}_{H_{N,t}}^{-1}(U_t)$ is non-empty. 
	\end{enumerate}  
	It follows that,  $(\sum_{i=1}^{\ell}T_{\tau_{i,t}}+\omega_{\varrho_{1,t}}+\ep  S_{\alpha,t})|_W$ is at most $m_t$-positive, where $m_t:=\dim {\rm Sh}_{H_{N,t}}(X_t)$.  In conclusion, we have $m_t\geq m$. This yields $\dim {\rm Sh}_{H_{N,t}}(X_t)\geq \dim {\rm Sh}_{H_{N,0}}(X_0)$ for any $t\in \bD_\delta$. Recall that ${\rm sh}_{H_{N,t}}$ is the Shafarevich morphism of $X_t$. This implies that
	 $
	\gamma d(X_t)\geq \gamma d(X_0)
	$ for any $t\in \bD_\delta$. \cref{mainC} is proved. 
	
	\medspace
	
	\noindent {\bf Proof of \cref{mainA}}:  Set \(N' := r_N + N\), where \(r_N\) is the rank of the \(\bR\)-VMHS \(\cM\) in \Cref{item:VMHS}. Let \(H_{N',t}\) be the intersection of the kernels of all linear representations \(\pi_1(X_t) \to \GL_{N'}(\bC)\). Define a representation \(\sigma': \pi_1(X) \to \GL_{N'}(\bC)\), which is the direct sum of \(\sigma\) and the trivial representation \(\pi_1(X) \to \GL_{r_N}(\bC)\). By our definition of \(H_{N'}\), we have \(H_{N'} \subset \ker \sigma' = \ker \sigma\).
	
	Since \(\sigma\) is big, it follows that \({\rm sh}_{H_{N'}}: X \to {\rm Sh}_{H_{N'}}(X)\) and \({\rm sh}_{H_N}: X \to {\rm Sh}_{H_N}(X)\) are both birational morphisms. The same argument shows that \(H_{N',t} \subset H_{N,t}\) for each \(t \in \bD\).
	
	Thus, for each \(t\), the following factorization map holds, due to the properties of the Shafarevich morphism and the universal property of Stein factorization:
	\[
	\begin{tikzcd}
		X_t \arrow[r, "{\rm sh}_{H_{N',t}}"] \arrow[dr, "{\rm sh}_{H_{N,t}}"'] & {\rm Sh}_{H_{N',t}}(X_t) \arrow[d] \\
		& {\rm Sh}_{H_{N,t}}(X_t)
	\end{tikzcd}
	\]
	Note that we might have \(\dim {\rm Sh}_{H_{N',t}}(X_t) < \dim {\rm Sh}_{H_{N,t}}(X_t)\) for some $t$.
	
	Let \(\varrho_{\cM}: \pi_1(X) \to \GL_{r_N}(\bC)\) be the monodromy representation of \(\cM\), which has a torsion-free image. Let \(\varrho_{\cM,t}: \pi_1(X_t) \to \GL_{r_N}(\bC)\) be the induced representation via the diffeomorphism \(F_t: X_t \to X\). Then we have \(H_{N',t} \subset \ker \varrho_{\cM,t}\). This implies that for each fiber \(Z\) of \({\rm sh}_{H_{N',t}}\), the image \(\varrho_{\cM,t}({\rm Im}[\pi_1(Z) \to \pi_1(X_t)])\) is finite, and thus trivial. 
	
	Since \(\alpha_t \in H^1(X_t, \cL_{0,t}^* \otimes \cL_{-1,t})\) is the extension class defining the local system \(\cM_t := F_t^* \cM\), corresponding to \(\varrho_{\cM,t}\) as an extension of \(\cL_{0,t}\) by \(\cL_{-1,t}\), we deduce that \(\alpha_t|_Z\) is trivial. By \eqref{eq:canonical2}, \(S_{\alpha_t}|_Z\) is also trivial for each smooth fiber \(Z\) of \({\rm sh}_{H_{N',t}}\). In conclusion, for each smooth fiber \(Z\) of \({\rm sh}_{H_{N',t}}\), 
\begin{align}\label{eq:trivial}
	 \left(\sum_{i=1}^\ell T_{\tau_{i,t}} + \omega_{\varrho_{1,t}} + \varepsilon S_{\alpha_t}\right)\big|_Z \text{ is trivial.}
\end{align} 
	
	Fix any \(t \in \bD_\delta\). We choose an open subset \(U_t\) of \({\rm Sh}_{H_{N',t}}(X_t)^{\rm reg}\) such that:
	\begin{enumerate}[label=(\arabic*)]
		\item \({\rm sh}_{H_{N',t}}^{-1}(U_t) \to U_t\) is a proper holomorphic submersion.
		\item The open subset \(W := \Omega_t \cap {\rm sh}_{H_{N',t}}^{-1}(U_t)\) is non-empty.
	\end{enumerate}
	Then  by \eqref{eq:trivial}, \(\left(\sum_{i=1}^\ell T_{\tau_{i,t}} + \omega_{\varrho_{1,t}} + \varepsilon S_{\alpha_t}\right)\big|_W\) is at most \(m'_t\)-positive, where \(m'_t := \dim {\rm Sh}_{H_{N',t}}(X_t)\). Recall that $$ (\sum_{i=1}^{\ell}T_{\tau_{i,t}}+\omega_{\varrho_{1,t}}+\ep  S_{\alpha,t})|_{\Omega_t}$$
	is at least $m$-positive, where $m:=  \dim {\rm Sh}_{H_N}(X)=\dim X$.   This implies that \(m'_t = m=\dim X_t\). Hence \({\rm sh}_{H_{N',t}}\) is a birational morphism for each  \(t \in \bD_\delta\). This proves that \(X_t\) has a big fundamental group for \(t \in \bD_\delta\). \Cref{mainA} is proved.
	
 \end{proof}

 % \bibliography{biblio2}
%  \bibliographystyle{ssmfalpha}

\end{document}